\documentclass[11pt,reqno]{amsart}

\usepackage[T1]{fontenc}
\usepackage[utf8]{inputenc}
\usepackage{lmodern}
\usepackage{comment}
\usepackage{microtype}
\usepackage{amsmath,amssymb,amsthm,mathtools,mathrsfs}
\usepackage{aliascnt}
\usepackage{enumitem}
\usepackage{booktabs}
\usepackage{array}
\usepackage{xcolor}
\usepackage[colorlinks=true,linkcolor=blue!55!black,citecolor=green!45!black,urlcolor=blue!60!black]{hyperref}
\usepackage[nameinlink,noabbrev]{cleveref}
\usepackage[left=2cm,right=2cm]{geometry}

\usepackage{longtable}

\usepackage{framed}
\usepackage{mdframed}
\newtheorem*{ftheo}{Main Theorem}
\newenvironment{fthm}
  {\begin{mdframed}[innertopmargin = 3pt, innerbottommargin=7pt,skipabove=5pt,skipbelow=5pt,linewidth=0.25pt,nobreak=true,align=center]\begin{ftheo}}
  {\end{ftheo}\end{mdframed}}

  \newtheorem*{ftheo2}{Main Conjecture}
\newenvironment{fthm2}
  {\begin{mdframed}[innertopmargin = 3pt, innerbottommargin=7pt,skipabove=5pt,skipbelow=5pt,linewidth=0.25pt,nobreak=true,align=center]\begin{ftheo2}}
  {\end{ftheo2}\end{mdframed}}

\usepackage{array,booktabs,tabularx}
\usepackage{graphicx}
\usepackage{microtype}
\usepackage{enumitem}
\usepackage{tikz}
\usepackage{tikz-cd}
\usetikzlibrary{arrows}
\usetikzlibrary{arrows.meta,positioning,calc}
\def \vertbar [#1](#2,#3,#4){
    \draw [#1] (#2,#3) -- (#2,#4);
    \draw [fill=white] (#2,#3) circle [radius=0.1];
    \draw [fill=black] (#2,#4) circle [radius=0.1];
}

\usepackage{pict2e} 

\newsavebox{\crossing}
\savebox{\crossing}(10,10)[bl]{
\thicklines 
\qbezier(5,5)(7,10)(10,10)
\qbezier(5,5)(3,0)(0,0)
\qbezier(5,5)(3,10)(0,10)
\qbezier(5,5)(7,0)(10,0)
}

\newsavebox{\closing}
\savebox{\closing}(10,10)[bl]{
\thicklines 
\qbezier(0,0)(5,0)(5,5)
\qbezier(0,10)(5,10)(5,5)
}

\newsavebox{\cclosing}
\savebox{\cclosing}(10,30)[bl]{
\thicklines 
\qbezier(0,0)(10,0)(10,15)
\qbezier(0,30)(10,30)(10,15)
}

\newsavebox{\opening}
\savebox{\opening}(10,10)[bl]{
\thicklines 
\qbezier(0,0)(-5,0)(-5,5)
\qbezier(0,10)(-5,10)(-5,5)
}

\newsavebox{\ssone}
\savebox{\ssone}(10,20)[bl]{
\thicklines 
\qbezier(5,5)(7,10)(10,10)
\qbezier(5,5)(3,0)(0,0)
\qbezier(5,5)(3,10)(0,10)
\qbezier(5,5)(7,0)(10,0)
\put(0,20){\line(1,0){10}} 
}

\newsavebox{\sstwo}
\savebox{\sstwo}(10,20)[bl]{
\thicklines 
\qbezier(5,15)(7,20)(10,20)
\qbezier(5,15)(3,10)(0,10)
\qbezier(5,15)(3,20)(0,20)
\qbezier(5,15)(7,10)(10,10)
\put(0,0){\line(1,0){10}} 
}

\newsavebox{\sssone}
\savebox{\sssone}(10,30)[bl]{
\thicklines 
\qbezier(5,5)(7,10)(10,10)
\qbezier(5,5)(3,0)(0,0)
\qbezier(5,5)(3,10)(0,10)
\qbezier(5,5)(7,0)(10,0)
\put(0,20){\line(1,0){10}} 
\put(0,30){\line(1,0){10}} 
}

\newsavebox{\ssstwo}
\savebox{\ssstwo}(10,30)[bl]{
\thicklines 
\qbezier(5,15)(7,20)(10,20)
\qbezier(5,15)(3,10)(0,10)
\qbezier(5,15)(3,20)(0,20)
\qbezier(5,15)(7,10)(10,10)
\put(0,0){\line(1,0){10}} 
\put(0,30){\line(1,0){10}} 
}

\newsavebox{\sssthree}
\savebox{\sssthree}(10,30)[bl]{
\thicklines 
\qbezier(5,25)(7,30)(10,30)
\qbezier(5,25)(3,20)(0,20)
\qbezier(5,25)(3,30)(0,30)
\qbezier(5,25)(7,20)(10,20)
\put(0,0){\line(1,0){10}} 
\put(0,10){\line(1,0){10}} 
}

\newsavebox{\sssonethree}
\savebox{\sssonethree}(10,30)[bl]{
\thicklines 
\qbezier(5,25)(7,30)(10,30)
\qbezier(5,25)(3,20)(0,20)
\qbezier(5,25)(3,30)(0,30)
\qbezier(5,25)(7,20)(10,20)
\qbezier(5,5)(7,10)(10,10)
\qbezier(5,5)(3,0)(0,0)
\qbezier(5,5)(3,10)(0,10)
\qbezier(5,5)(7,0)(10,0)
}

\newcommand{\LissajousFiveFourQuiverPicture}{%
\begin{tikzpicture}[
  x=1.5pt,
  y=1.5pt,
  >=latex,
  dividepiece/.style={
    inner sep=0pt,
    outer sep=0pt,
    anchor=center
  },
  qarrow/.style={
    draw=red,
    ->,
    line width=.55pt,
    shorten <=1.4pt,
    shorten >=1.4pt
  },
  qvertex/.style={
    circle,
    fill=green,
    draw=blue,
    inner sep=0pt,
    minimum size=2.4pt
  }
]

  \node[dividepiece] at (0,0)
    {\usebox{\opening}};

  \node[dividepiece] at (0,0)
    {\usebox{\sssonethree}};

  \node[dividepiece] at (10,0)
    {\usebox{\ssstwo}};

  \node[dividepiece] at (20,0)
    {\usebox{\sssonethree}};

  \node[dividepiece] at (30,0)
    {\usebox{\ssstwo}};

  \node[dividepiece] at (40,-10)
    {\usebox{\closing}};

  \node[dividepiece] at (40,10)
    {\usebox{\closing}};

  \coordinate (nUL) at (0,10);
  \coordinate (nLL) at (0,-10);
  \coordinate (nC)  at (10,0);
  \coordinate (nUR) at (20,10);
  \coordinate (nLR) at (20,-10);
  \coordinate (nR)  at (30,0);

  \coordinate (mL)  at (0,0);
  \coordinate (pU)  at (10,10);
  \coordinate (pD)  at (10,-10);
  \coordinate (mC)  at (20,0);
  \coordinate (pRU) at (30,10);
  \coordinate (pRD) at (30,-10);

  \draw[qarrow] (nUL) -- (pU);
  \draw[qarrow] (nLL) -- (pD);

  \draw[qarrow] (nC) -- (pU);
  \draw[qarrow] (nC) -- (pD);

  \draw[qarrow] (nUR) -- (pU);
  \draw[qarrow] (nUR) -- (pRU);

  \draw[qarrow] (nLR) -- (pD);
  \draw[qarrow] (nLR) -- (pRD);

  \draw[qarrow] (nR) -- (pRU);
  \draw[qarrow] (nR) -- (pRD);

  \draw[qarrow] (pU) -- (mL);
  \draw[qarrow] (pD) -- (mL);

  \draw[qarrow] (pU) -- (mC);
  \draw[qarrow] (pD) -- (mC);

  \draw[qarrow] (pRU) -- (mC);
  \draw[qarrow] (pRD) -- (mC);

  \draw[qarrow] (mL) -- (nUL);
  \draw[qarrow] (mL) -- (nLL);
  \draw[qarrow] (mL) -- (nC);

  \draw[qarrow] (mC) -- (nC);
  \draw[qarrow] (mC) -- (nUR);
  \draw[qarrow] (mC) -- (nLR);
  \draw[qarrow] (mC) -- (nR);

\foreach \v in {nUL,nLL,nC,nUR,nLR,nR,mL,pU,pD,mC,pRU,pRD}
  {\node[qvertex] at (\v) {};}

\end{tikzpicture}%
}

\numberwithin{equation}{section}

\newtheorem{theorem}{Theorem}[section]
\newaliascnt{proposition}{theorem}
\newtheorem{proposition}[proposition]{Proposition}
\aliascntresetthe{proposition}
\newaliascnt{lemma}{theorem}
\newtheorem{lemma}[lemma]{Lemma}
\aliascntresetthe{lemma}
\newaliascnt{corollary}{theorem}
\newtheorem{corollary}[corollary]{Corollary}
\aliascntresetthe{corollary}
\newaliascnt{conjecture}{theorem}

\aliascntresetthe{conjecture}
\theoremstyle{definition}
\newaliascnt{definition}{theorem}
\newtheorem{definition}[definition]{Definition}
\aliascntresetthe{definition}
\newaliascnt{example}{theorem}
\newtheorem{example}[example]{Example}
\aliascntresetthe{example}
\newaliascnt{remark}{theorem}
\newtheorem{remark}[remark]{Remark}
\aliascntresetthe{remark}

\newtheorem*{conjecture*}{Main Conjecture}
\newtheorem*{theorem*}{Main Theorem}

\crefname{theorem}{theorem}{theorems}
\Crefname{theorem}{Theorem}{Theorems}
\crefname{proposition}{proposition}{propositions}
\Crefname{proposition}{Proposition}{Propositions}
\crefname{lemma}{lemma}{lemmas}
\Crefname{lemma}{Lemma}{Lemmas}
\crefname{corollary}{corollary}{corollaries}
\Crefname{corollary}{Corollary}{Corollaries}
\crefname{conjecture}{conjecture}{conjectures}
\Crefname{conjecture}{Conjecture}{Conjectures}
\crefname{definition}{definition}{definitions}
\Crefname{definition}{Definition}{Definitions}
\crefname{example}{example}{examples}
\Crefname{example}{Example}{Examples}
\crefname{remark}{remark}{remarks}
\Crefname{remark}{Remark}{Remarks}

\newcommand{\Tors}{\operatorname{Tors}}
\newcommand{\lk}{\operatorname{lk}}

\DeclareMathOperator{\GL}{GL}
\DeclareMathOperator{\coker}{coker}

\DeclareMathOperator{\Aut}{Aut}
\DeclareMathOperator{\Hom}{Hom}
\DeclareMathOperator{\Ext}{Ext}

\DeclareMathOperator{\im}{im}
\DeclareMathOperator{\id}{id}
\DeclareMathOperator{\diag}{diag}

\DeclareMathOperator{\per}{per}
\newcommand{\kk}{\Bbbk}
\newcommand{\bG}{\mathbb G}
\newcommand{\C}{\mathbb C}
\newcommand{\N}{\mathbb N}
\newcommand{\Z}{\mathbb Z}
\newcommand{\R}{\mathbb R}
\newcommand{\m}{\mathfrak m}

\DeclareMathOperator{\Gclass}{Dil}
\newcommand{\Pclass}{\mathscr{P}}

\newcommand{\Eul}{\mathcal E}
\newcommand{\Poin}{\mathcal P}

\newcommand{\cyc}{\mathrm{cyc}}
\newcommand{\gr}{\mathrm{gr}}

\newcommand{\angles}[1]{\langle #1\rangle}

\newcommand{\set}[1]{\left\{#1\right\}}
\newcommand{\suchthat}{\,\middle|\,}

\newcommand{\lr}{\longrightarrow}
\newcommand{\sse}{\subseteq}

\newcommand{\cE}{\mathcal E}
\newcommand{\FNDelta}{\Delta^{\mathrm{FN}}}

\newcommand{\Gm}{\mathbb{G}_{m}}

\newcommand{\HOMFLY}{\operatorname{P}}

\renewcommand{\gr}{\mathrm{gr}}
\newcommand{\fd}{\mathrm{fd}}
\newcommand{\cont}{\mathrm{cont}}

\newcommand{\D}{\mathcal D}

\renewcommand{\Eul}{\mathcal E}
\renewcommand{\Poin}{\mathcal P}
\DeclareMathOperator{\RHom}{RHom}

\DeclareMathOperator{\thick}{thick}

\tikzset{
  qvertex/.style={circle,draw,inner sep=1.4pt,minimum size=5.7mm,font=\small},
  qarrow/.style={-{Stealth[length=2.2mm,width=1.5mm]},line width=.6pt},
  rootvertex/.style={qvertex,double,double distance=.8pt}
}

\usepackage{etoolbox}

\makeatletter
\patchcmd{\@tocline}
  {\hfil}
  {\dotfill}
  {}{}

\def\l@subsection{\@tocline{2}{0pt}{3.5pc}{4pc}{}}
\makeatother

\usepackage{hyperref}
\hypersetup{
    colorlinks=true,   
    linkcolor=blue,    
    citecolor=blue,    
    urlcolor=blue      
}

\title[Monodromy of plane curve singularities and quiver mutation]
{Monodromy of plane curve singularities and quiver mutation}
\author{Roger Casals}

\begin{document}

\begin{abstract}
The main result of this article shows that the quiver mutation class of a malleable real Morsification uniquely determines the integral monodromy module of a plane curve singularity. In particular, we show that the quiver mutation class of a malleable divide determines the complex topological type of an irreducible plane curve singularity. This establishes the algebraic-to-topological implication of a conjecture of S.~Fomin, P.~Pylyavskyy, E.~Shustin and D.~Thurston in the malleable irreducible case. The result is proven by developing representation-theoretic techniques based on an equivariant Euler pairing in the derived category of continuous finite-dimensional dg modules over a differential bigraded Ginzburg algebra. A key step uses these techniques to show that the quiver mutation class of a plabic fence uniquely recovers the torsion part of the Alexander module of the associated smooth link.
\end{abstract}

\maketitle
\setcounter{tocdepth}{2}
\tableofcontents

\section{Introduction}\label{sec:introduction}

In their work on real Morsifications and quiver mutations, S.~Fomin, P.~Pylyavskyy, E.~Shustin and D.~Thurston put forth the following conjecture:\\
\vspace{-0.2cm}
\begin{fthm2}[\cite{FPST}]\label{conj:FPST}
Given two real morsifications of real isolated plane curve singularities, the following are equivalent:
\begin{itemize}
\item[$(i)$] the two singularities have the same complex topological type;
\item[$(ii)$] the quivers associated with the two morsifications are mutation equivalent.
\end{itemize}
\end{fthm2}
\vspace{0.2cm}
\noindent The precise reference is \cite[Conjecture 5.5]{FPST}. In my view, the Main Conjecture is remarkable, as it asserts that isolated plane curve singularities are
topologically classified by the mutation classes of their associated quivers. The former are central objects of study in algebraic geometry and the theory of singularities. The study of plane curves and their singularities has had a prominent role in mathematics for centuries, see e.g.~\cite{brieskorn1986, casas2000, ghys2017, milnor1968, wall2004, zariski2006}. The latter, the study of quivers and their mutation, is a much more recent enterprise, particularly bolstered by the series of articles \cite{Berenstein2005,Fomin2001,Fomin2003} and the subsequent study of cluster algebras. In a sense, $(i)$ and $(ii)$ in the Main Conjecture belong to apparently different areas of mathematics: topology and algebraic geometry, for $(i)$, and combinatorics and graph theory, for $(ii)$.\\

Heretofore, both directions of the equivalence in the Main Conjecture are open. The case of simple singularities, equivalently quivers of finite type, is established in \cite[Theorem 18.3]{FPST}. See also \cite[Conjecture 7.11 \& Corollary 9.10]{FPST}. The object of this article is to establish the equivalence $(ii)\Longrightarrow(i)$ of the Main Conjecture in a central case. Namely, when the quivers are associated to irreducible malleable divides. It is a central case because of the likelihood of \cite[Conjecture 14.10]{FPST}, which asserts that all algebraic divides are malleable. See \Cref{fig:divides-lissajous} for an instance of a divide $D$ and its quiver $Q(D)$. More generally, our main result establishes that the quiver mutation class of any malleable divide, irreducible or not, uniquely recovers the integral monodromy module of the singularity.\\

The precise main result of the article reads as follows:
\vspace{0.1cm}
\begin{fthm}\label{thm:main} Let \((C,0)\) be an isolated plane curve singularity, $D$ a divide arising from a real Morsification of $(C,0)$ and $Q(D)$ its quiver. Suppose that $D$ is malleable. Then the quiver mutation class $[Q(D)]$ uniquely determines the integral monodromy module of $(C,0)$.
\end{fthm}
\vspace{0.2cm}
In particular, the Main Theorem proves the implication $(ii)\Longrightarrow(i)$ of the ``Main Conjecture'' \cite[Conjecture 5.5]{FPST} for irreducible plane curve singularities with malleable divides. Indeed, the integral monodromy module of $(C,0)$ determines the Alexander polynomial $\Delta_C(t)$, which itself determines the topological type of an irreducible plane curve singularity. There are also infinitely many instances of reducible plane curve singularities which are determined by the integral monodromy module, or even its one-variable Alexander polynomial, cf.~\Cref{sec:illustrative_examples}, and thus the Main Theorem can be applied more generally to resolve $(ii)\Longrightarrow(i)$ of the ``Main Conjecture'' in such cases.\\

Note that the Main Theorem can also be phrased as follows. Let \((C_1,0)\) and \((C_2,0)\) be two isolated plane curve singularities, with $D_1,D_2$ two divides arising from two real Morsifications of these two singularities. Suppose that $D_1$ and $D_2$ are malleable and that the integral monodromy module of either \((C_1,0)\) or \((C_2,0)\) determines its complex topological type. Then\\
\begin{center}
\fbox{$Q(D_1)\sim_{\mathrm{mut}}Q(D_2)\Longrightarrow (C_1,0) \mbox{ and }(C_2,0) \mbox{ have the same complex topological type.}$}
\end{center}
\vspace{0.2cm}

\begin{figure}[htbp]
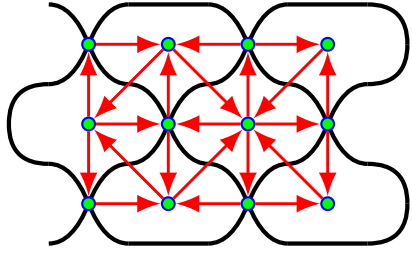

\centering

\scalebox{2}{\LissajousFiveFourQuiverPicture}

\caption{An irreducible malleable divide $D$ for the singularity $(C,0)=\{x^4=y^5\}$, in black, together with its
associated quiver $Q(D)$, with blue vertices and red arrows. The quiver $Q(D)$ is the quiver associated to the Grassmannian $\mbox{Gr}(4,9)$, cf.~\cite[Remark 16.2]{FPST}. The link of this irreducible singularity is the $(4,5)$-torus knot.}
\label{fig:divides-lissajous}
\end{figure}

A core challenge of the implication $(ii)\Longrightarrow (i)$ of the Main Conjecture, which the Main Theorem overcomes, is as follows. Given a divide $D$ of a real Morsification of an isolated plane curve singularity $(C,0)$, its quiver $Q(D)$ has a direct connection to the geometry of $(C,0)$. To wit, the vertices of $Q(D)$ can be identified with the vanishing cycles associated to the Morsification, and the arrows of $Q(D)$ exactly capture the intersection numbers between such vanishing cycles. In fact, \cite[Theorem 4.4]{FPST} shows that $Q(D)$ determines the topological type of the singularity. The challenge is that another quiver $\mu(Q)\in[Q(D)]$ in the quiver mutation class of $Q(D)$ will, a priori, have little to do with the geometry of the singularity $(C,0)$. Among other issues: such quiver $\mu(Q)$ will almost never be associated to a divide, nor there will be a compatible $\{\oplus,\ominus,\bullet\}$-coloring, nor the arrows of $\mu(Q)$ will have any direct meaning in relation to $(C,0)$. Plainly, suppose we asked a computer to mutate the quiver $Q(D)$, say 8937284651872 times, and it would return the resulting quiver $\mu(Q)$ to us, but without specifying the mutation sequence it used. How would we recover from such output $\mu(Q)$ any information such as the characteristic polynomial of the monodromy of the singularity $(C,0)$?\\

In order to prove the Main Theorem, we develop the study of a 3-Calabi-Yau differential bigraded Ginzburg algebra $\Gamma(Q,W,d)$ associated to a graded quiver with potential $(Q,W,d)$. Such a dg algebra incorporates an Adams grading associated to an arrow-grading $d:Q_1\lr\Z$ for which the potential $W$ is homogeneous of degree one. The two key new ideas to prove the Main Theorem are:
\begin{itemize}
\item[$(i)$] First, for quivers associated to plabic fences, we show that the cokernel of a certain equivariant Euler pairing in the derived category of continuous finite-dimensional dg modules over $\Gamma(Q,W,d)$, weighted with respect to the internal $d$-grading, is an invariant of the quiver mutation class $[Q]$. Here we choose a non-degenerate potential $W$ and a corresponding arrow-grading $d$, which we show are uniquely associated to such quivers. The content of Sections \ref{sec:graded-QP}, \ref{sec:euler} and \ref{sec:dilation_plabic_fence} is aimed at developing and establishing these results.\\

\item[$(ii)$] Second, we show that such cokernel is isomorphic, as a $\Z[t,t^{-1}]$-module, to the torsion part of the Alexander module of the smooth link associated to the plabic fence. This is achieved in \Cref{sec:polynomial_plabic_fence}, by establishing a combinatorial expression for the equivariant Euler pairing and relating it to a Seifert matrix. Such results, along with $(i)$, are then specialized to algebraic links in \Cref{sec:proof_main} to finally establish the Main Result.
\end{itemize}

The results needed to implement $(i)$ are representation-theoretic in nature, building on the study of quivers with potential, as initiated in \cite{DWZ} by H.~Derksen, J.~Weyman, and A.~Zelevinsky, and the derived categorical enhancements of B.~Keller and D.~Yang in \cite{KellerYang}, and C.~Amiot and S.~Oppermann in \cite{AO}. Part of the new results establishing $(i)$ are developed in Sections \ref{sec:graded-QP}, \ref{sec:euler} and \ref{sec:dilation_plabic_fence} and might be of independent interest. It should also be noted that differential bigraded Ginzburg algebras and equivariant Euler pairings have been explored by L.~Fan, B.~Keller and Y.~Qiu in \cite{FanKellerQiu2024}, by R.~Fujita and K.~Murakami in \cite{FujitaMurakami2022}, and see also \cite{Keller2020Cartan} and recent work of Y.~Wu and L.~Fan on Higgs categories for differential bigraded Ginzburg algebras \cite{Wu2026Graded}.\\

From a combinatorial perspective, we obtain a perhaps surprising invariance result from our representation-theoretic results. Namely, that for any quiver $Q\in[Q(\bG)]$ in the mutation class of a quiver $Q(\bG)$ of a plabic fence $\bG$, there exists a canonical arrow grading $d:Q_1\lr\Z$, unique up to certain equivalence, such that the cokernel of the $|Q_0|\times|Q_0|$ matrix $\Eul(t)$ with entries
\begin{equation}\label{eq:euler-formula-intro}
 \boxed{
 \Eul(t)_{ij}
 :=(1-t)\delta_{ij}
 -\sum_{a:j\to i}t^{d(a)}
 +\sum_{a:i\to j}t^{1-d(a)}}
\end{equation}
is the torsion part of the Alexander module of the smooth link associated to $\bG$. In particular, its determinant is the one-variable Alexander polynomial of such a link. The sums in \eqref{eq:euler-formula-intro} run over the arrows $a\in Q_1$ of the quiver $Q$, and $a:i\to j$ denotes an arrow with source vertex $i\in Q_0$ and target vertex $j\in Q_0$. See~\Cref{ssec:graded_Euler_pairing} and \Cref{thm:unique-fence-G1}. Thus, in the case of a malleable divide $D$ of a singularity $(C,0)$, even if many aspects of a quiver $Q$ in $[Q(D)]$ might not have a direct interpretation in terms of $(C,0)$, the Main Theorem shows that it still remembers the action of the algebraic monodromy of $(C,0)$. A combinatorial proof of such fact remains unknown to the author, as I obtained the expression \eqref{eq:euler-formula-intro} by using the weight decomposition of the Ext groups of the simple modules of a differential bigraded Ginzburg algebra, and its invariance is deduced from rather categorical considerations.\\

\noindent The article concludes with \Cref{sec:illustrative_examples} and \Cref{sec:comments_and_examples}, where several detailed examples and related applications are provided. In particular:

\begin{enumerate}
    \item We provide a counterexample to \cite[Conjecture 6.17]{FPST}, exhibiting two plabic graphs which are not move-and-switch equivalent but have mutation-equivalent quivers. These are distinguished by using the top $a$-degree part of their HOMFLY polynomials.\\

    \item We show that the implication $(ii)\Longrightarrow(i)$ in the Main Conjecture holds for all plane curve singularities with Milnor number $\mu\leq16$. In fact, a table of all such singularities with their associated Alexander polynomials is provided, to illustrate how the Main Result and known invariants suffice to classify all of these 74 singularities via their quiver mutation classes.\\
\end{enumerate}

\noindent{\bf Acknowledgements}. I am very thankful to Sergey Fomin and Bernhard Keller for valuable comments that helped improve the original manuscript. I am grateful to the authors of \cite{FPST} for producing such an enticing conjecture. Finally, I thank PCMI-IAS for its hospitality during the research program ``Knotted Surfaces in Four-Manifolds'', where parts of this article were developed. R.C.~is supported by the National Science Foundation under the grant DMS-2505760, and a UC Davis College of L\&S Dean's Fellowship.\hfill$\Box$

\vspace{0.3cm}
\noindent{\bf Conventions}.
We work over \(\kk=\C\).  Quivers are finite and have no loops unless stated
otherwise.  Path multiplication is right-to-left: if \(a:i\to j\) and
\(b:j\to k\), then \(ba\) means first \(a\), then \(b\). Throughout the manuscript $\Lambda:=\Z[t,t^{-1}]$ will denote the Laurent polynomial ring in the $t$-variable, whose units are $\pm t^{m}$ for any $m\in\Z$. \hfill$\Box$


\section{Graded quivers with potential and dilation classes}\label{sec:graded-QP}

The main goals of this section are to develop the notion of dilation classes, cf.~\Cref{ssec:dilation_classes}, and establish that graded mutations of quivers with potential provide a bijection between dilation classes, cf.~\Cref{thm:G1-mutation}.

\subsection{Completed path algebras and QPs}\label{ssec:prelim_pathalgebras_and_QPs}

Let us summarize the ingredients from \cite{DWZ} that we use in this section, referring to \cite{DWZ} for further details. Let $Q$ be a quiver with vertex set $Q_0$ and arrow set $Q_1$. Consider
\[
  R_Q:=\bigoplus_{i\in Q_0}\kk e_i,
  \qquad
  A_Q:=\bigoplus_{a\in Q_1}\kk a,
\]
where $A_Q$ is an $R_Q$-bimodule by letting the idempotents $e_i\in R_Q$ act via concatenation, where $e_i$ is considered as a trivial constant path at the vertex $i\in Q_0$.  The completed path algebra $\widehat{\kk Q}$ of the quiver $Q$ and its cyclic quotient $\widehat{\kk Q}_{\cyc}$ are
\[
  \widehat{\kk Q}:=\widehat T_{R_Q}(A_Q)
  =\prod_{n\ge 0}A_Q^{\otimes_{R_Q}n},\qquad
  \widehat{\kk Q}_{\cyc}
  :=\widehat{\kk Q}/\overline{[\widehat{\kk Q},\widehat{\kk Q}]}.
\]
They respectively contain the ideals
\[
  \m_Q:=\prod_{n\ge1}A_Q^{\otimes_{R_Q}n}\sse \widehat{\kk Q},\qquad \m^2_{\cyc}\sse \widehat{\kk Q}_{\cyc}.
\]

\noindent The algebra morphisms that we will consider are all continuous for the \(\m_Q\)-adic topology. The following is \cite[Definition 4.1]{DWZ}:

\begin{definition}[Quiver with potential]
A potential $W$ for a quiver $Q$ is an element $W\in\m^2_{\cyc}$, and a quiver with potential (QP) is a pair $(Q,W)$.
\hfill$\Box$\end{definition}

\noindent The foundational aspects of the study of quivers with potential are developed in \cite{DWZ}. Quivers with potential will be considered up to right-equivalence, cf.~\cite[Definition 4.2]{DWZ}:

\begin{definition}[Right-equivalence]\label{def:right_equivalent}
Let $(Q,W)$ and $(Q',W')$ be two QPs on the same vertex set. A right-equivalence between $(Q,W)$ and $(Q',W')$ is a $\kk$-algebra isomorphism $\varphi:\widehat{\kk Q}\lr\widehat{\kk Q'}$ such that $\varphi|_{R_Q}=\mbox{id}$ and $\varphi(W)$ is cyclically equivalent to $W'$.\hfill$\Box$
\end{definition}

\noindent In Definition \ref{def:right_equivalent}, two potentials $W,W'$ are said to be cyclically equivalent if their difference $W-W'$ lies in the closure of the span of all elements of the form $a_1\cdots a_d-a_2\cdots a_da_1$, where $a_1,\ldots,a_d$ is a cyclic path, cf.~\cite[Definition 3.2]{DWZ}. We often denote a right-equivalence $\varphi$ from $(Q,W)$ to $(Q',W')$ by $\varphi:(Q,W)\lr(Q',W')$.\\

In \cite{DWZ} the notions of {\it reduced} and {\it trivial} QPs are also introduced, as follows. A QP $(Q,W)$ is said to be reduced if the degree-2 homogeneous part $W^{(2)}$ vanishes, i.e.~$W^{(2)}=0$. That is, $(Q,W)$ is reduced if $W$ contains no quadratic monomial terms, i.e.~no terms of the form $ab$, $a,b\in Q_1$. A QP $(Q,W)$ is said to be trivial if $W$ is entirely quadratic, i.e.~ $W\in \widehat{\kk Q}^{(2)}$ belongs to the degree-2 homogeneous part of the path algebra $\widehat{\kk Q}$, and the Jacobian algebra of $(Q,W)$ is isomorphic to the semisimple vertex algebra $R_Q$.

\begin{example}\label{ex:trivialQP} Let $Q$ be a quiver with two vertices $i,j\in Q_0$, $n\in\N$, and choose $2n$ arrows $x_\nu:i\to j$, $y_\nu:j\to i$, for $\nu\in[1,n]$. If we choose the potential $W=x_1y_1+\ldots+x_ny_n$, then $(Q,W)$ is a trivial QP.\hfill$\Box$
\end{example}

\noindent The Splitting Theorem \cite[Theorem~4.6]{DWZ} decomposes a QP, uniquely up to right-equivalence, as a direct sum of a reduced QP and a trivial quadratic QP, the latter as in \Cref{ex:trivialQP}.

\begin{definition}[Stable QP class]\label{def:stable-class}
The stable right-equivalence class $[Q,W]$ of a QP \((Q,W)\) is the set of QPs which can be obtained from $(Q,W)$ via a sequence intertwining right-equivalences, and adjoining or deleting trivial quadratic QP summands.
\hfill$\Box$\end{definition}

\noindent Note that QPs representing the same stable right-equivalence class have the same vertex set, up to relabeling. The notion of mutation of a quiver with potential consists of a pre-mutation followed by reduction to the reduced QP part, cf.~\cite[Section 5]{DWZ}. It is an involutive operation up to right-equivalence by \cite[Theorem~5.7]{DWZ}. We will also use the notion of a non-degenerate QP $(Q,W)$: a reduced QP is said to be non-degenerate if every finite sequence of QP mutations ends at a QP with a $2$-acyclic reduced underlying quiver, cf.~\cite[Section 7]{DWZ}.

\subsection{$\Z$-graded QPs and vertex gauge}

We use the notion of $\Z$-graded QPs and their mutations, as introduced in \cite[Section 6]{AO}; see also \cite{Keller2026_GradedPreprint} and \cite{HerschendIyama2010}. We must add a new notion beyond those in \cite{AO}, that of a vertex gauge. This subsection presents the necessary ingredients on $\Z$-graded QPs and vertex gauge. The following definition appears in \cite[Section 6.3]{AO}:

\begin{definition}[Degree-one $\Z$-grading]\label{def:degree-one-grading}
Let $(Q,W)$ be a QP. By definition, a $\Z$-grading of $(Q,W)$ is a function
\[
 d:Q_1\longrightarrow\Z.
\]
By extending additively to paths, $W$ is said to be of degree one with respect to $d$ if every homogeneous component of \(W\) has total internal degree
one. Equivalently, \(W\) is homogeneous of degree one in the completed
cyclic quotient.
\hfill$\Box$\end{definition}
\noindent In this manuscript we only consider $\Z$-graded QPs of degree one: to ease notation, a graded QP $(Q,W,d)$ is henceforth understood to be a QP $(Q,W)$ endowed with a degree-one $\Z$-grading as in \Cref{def:degree-one-grading}. Note that a $\Z$-grading $d:Q_1\lr\Z$ for which $W$ is homogeneous of degree one and $d$ only takes value in $\{0,1\}\sse\Z$ relates to the notion of a cut, cf.~\cite[Definition 3.1]{HerschendIyama2010}.\\

Note that \Cref{def:degree-one-grading} grades the arrows of $Q$, rather than its vertices. Grading vertices yields an equivalence relation on arrow $\Z$-gradings, as follows:

\begin{definition}[Vertex gauge]\label{def:vertex-gauge}
Let $(Q,W,d)$ be a graded QP. Given a vertex $\Z$-grading \(h:Q_0\to\Z\) of $Q$, consider the arrow $\Z$-grading
\begin{equation}\label{eq:gauge}
 d^h(a):=d(a)+h(t(a))-h(s(a)).
\end{equation}
By definition, the two graded QPs $(Q,W,d)$ and $(Q,W,d^h)$ are said to be vertex-gauge equivalent to each other, i.e.~two graded QPs are vertex-gauge equivalent if one is obtained from the
other by changing the $\Z$-grading via \eqref{eq:gauge}.
\hfill$\Box$\end{definition}

\noindent In \Cref{def:vertex-gauge}, $s(a),t(a)\in Q_0$ denote the source and target vertices of an arrow $a\in Q_1$. The notion of a vertex gauge is well-defined since vertex-gauging preserves the degree-one property of $W$. Indeed, suppose that $W$ is of degree one with respect to $d$. Then for a cyclic monomial \(a_m\cdots a_1\) in $W$, the difference between the $d$ and $d^h$ degrees of the monomial $a_m\cdots a_1$ is
\[
 \sum_{r=1}^m\bigl(h(t(a_r))-h(s(a_r))\bigr)=0,
\]
since the sum telescopes around the cycle.  Hence every cyclic monomial has the same
degree for \(d\) and \(d^h\) and \Cref{def:vertex-gauge} is well-defined.\\

\subsection{Graded right-equivalence and homogeneous stabilization}

In the same manner that QPs are considered up to right-equivalence and stabilizations (i.e.~adjoining or deleting a trivial quadratic QP summand), and there is a Splitting Theorem, these notions also exist for graded QPs, as explained in \cite[Section 6.3]{AO}. For completeness, we recall them, see also \cite{Keller2026_GradedPreprint}. First, graded right-equivalence is \cite[Definition~6.5]{AO}:

\begin{definition}[Graded right-equivalence]\label{def:graded-right-equivalence}
Let $(Q,W,d),(Q',W',d')$ be two graded QPs. A graded right-equivalence $\varphi:\widehat{\kk Q}\lr\widehat{\kk Q'}$ is a right-equivalence from $(Q,W)$ to $(Q',W')$ such that $(\widehat{\kk Q},d)\lr(\widehat{\kk Q'},d')$ is an isomorphism of graded algebras.
\hfill$\Box$\end{definition}

\begin{definition}[Homogeneous stabilization]\label{def:hom-trivial}
By definition, a graded QP $(Q,W,d)$ is said to be trivial if the potential $W$ is entirely quadratic, homogeneous of $d$-degree one, and its Jacobian algebra is isomorphic to the vertex algebra $R_Q$. By definition, two graded QPs are said to be related by a homogeneous stabilization or reduction if one is obtained from the other by adjoining or deleting a trivial graded QP.
\hfill$\Box$\end{definition}

The graded version of the Splitting Theorem \cite[Theorem~4.6]{DWZ} is \cite[Theorem~6.6]{AO}: it states that a graded QP splits into reduced and
trivial graded parts, each unique up to graded right-equivalence. See also \cite{Keller2026_GradedPreprint} for more details.

\subsection{Dilation classes for QPs}\label{ssec:dilation_classes}

Let us now introduce a set that encodes the possible gradings of a QP $(Q,W)$ and incorporates the necessary equivalence relations to be able to work invariantly:

\begin{definition}[Dilations]\label{def:G1}
Let $[Q,W]$ be a stable right-equivalence class of a QP $(Q,W)$. By definition, the set of dilations $\Gclass[Q,W]$ associated to $[Q,W]$ is the set of all equivalence classes of graded QPs \((Q',W',d')\) such that $[Q',W']=[Q,W]$, where the equivalence relation is generated by the following:
\begin{enumerate}[label=(\alph*)]
\item vertex gauge, as in \Cref{def:vertex-gauge},
\item graded right-equivalence, as in \Cref{def:graded-right-equivalence},
\item homogeneous trivial stabilization and reduction, as in \Cref{def:hom-trivial}.
\end{enumerate}
By definition, an element \(\delta\in\Gclass[Q,W]\) is said to be a dilation or a dilation class.
\hfill$\Box$\end{definition}

\noindent By \Cref{def:G1}, the quiver $Q'$ of any representative of a dilation class $\delta=[(Q',W',d')]\in\Gclass[Q,W]$ has a vertex set $Q'_0$ identical to the vertex set $Q_0$ of Q, as neither right-equivalence nor stabilization of QPs, and none of the three equivalence relations (a), (b), (c) in \Cref{def:G1}, affect the vertex set of $Q$. Similarly, the arrow set $Q'_1$ of any such quiver differs from the arrow set $Q_1$ of $Q$ solely in adding or subtracting canceling pairs of arrows.

\begin{remark} If one considers $Q$ up to relabeling of the vertices, one can incorporate simultaneous vertex relabeling as a fourth equivalence relation to be imposed in \Cref{def:G1}. The main results of the article are independent of vertex relabeling.\hfill$\Box$
\end{remark}

\noindent Note that a right-equivalence of (ungraded) QPs from $(Q,W)$ to $(Q',W')$ might not be compatible with some given gradings $(Q,W,d)$ to $(Q',W',d')$ on those QPs. That is, the dilations associated to $(Q,W,d)$ and $(Q',W',d')$, where $d,d'$ are chosen independently of each other, are typically different in $\Gclass[Q,W]$ even if $(Q,W)$ and $(Q',W')$ are right-equivalent (or even identical). That said, given a right-equivalence $\varphi$ from $(Q,W)$ to $(Q',W')$ and a grading $(Q,W,d)$, it is possible to transport $d$ via $\varphi$ to a grading of $(Q',W')$, as follows.

\begin{lemma}[Homogeneous rectification]\label{lem:rectification}
Let \((Q,W,d)\) be a graded QP, $(Q',W')$ an ungraded QP, and $\varphi:(Q,W)\lr(Q',W')$ an $($ungraded$)$ right-equivalence. Then the following three objects exist:
\begin{enumerate}
\item a QP $(Q',V)$ right-equivalent to $(Q',W')$,
\item a graded QP $(Q',V,d')$,
\item a graded right-equivalence
\[
 \psi:(Q,W,d)\longrightarrow(Q',V,d').
\]
\end{enumerate}
In addition, the dilation \([(Q',V,d')]\in\Gclass[Q',W']\) is
independent of all choices.
\end{lemma}

\begin{proof}
First, note that the arrow grading \(d\) gives a complete internal \(\Z\)-grading on
\(\widehat{\kk Q}\): each path is homogeneous, with degree the sum of its arrow
degrees, and an arbitrary formal series is resolved coefficient-wise into its
homogeneous parts. We transport this grading through the (ungraded) right-equivalence \(\varphi\) to
\(\widehat{\kk Q'}\) by declaring
\[
 \widehat{\kk Q'}_r:=\varphi(\widehat{\kk Q}_r),
\]
where the subindex $r$ indicates the piece with internal degree $r$. Because \(\varphi\) fixes the vertex algebra and preserves the arrow ideal $\mathfrak{m}_Q$,
\(R_{Q'}\) is concentrated in degree zero, \(\m_{Q'}\) is homogeneous, and
\(\m_{Q'}/\m_{Q'}^2\) is a finite-dimensional graded
\(R_{Q'}\)-bimodule. Now choose a homogeneous \(R_{Q'}\)-bimodule section
\[
 s:\m_{Q'}/\m_{Q'}^2\longrightarrow\m_{Q'}.
\]
Such a section exists: e.g.~choose a linear section separately in every
internal degree and every \(e_jA_{Q'}e_i\).  Choose a homogeneous basis
of \(\m_{Q'}/\m_{Q'}^2\) and identify it with the displayed arrow space
\(A_{Q'}\).  Let \(d'\) be the resulting arrow grading.  By the universal
property of the completed tensor algebra, the assignment sending each
displayed arrow to its lift under the section \(s\) extends to a continuous endomorphism
\[
 \theta:\widehat{\kk Q'}\longrightarrow \widehat{\kk Q'},
\]
which fixes the idempotent elements. Its linear part on \(\m_{Q'}/\m_{Q'}^2\) is invertible, so the formal inverse
lemma applied to these complete path algebras shows that \(\theta\) is an automorphism. We then define
\[
 V:=\theta^{-1}(W'),
 \qquad
 \psi:=\theta^{-1}\circ\varphi.
\]
By construction, the map \(\psi\) is graded from \(d\) to \(d'\), and
\(\psi(W)=V\) in the cyclic quotient.  Hence \(V\) is homogeneous of degree
one and \(\psi\) is a graded right-equivalence. This establishes the existence of the objects $(1),(2),(3)$.\\

For independence of choices, let \(\theta_1,\theta_2\) be two rectifications, as above, producing
\((V_1,d_1')\) and \((V_2,d_2')\).  Since both identify their corresponding graded
algebras with the same transported graded algebra \(\widehat{\kk Q'}\) we conclude that
\[
 \theta_2^{-1}\theta_1:(Q',V_1,d_1')\longrightarrow(Q',V_2,d_2')
\]
is a graded right-equivalence, sending $V_1$ to $V_2$ in the corresponding cyclic quotients.
\end{proof}
\color{black}
Since homogeneous stabilizations provide a bijection between dilation classes, it then follows from the homogeneous rectification provided by \Cref{lem:rectification} that we have a canonical bijection between the sets of dilations associated to two QPs that are right-equivalent:

\begin{corollary}\label{cor:G1-presentation}
If \((Q,W)\) and \((Q',W')\) are stably right-equivalent, then homogeneous
rectification induces a canonical bijection
\[
 \Gclass[Q,W]
 \xrightarrow{\sim}
 \Gclass[Q',W']
\]
between their sets of dilations.\hfill$\Box$
\end{corollary}

\subsection{Graded QP mutation and vertex gauge}\label{sec:graded-mutation}

The QP mutation of \cite[Section 5]{DWZ} is generalized to a graded QP mutation in \cite[Section 6.3]{AO}, and see also \cite{Keller2026_GradedPreprint}. After briefly recalling such notion, we show in \Cref{thm:G1-mutation} that graded QP mutation induces a canonical bijection between the dilations of a QP $(Q,W)$ and its QP-mutation $\mu_k(Q,W)$ at a vertex $k$. Let \((Q,W,d)\) be a graded QP and \(k\in Q_0\) be a vertex incident
to no loop or oriented two-cycle. Let us also choose a cyclic representative of $W$ so that no monomial begins or ends at \(k\).\\

The graded QP mutation of \cite[Section 6.3]{AO} consists of two steps, as does QP-mutation: a pre-mutation and a reduction, the latter using the corresponding Splitting Theorem. The pre-mutation step is as follows. Given two arrows
\[
 a:i\longrightarrow k,
 \qquad
 b:k\longrightarrow j,\qquad i,j,k\in Q_0,\quad a,b\in Q_1,
\]
the QP pre-mutation step creates a composite \([ba]:i\to j\), reverses \(a,b\), and
adds the corresponding cubic term to the potential $W$, resulting in a QP $(\tilde{\mu}_k^L Q,\tilde{\mu}_k^L W)$. By definition, the {\it left graded} pre-mutation assigns gradings
\begin{align}
 \widetilde d^L([ba])&=d(b)+d(a),\label{eq:mut-composite}\\
 \widetilde d^L(a^*)&=1-d(a),\label{eq:mut-incoming}\\
 \widetilde d^L(b^*)&=-d(b),\label{eq:mut-outgoing}
\end{align}
while unchanged arrows retain their degree. The pre-mutated potential $\tilde{\mu}_k^LW$ is
homogeneous of degree one with respect to this pre-mutated grading $\tilde{\mu}_k^L(d):=\widetilde{d}^L$. This thus defines a new graded QP
$$\tilde{\mu}_k^L(Q,W,d):=(\tilde{\mu}_k^L Q,\tilde{\mu}_k^L W,\tilde{\mu}_k^L d).$$
Then \cite[Definition 6.8]{AO}, which uses \cite[Thm.~6.6 \& Prop.~6.7]{AO}, reads as follows:

\begin{definition}[Left graded QP-mutation]
Let $(Q,W,d)$ be a graded QP and $k\in Q_0$ as above. By definition, its left graded QP-mutation
$$\mu_k^L(Q,W,d)=(\mu_k^LQ,\mu_k^LW,\mu_k^Ld)$$
is the graded QP given by the reduced graded part of the pre-mutated graded QP $\tilde{\mu}_k^L(Q,W,d)$.\hfill$\Box$
\end{definition}

\begin{remark}
There is an analogous notion of right graded QP-mutation $\mu_k^R(Q,W,d)$: the only change is that the pre-mutated degree $d_R$ assigns $d_R(a^*)=-d(a)$ and $d_R(b^*)=1-d(b)$.\hfill$\Box$
\end{remark}

Before establishing that both left and right graded QP-mutations induce bijections between the sets of dilations, we must show that the vertex gauging in \Cref{def:vertex-gauge} commutes with graded QP-mutations. The statement reads as follows:

\begin{lemma}[Graded QP-mutation commutes with vertex gauge]\label{prop:gauge-mutation}
Let $(Q,W,d)$ be a graded QP and \(h:Q_0\to\Z\) a vertex grading. Then:

\begin{enumerate}
    \item We have an equality of (arrow) gradings
    \begin{equation}\label{eq:gauge-premutation}
 \widetilde\mu_k^L(d^h)
 =\bigl(\widetilde\mu_k^L(d)\bigr)^h,
\end{equation}
where $h$ is understood as the same function \(h\) on the mutated vertex set, and the right hand side of \Cref{eq:gauge-premutation} denotes applying the corresponding vertex gauge to $\widetilde\mu_k^L(d)$.\\

 \item The left-graded mutated QP satisfies
 \begin{equation}\label{eq:gauge-reduced-mutation}
 \mu_k^L(Q,W,d^h)
 \simeq
 \mu_k^L(Q,W,d)^h,
\end{equation}
where $\simeq$ in \Cref{eq:gauge-reduced-mutation} denotes graded right-equivalence, and its right hand side denotes applying the corresponding vertex gauge to $\mu_k^L(d)$.\\

\item Both statements above hold for right graded QP-mutations.

\end{enumerate}
\end{lemma}

\begin{proof}
For Part (1), consider the composite arrow \([ba]:i\to j\). Then
\begin{align*}
 \widetilde d^{\,h}([ba])
 &=d^h(b)+d^h(a)\\
 &=d(b)+h(j)-h(k)+d(a)+h(k)-h(i)\\
 &=\widetilde d([ba])+h(j)-h(i).
\end{align*}
Similarly, for \(a^*:k\to i\) we have
\[
 \widetilde d^{\,h}(a^*)
 =1-d^h(a)
 =1-d(a)+h(i)-h(k),
\]
and for \(b^*:j\to k\) we obtain
\[
 \widetilde d^{\,h}(b^*)
 =-d^h(b)
 =-d(b)+h(k)-h(j).
\]
These are exactly the vertex-gauge corrections for the mutated arrows, thus proving
\Cref{eq:gauge-premutation}. For Part (2), a graded splitting right-equivalence remains
graded after applying the same vertex-gauge, because every path from \(i\) to \(j\)
has its degree shifted by the common amount \(h(j)-h(i)\). Hence graded
reduction and \Cref{eq:gauge-premutation} together imply \Cref{eq:gauge-reduced-mutation}. Part (3) is proven analogously.
\end{proof}

\begin{theorem}[Graded QP-mutation gives bijection on dilation classes]\label{thm:G1-mutation}
Let $[Q,W]$ be a stable right-equivalence class of a reduced
non-degenerate QP $(Q,W)$ and let \(k\) be a vertex. Then left graded QP-mutation induces a
canonical bijection
\[
 \mu_k^L:\Gclass[Q,W]
 \xrightarrow{\sim}
 \Gclass[\mu_k(Q,W)].
\]
Furthermore, right graded QP-mutation provides its inverse.
\end{theorem}

\begin{proof}
We must verify the three equivalence relations from \Cref{def:G1}. For vertex gauge, \Cref{prop:gauge-mutation} gives compatibility with vertex gauge. For graded right-equivalence: graded pre-mutation is functorial under graded
right-equivalence, and the Graded Splitting Theorem \cite[Theorem 6.6]{AO} makes the reduced output
well-defined up to graded right-equivalence. Similarly, homogeneous stabilization is absorbed by the Graded Splitting Theorem. Therefore both left and
right graded QP-mutations descend to maps on dilation classes. Since \cite[Lemma 6.9]{AO} shows that
\[
 \mu_k^R\mu_k^L(Q,W,d)
 \simeq(Q,W,d),
 \qquad
 \mu_k^L\mu_k^R(Q,W,d)
 \simeq(Q,W,d),
\]
where $\simeq$ denotes a graded right-equivalence, the induced maps on the set of dilations must therefore be inverse bijections.
\end{proof}

\section{Bigraded Ginzburg algebras and an invariant polynomial}
\label{sec:euler}

The object of this section is to develop the study of the graded Euler pairing of a differential bigraded Ginzburg algebra associated to a graded quiver with potential. Two conceptually important facts we establish are \Cref{thm:intrinsic-euler}, showing that the determinant of the graded Euler pairing is a well-defined invariant of the dilation class, and \Cref{prop:mutation-euler}, showing that it is invariant under graded QP mutations. Several examples are provided in \Cref{ssec:two_examples} and \Cref{ssec:tree_quivers}.

\subsection{The internal grading of the Ginzburg dg algebra of a graded QP}

Let \((Q,W)\) be a QP and $\Gamma(Q,W)$ its complete 3-Calabi-Yau Ginzburg dg algebra, cf.~\cite[Section 4.2]{Ginzburg} and \cite[Section~2.6]{KellerYang}. That is, $\Gamma(Q,W)=\widehat{\kk{\tilde{Q}}}$ is the completed path algebra of the doubled quiver $\tilde{Q}$, where we take the completion as in \cite{KellerYang}. The original arrows $a:i\to j$, $a\in Q_1$, are graded in cohomological degree 0, the dual arrows $a^*:j\to i$, $a^*\in \tilde{Q}_1$, are given cohomological degree -1 and the loop arrows $t_i:i\to i$, dual to the vertices, are given cohomological degree -2. The differential $\Gamma(Q,W)$ is declared on these generators to be

\begin{align*}
 \partial(a)=0,\quad \partial(a^*)=\partial_aW,\quad \partial(t_i)=e_i\sum_{a\in Q_1}[a,a^*]e_i.
\end{align*}

If $(Q,W,d)$ is a graded QP, then the dg algebra $\Gamma(Q,W)$ can be further endowed with an internal grading, a.k.a.~an Adams grading, as follows:

\begin{definition}[Bigraded Ginzburg dga of a graded QP]\label{def:bigradedGinzburg} Let \((Q,W,d)\) be a graded QP. By definition, the internal (or Adams) grading on the Ginzburg dg algebra $\Gamma(Q,W)$ of $(Q,W)$ is given by:
\begin{enumerate}
    \item The original arrows $a:i\to j$ in $Q_1\sse \tilde{Q}_1$ are assigned internal degree $d(a)$,
    \item The dual arrows $a^*:j\to i$ in $\tilde{Q}_1$ are assigned internal degree $1-d(a)$,
    \item The loop arrows $t_i:i\to i$ in $\tilde{Q}_1$ are assigned internal degree $1$.
\end{enumerate}
The resulting differential bigraded Ginzburg algebra is denoted by $\Gamma(Q,W,d)$. In table form, the bidegrees of the generators of $\Gamma(Q,W,d)$ are
\begin{center}
\begin{tabular}{c|c|c}\label{table:gradings}
 generator & cohomological degree & internal degree\\
\midrule
 \(a:i\to j\) & \(0\) & \(d(a)\)\\
 \(a^*:j\to i\) & \(-1\) & \(1-d(a)\)\\
 \(t_i:i\to i\) & \(-2\) & \(1\)
\end{tabular}
\end{center}
\hfill$\Box$
\end{definition}

\noindent Since \(W\) has internal degree one, the differential in $\Gamma(Q,W,d)$ has bidegree
\((1,0)\). Note also that sending all arrows to zero produces an augmentation $\Gamma(Q,W,d)\to R_Q$.

\subsection{Module categories of the bigraded Ginzburg algebra}\label{ssec:module_categories_bigraded} The topology of $\Gamma(Q,W)$ -- and in our case $\Gamma(Q,W,d)$ -- should be considered, as in \cite[Appendix]{KellerYang} and see also \cite{Keller2026_GradedPreprint}. The complete Ginzburg algebra $\Gamma(Q,W)$ is pseudocompact for the $\mathfrak{m}_{\tilde{Q}}$-adic topology, and we consider {\it continuous} right dg modules over it, cf.~\cite[Appendix]{KellerYang}. In the graded case, for $\Gamma:=\Gamma(Q,W,d)$, the internal grading is incorporated by considering continuous dg modules $M$ over $\Gamma(Q,W)$ equipped with closed subcomplexes $M^{(r)}$, $r\in\Z$, and a topological product decomposition
\[
  M\cong\prod_{r\in\Z}M^{(r)},\quad\mbox{such that}\quad
  M^{(r)}\Gamma^{(s)}\subset M^{(r+s)}\quad\mbox{and}\quad
  \partial_M(M^{(r)})\subset M^{(r)}.
\]
In addition, we require a neighborhood basis by homogeneous dg-submodules of finite total codimension.  Equivalently, such a continuous differential bigraded module $M$ over $\Gamma(Q,W,d)$ is an inverse limit of finite-dimensional \emph{internally graded} discrete dg modules and degree-preserving transition maps. As before, the cohomological differential therefore has bidegree $(1,0)$, and a morphism in the graded category is continuous of bidegree $(0,0)$.\\

Localizing the homotopy category of these continuous internally graded dg modules over $\Gamma(Q,W,d)$ at quasi-isomorphisms gives the derived category $\D^{\gr}_{\cont}(\Gamma)$. For $M,N\in D^{\mathrm{gr}}_{\mathrm{cont}}(\Gamma)$, we define
\[
\operatorname{Ext}^p_\Gamma(M,N)_r
:=
H^p\!\left(
\operatorname{RHom}_\Gamma(M,N)_{-r}
\right)
\cong
\operatorname{Hom}_{D^{\mathrm{gr}}_{\mathrm{cont}}(\Gamma)}
\bigl(M,N[p]\langle-r\rangle\bigr).
\]
Here $[p]$ denotes the cohomological shift and $\langle r\rangle$
the internal shift. Our internal-shift convention is $(N\langle r\rangle)_p:=N_{p+r}$. Thus an element of internal degree $p$ in $N$ has internal degree
$p-r$ in $N\langle r\rangle$. In particular,
$\Bbbk\langle r\rangle$ is concentrated in ordinary internal degree
$-r$. The subscript $r$ on
$\operatorname{Ext}^p_\Gamma(M,N)_r$ records the shift of the internal
weight, i.e.~ the negative of the (standard) internal
degree of the corresponding homogeneous morphism.\\

We will be working with the following subcategories of $\D^{\gr}_{\cont}(\Gamma)$:

\begin{definition}\label{def:categories_over_bigraded}
Let $(Q,W,d)$ be a graded QP and $\D^{\gr}_{\cont}(\Gamma)$ its derived category of continuous internally graded dg modules over $\Gamma:=\Gamma(Q,W,d)$. By definition:

\begin{enumerate}
    \item The perfect category $\per^{\gr}(\Gamma)$ is the smallest idempotent-complete triangulated subcategory of $\D^{\gr}_{\cont}(\Gamma)$ containing the free module and all of its internal shifts:
\begin{equation}\label{eq:perfect-category}
  \per^{\gr}(\Gamma)
  :=\thick\set{\Gamma\angles r\suchthat r\in\Z}
  =\thick\set{e_i\Gamma[p]\angles r
      \suchthat i\in Q_0,\ p,r\in\Z}.
\end{equation}

\item The finite-dimensional category $\D^{\gr}_{\fd}(\Gamma)$ is
\[
  \D^{\gr}_{\fd}(\Gamma)
  :=\set{M\in\D^{\gr}_{\cont}(\Gamma)\suchthat
  \sum_p\dim_\kk H^p(M)<\infty}.
\]
\end{enumerate}
\end{definition}

For $\per^{\gr}(\Gamma)$, its objects are precisely those objects which are isomorphic in the derived category to direct summands of finite-cell, topologically semi-free dg modules.\footnote{That is, we start with finitely many shifted modules $e_i\Gamma[p]\angles r$, form finitely many cones, and then take finite direct sums and direct summands.} For $\D^{\gr}_{\fd}(\Gamma)$ note that, since the total cohomologies in $\D^{\gr}_{\fd}(\Gamma)$ are finite-dimensional, only finitely many internal weights can occur in its objects.\\

Let us now focus on certain objects of $\D^{\gr}_{\cont}(\Gamma)$, namely, the simple modules associated to the vertices of $Q$. By definition, the simple right dg module $S_i$ over $\Gamma$ associated to the vertex $i\in Q_0$ is
\begin{equation}\label{eq:Si}
  S_i:=e_iR_Q,
\end{equation}

\noindent where $\Gamma$ acts via the augmentation $\Gamma\lr R_Q$ that sends all arrows to zero. See \Cref{ssec:two_examples} for some examples of such simple modules. Such simple modules $S_i$ are concentrated in cohomological and internal degree zero.

\begin{proposition}[Finiteness of Ext groups between simples]\label{prop:ext-finite} Let $(Q,W,d)$ be a graded QP and $\Gamma=\Gamma(Q,W,d)$ its associated Ginzburg algebra. Then:
\begin{enumerate}
    \item Each simple $S_i$ belongs to $\D^{\gr}_{\fd}(\Gamma)\cap\per^{\gr}(\Gamma)$.
    
    \item For every pair of vertices $i,j\in Q_0$, the vector spaces $\Ext^{p}_{\Gamma}(S_i,S_j)$ vanish for $p\notin\{0,1,2,3\}$. In particular, $\Ext^{p}_{\Gamma}(S_i,S_j)_r$ vanish for $p\notin\{0,1,2,3\}$

    \item For every pair of vertices $i,j\in Q_0$, the vector spaces $\Ext^{p}_{\Gamma}(S_i,S_j)$ are finite-dimensional for all $p\in\Z$. In particular, so are $\Ext^{p}_{\Gamma}(S_i,S_j)_r$.
\end{enumerate}
\end{proposition}

\begin{proof}
For Part (1), note that the module $S_i$ is one-dimensional by \Cref{eq:Si}, and hence it necessarily lies in $\D^{\gr}_{\fd}(\Gamma)$. For its perfectness, consider the explicit cofibrant resolution of $S_i$ constructed in \cite[Section~2.14]{KellerYang} and note that their construction is homogeneous for the additional internal grading because the Ginzburg differential has bidegree $(1,0)$. (See also \Cref{ssec:two_examples} below for some explicit instances of these cofibrant resolutions.) Since the resolution is a finite semi-free complex built from the projectives $e_j\Gamma$, it follows that $S_i$ is perfect and thus Part (1) is established.\\

\noindent For Part (2), it follows from \cite[Lemma~2.15]{KellerYang} that a basis of $\Hom_{\D(\Gamma)}(S_i,S_j[p])$ is given by generators of the Ginzburg quiver from $j$ to $i$ of cohomological degree $-p+1$, together with the identity when $i=j$ and $p=0$.  Since the Ginzburg quiver has arrows only in cohomological degrees $0,-1,-2$, Part (2) follows. Finally, finiteness of $Q$ gives finite-dimensionality and finite internal-weight support, concluding Part (3).
\end{proof}

\subsection{Graded Euler pairing}\label{ssec:graded_Euler_pairing} Let $(Q,W,d)$ be a graded QP and $\Gamma=\Gamma(Q,W,d)$ its associated differential bigraded Ginzburg algebra, as in \Cref{def:bigradedGinzburg}. If the internal grading is not taken into account, the Euler form in the Grothendieck group of the derived category of finite-dimensional modules over $\Gamma(Q,W)$ is an integer, to be interpreted as an element of the Grothendieck group of the trivial group. By incorporating the internal grading of $\Gamma$, as in \Cref{def:bigradedGinzburg}, the Euler form in the Grothendieck group of $\D^{\gr}_{\fd}(\Gamma)$ should be valued in $K_0(\mathbb{G}_m)\cong\Z[t,t^{-1}]$, i.e.~it should be a $\mathbb{G}_m$-equivariant (or $\Z$-graded) Euler characteristic.\footnote{It is possible to phrase the study of $\Z$-graded QPs in terms of $\mathbb{G}_m$ actions $\rho:\mathbb{G}_m\lr\mbox{Aut}_{R_Q}(\widehat{\kk Q})$ on the completed path algebra which act via $\rho_z(W)=zW$. In that viewpoint, it is natural to consider an equivariant Euler characteristic.} Here we use the coordinate $t$ to correspond with the internal shift, so that $[M\langle r\rangle]=t^r[M]$.\\

The simple modules $S_i$ are such that the Grothendieck group $K_0(\D^{\gr}_{\fd}(\Gamma))$ is a free $\Z[t,t^{-1}]$-module with basis $[S_i]$. Thus we can express the graded Euler pairing in terms of this basis of simple classes:

\begin{definition}[Graded Euler pairing in the simple basis]\label{def:euler}
Let $(Q,W,d)$ be a graded QP and $\Gamma=\Gamma(Q,W,d)$ its associated differential bigraded Ginzburg algebra. By definition, the graded Euler matrix of
\((Q,W,d)\) is
\begin{equation}\label{eq:euler-ext}
 \Eul_{(Q,W,d)}(t)_{ij}
 :=\sum_{p,r\in\Z}(-1)^p
 \dim_\kk\Ext^{p}_{\Gamma}(S_i,S_j)_r\cdot t^{r}.
\end{equation}
\hfill$\Box$\end{definition}
\noindent That is, the entry $\Eul_{(Q,W,d)}(t)_{ij}$ in \Cref{def:euler} is the graded Euler characteristic of $\RHom_{\Gamma}(S_i,S_j)$. It is written in
the internal-shift coordinate $t$, so that a copy of
$\Bbbk\langle r\rangle$ contributes $t^r$. The expansion \eqref{eq:euler-ext} can also be written as

\begin{equation}\label{eq:euler-ext2}
 \Eul_{(Q,W,d)}(t)_{ij}
 =\sum_{p\in\Z}(-1)^p
 \mbox{ch}_t\Ext^{p}_{\Gamma}(S_i,S_j)=\sum_{p=0}^3(-1)^p
 \mbox{ch}_t\Ext^{p}_{\Gamma}(S_i,S_j)
\end{equation}
where the second equality in \eqref{eq:euler-ext2} is \Cref{prop:ext-finite}.(2). In \eqref{eq:euler-ext2}, $ \mbox{ch}_t\Ext^{p}_{\Gamma}(S_i,S_j)$ denotes the equivariant character for the corresponding $\Gm$-action on $\Ext^{p}_{\Gamma}(S_i,S_j)$. Specifically, we used the equivariant character
\[
\operatorname{ch}_t(V)
:=
\sum_{r\in\mathbb Z}
\dim_\Bbbk\!\left(V_{-r}\right)t^r
\]
for a graded vector space $V$. Equivalently, if $V\cong\bigoplus_{r\in\mathbb Z}V_r\otimes_\Bbbk\Bbbk\langle r\rangle$ then $\operatorname{ch}_t(V)=\sum_r(\dim_\Bbbk V_r)t^r$. Such graded (or equivariant) Euler characteristic \eqref{eq:euler-ext} is additive in distinguished triangles and therefore defines an Euler pairing on the equivariant Grothendieck group. It is a sesquilinear form for the involution $t\to t^{-1}$. Note that \Cref{prop:ext-finite} implies that the entries \eqref{eq:euler-ext} are Laurent polynomials in the $t$-variable, as there are only finitely many non-vanishing coefficients and those are themselves finite. To ease notation we often denote $\Eul_{ij}(t):=\Eul_{(Q,W,d)}(t)_{ij}$ if $(Q,W,d)$ is clear by context.

\begin{remark} (1) The more ambitious prospect of considering the bivariate polynomial
\begin{align*}
  \Poin_{\Gamma}(u,t)_{ij}
  &=\sum_{p,r\in\Z}
  \dim_\kk\Ext^{p}_{\Gamma}(S_i,S_j)_ru^pt^{r}
\end{align*}
does not work for our purposes, as its behavior under equivalences and mutations is not by congruence. In part, this failure of functoriality for the above bivariate polynomial comes from the fact that such double-Euler characteristic is not additive in distinguished triangles.\\

\noindent (2) Some formulas in the quiver literature index the Euler matrix by arrows
\(i\to j\) rather than by \(\Ext(S_i,S_j)\): the resulting matrix is then \(\Eul(t)^{\mathsf T}\). Note that, nevertheless, their determinants are unaffected.
\hfill$\Box$\end{remark}

It will be useful to have a more hands-on description of the graded Euler pairing directly in terms of the quiver and its grading. We establish combinatorial formulas for the entries as follows:

\begin{proposition}[Combinatorial Euler pairing in the simple basis]\label{prop:euler-formula}
Let $(Q,W,d)$ be a graded QP. In the notation above,
\begin{equation}\label{eq:euler-formula}
 \boxed{
 \Eul_{ij}(t)
 =(1-t)\delta_{ij}
 -\sum_{a:j\to i}t^{d(a)}
 +\sum_{a:i\to j}t^{1-d(a)}.}
\end{equation}
\end{proposition}

\begin{proof}
It suffices to use \cite[Lemma~2.15]{KellerYang}, which computes
\(\Hom(S_i,S_j[p])\) from the generators of the Ginzburg quiver that are directed from
\(j\) to \(i\). The contributions are thus as follows:

\begin{enumerate}
    \item The identity contributes \(+1\) for \(i=j,p=0\), corresponding to the summand $\delta_{ij}$ in \Cref{eq:euler-formula}.\\
    \item An original arrow \(a:j\to i\), $a\in Q_1$, has cohomological degree zero and internal degree \(d(a)\). The corresponding element of
$\operatorname{Ext}^1_\Gamma(S_i,S_j)$ has (standard) internal degree
$-d(a)$ and therefore contributes $-t^{d(a)}$. These are the contributions to the first (negative) sum in \Cref{eq:euler-formula}.\\
    \item If \(a:i\to j\), the opposite arrow \(a^*:j\to i\) has cohomological degree \(-1\) and internal degree \(1-d(a)\). The corresponding element of $\operatorname{Ext}^2_\Gamma(S_i,S_j)$ has standard internal degree $d(a)-1$ and hence contributes \(+t^{1-d(a)}\). These are the contributions to the second sum in \Cref{eq:euler-formula}.\\
    \item Finally, a loop \(t_i:i\to i\) has internal degree 1 and the
corresponding element of
$\operatorname{Ext}^3_\Gamma(S_i,S_i)$ has internal degree
$-1$. Thus a loop contributes $-t$. This is the term $-t\delta_{ij}$ in \Cref{eq:euler-formula}.
\end{enumerate}
Thus adding these four contributions gives \eqref{eq:euler-formula}.
\end{proof}

\noindent Note that \Cref{eq:euler-formula} depends only on the graded arrow bimodule and not
on the coefficients of \(W\). It is particularly apparent from \Cref{eq:euler-formula} that the ungraded Euler pairing, obtained by setting $t=1$, yields the exchange matrix of the quiver $Q$. In this sense, the graded Euler pairing can be understood as a deformation of the exchange matrix with parameter $t$.\\

The key point now is to understand how the equivariant Euler characteristics in \Cref{eq:euler-ext} behave in the following situations:\\

\begin{itemize}
    \item[(a)] The first task is to show that, while $\Eul_{(Q,W,d)}(t)_{ij}$ is defined using $(Q,W,d)$, the conjugacy class of the matrix $(\Eul_{(Q,W,d)}(t)_{ij})\in\mbox{Mat}(\Z[t,t^{-1}])$ is actually a well-defined invariant associated to any dilation class $\delta\in\Gclass[Q,W]$. In fact, more is true, as its monomial-conjugacy class will be well-defined for a dilation class.\\

    \item[(b)] The second task is to show that the congruence class of the matrix $(\Eul_{(Q,W,d)}(t)_{ij})$ is an invariant of the graded QP-mutation class of $(Q,W,d)$.\\
\end{itemize}

\noindent These verifications can be made using categorical arguments or by explicitly computing with \Cref{eq:euler-formula}. \Cref{sec:euler-dilation} implements task $(a)$ and \Cref{sec:euler-mutation} addresses task $(b)$.

\subsection{Two examples}\label{ssec:two_examples} Let us provide the details for how to compute in some examples the graded Euler pairing \eqref{eq:euler-ext}, in the basis of simple classes, as defined in \Cref{def:euler}.\\

First, let $(Q,W)$ be the 4-cycle QP, with quiver 
\begin{center}
\begin{tikzcd}[row sep=large, column sep=large]
    4 \arrow[d, "d"'] & 3 \arrow[l, "c"'] \\
    1 \arrow[r, "a"'] & 2 \arrow[u, "b"']
\end{tikzcd}
\end{center}
and potential $W = dcba$. For the arrow grading $\deg:Q_1\lr\Z$, let us assign internal degrees as
$$\deg(a)=1,\deg(b)=\deg(c)=\deg(d)=0,$$ and thus \Cref{def:bigradedGinzburg} dictates $$\deg(a^*)=0,\deg(b^*)=\deg(c^*)=\deg(d^*)=1,$$
while all the loops $t_i$ have $\deg(t_i)=1$. Note that the potential $W$ is indeed homogeneous of degree 1. This defines the differential bigraded Ginzburg algebra $\Gamma(Q,W,\deg)$ in this case.\\

\noindent Both the proof of \Cref{prop:ext-finite} and \Cref{prop:euler-formula} implicitly use cofibration resolutions of the simple modules $S_i$, as built in \cite[Section 2.14]{KellerYang}. Let us recall here that, as the proof of \cite[Section 3.12]{KellerYang} illustrates, such cofibrant resolution $pS_i$ of $S_i$ is constructed as a shifted mapping cone, with underlying vector space:

\begin{equation}\label{eq:cofibrant_resolution}
pS_i := \Sigma^3 P_i \oplus \left(\bigoplus_{\alpha \in Q_1: s(\alpha)=i} \Sigma^2 P_{t(\alpha)}\right) \oplus \left(\bigoplus_{\beta \in Q_1: t(\beta)=i} \Sigma P_{s(\beta)} \oplus P_i\right),
\end{equation}
where $(P_i,\partial_{P_i})$ denotes the projective dg module associated to the vertex $i\in Q_0$, and $\Sigma$ denotes suspension. The corresponding differential $\partial_{pS_i}$ acting on the vector space \eqref{eq:cofibrant_resolution} is given by the following block lower-triangular matrix, where the entries operate via left multiplication:
\begin{equation}\label{eq:cofibration_resolution_differential}
\partial_{pS_i} := \begin{pmatrix} 
\partial_{\Sigma^3 P_i} & 0 & 0 & 0 \\ 
\alpha & \partial_{\Sigma^2 P_{t(\alpha)}} & 0 & 0 \\ 
-\beta^* & -\partial_{\alpha\beta}W & \partial_{\Sigma P_{s(\beta)}} & 0 \\ 
t_i & \alpha^* & \beta & \partial_{P_i} 
\end{pmatrix}.
\end{equation}
Recall that we use the internal grading convention $(M\langle r \rangle)_{q} = M_{q+r}$,  cf.~\Cref{ssec:module_categories_bigraded}. In particular, with this convention, if a component of a differential is given by left multiplication by a homogeneous element of internal degree $s$, then the corresponding source projective module must be shifted by
$\langle-s\rangle$ in order for that component to have
internal degree zero. The minimal semi-free cofibrant resolutions $pS_i$ for the four simples in our example of $(Q,W)$ then are:
\begin{align*}
pS_1 &= P_1[3]\langle -1 \rangle \oplus P_2[2]\langle 0 \rangle \oplus P_4[1]\langle 0 \rangle \oplus P_1[0]\langle 0 \rangle \\
pS_2 &= P_2[3]\langle -1 \rangle \oplus P_3[2]\langle -1 \rangle \oplus P_1[1]\langle -1 \rangle \oplus P_2[0]\langle 0 \rangle \\
pS_3 &= P_3[3]\langle -1 \rangle \oplus P_4[2]\langle -1 \rangle \oplus P_2[1]\langle 0 \rangle \oplus P_3[0]\langle 0 \rangle \\
pS_4 &= P_4[3]\langle -1 \rangle \oplus P_1[2]\langle -1 \rangle \oplus P_3[1]\langle 0 \rangle \oplus P_4[0]\langle 0 \rangle
\end{align*}

\begin{table}[h]
\centering
\caption{Non-vanishing internally graded Ext groups $\Ext^p(S_i,S_j)$ for the non-degenerate 4-cycle graded QP $(Q,W,\deg)$.}
\label{table:Ext_groups1}
\begin{tabular}{@{}l c c c c@{}}
\toprule
Pair $(S_i, S_j)$ & $p=0$ & $p=1$ & $p=2$ & $p=3$ \\ 
\midrule
$(S_1, S_1)$ & $k\langle 0 \rangle$ & 0 & 0 & $k\langle 1 \rangle$ \\
$(S_2, S_2)$ & $k\langle 0 \rangle$ & 0 & 0 & $k\langle 1 \rangle$ \\
$(S_3, S_3)$ & $k\langle 0 \rangle$ & 0 & 0 & $k\langle 1 \rangle$ \\
$(S_4, S_4)$ & $k\langle 0 \rangle$ & 0 & 0 & $k\langle 1 \rangle$ \\
\addlinespace
$(S_1, S_2)$ & 0 & 0 & $k\langle 0 \rangle$ & 0 \\
$(S_2, S_3)$ & 0 & 0 & $k\langle 1 \rangle$ & 0 \\
$(S_3, S_4)$ & 0 & 0 & $k\langle 1 \rangle$ & 0 \\
$(S_4, S_1)$ & 0 & 0 & $k\langle 1 \rangle$ & 0 \\
\addlinespace
$(S_2, S_1)$ & 0 & $k\langle 1 \rangle$ & 0 & 0 \\
$(S_3, S_2)$ & 0 & $k\langle 0 \rangle$ & 0 & 0 \\
$(S_4, S_3)$ & 0 & $k\langle 0 \rangle$ & 0 & 0 \\
$(S_1, S_4)$ & 0 & $k\langle 0 \rangle$ & 0 & 0 \\
\bottomrule
\end{tabular}
\end{table}

\noindent By considering the homology of the morphism complex $\mathcal{H}om_{\Gamma}(pS_i, S_j)$, a summand $P_j[h]\langle r\rangle$ contributes a one-dimensional
vector space in internal degree $r$, equivalently to the component $\operatorname{Ext}^p_\Gamma(S_i,S_j)_{-r}$. The resulting non-zero graded Ext groups between the simples of $\Gamma(Q,W,d)$ are computed in \Cref{table:Ext_groups1}.\\

In order to compute the graded Euler pairing, as in \Cref{def:euler}, let $t$ be the formal parameter keeping track of the shift in internal degree. From \Cref{table:Ext_groups1}, the equivariant Euler pairing in the basis of the simples gives the matrix:

\[
\Eul_{(Q,W,d)}(t) = \begin{pmatrix} 
1 - t & 1 & 0 & -1 \\ 
-t & 1 - t & t & 0 \\ 
0 & -1 & 1 - t & t \\ 
t & 0 & -1 & 1 - t 
\end{pmatrix}.
\]
This concludes this first example. For our second example, we will consider the following QP $(Q',W')$, which adds two vertices to $(Q,W)$:

\begin{center}
$Q':=$\begin{tikzcd}[row sep=large, column sep=large]
    4 \arrow[d, "d"'] & 3 \arrow[l, "c"'] \arrow[r, "g"] & 6 \arrow[d, "f"] \\
    1 \arrow[r, "a"'] & 2 \arrow[u, "b"'] & 5 \arrow[l, "e"]
\end{tikzcd}
\end{center}
with potential $W' := dcba + efgb$. The underlying quiver is of finite mutation type $E_6$. The internal grading $\deg'$ for the graded QP $(Q',W',\deg')$ is that of $(Q,W,d)$ now extended such that
$$\deg(e)=1,\quad \deg(f)=\deg(g)=0,$$
and so their duals have degrees
$$\deg(e^*)=0,\quad \deg(f^*)=\deg(g^*)=1.$$ 

\noindent Following \eqref{eq:cofibrant_resolution} as above, the cofibrant resolutions $pS_i$ for the simples of $Q'$ have underlying vector spaces:
\begin{align*}
pS_1 &= P_1[3]\langle -1 \rangle \oplus P_2[2]\langle 0 \rangle \oplus P_4[1]\langle 0 \rangle \oplus P_1[0]\langle 0 \rangle \\
pS_2 &= P_2[3]\langle -1 \rangle \oplus P_3[2]\langle -1 \rangle \oplus P_1[1]\langle -1 \rangle \oplus P_5[1]\langle -1 \rangle \oplus P_2[0]\langle 0 \rangle \\
pS_3 &= P_3[3]\langle -1 \rangle \oplus P_4[2]\langle -1 \rangle \oplus P_6[2]\langle -1 \rangle \oplus P_2[1]\langle 0 \rangle \oplus P_3[0]\langle 0 \rangle \\
pS_4 &= P_4[3]\langle -1 \rangle \oplus P_1[2]\langle -1 \rangle \oplus P_3[1]\langle 0 \rangle \oplus P_4[0]\langle 0 \rangle \\
pS_5 &= P_5[3]\langle -1 \rangle \oplus P_2[2]\langle 0 \rangle \oplus P_6[1]\langle 0 \rangle \oplus P_5[0]\langle 0 \rangle \\
pS_6 &= P_6[3]\langle -1 \rangle \oplus P_5[2]\langle -1 \rangle \oplus P_3[1]\langle 0 \rangle \oplus P_6[0]\langle 0 \rangle
\end{align*}

\begin{table}[h]
\centering
\caption{Non-vanishing internally graded Ext groups $\Ext^p(S_i,S_j)$ for our second example, the non-degenerate graded QP $(Q',W',\deg')$.}
\label{table:Ext_groups2}
\begin{tabular}{@{}l c c c c@{}}
\toprule
Pair $(S_i, S_j)$ & $p=0$ & $p=1$ & $p=2$ & $p=3$ \\ 
\midrule
All $(S_i, S_i)$ & $k\langle 0 \rangle$ & 0 & 0 & $k\langle 1 \rangle$ \\
\addlinespace
$(S_1, S_2)$ & 0 & 0 & $k\langle 0 \rangle$ & 0 \\
$(S_2, S_3)$ & 0 & 0 & $k\langle 1 \rangle$ & 0 \\
$(S_3, S_4)$ & 0 & 0 & $k\langle 1 \rangle$ & 0 \\
$(S_4, S_1)$ & 0 & 0 & $k\langle 1 \rangle$ & 0 \\
$(S_5, S_2)$ & 0 & 0 & $k\langle 0 \rangle$ & 0 \\
$(S_3, S_6)$ & 0 & 0 & $k\langle 1 \rangle$ & 0 \\
$(S_6, S_5)$ & 0 & 0 & $k\langle 1 \rangle$ & 0 \\
\addlinespace
$(S_2, S_1)$ & 0 & $k\langle 1 \rangle$ & 0 & 0 \\
$(S_3, S_2)$ & 0 & $k\langle 0 \rangle$ & 0 & 0 \\
$(S_4, S_3)$ & 0 & $k\langle 0 \rangle$ & 0 & 0 \\
$(S_1, S_4)$ & 0 & $k\langle 0 \rangle$ & 0 & 0 \\
$(S_2, S_5)$ & 0 & $k\langle 1 \rangle$ & 0 & 0 \\
$(S_6, S_3)$ & 0 & $k\langle 0 \rangle$ & 0 & 0 \\
$(S_5, S_6)$ & 0 & $k\langle 0 \rangle$ & 0 & 0 \\
\bottomrule
\end{tabular}
\end{table}

The corresponding differentials, as above, follow from \eqref{eq:cofibration_resolution_differential}. The non-vanishing graded Ext groups between the simples of $\Gamma(Q',W',d')$ are summarized in \Cref{table:Ext_groups2}. Consequently, the graded Euler pairing for $(Q',W',\deg')$ in the basis of simple classes is given by:

\[
\Eul_{(Q',W',\deg')}(t) = \begin{pmatrix} 
1 - t & 1 & 0 & -1 & 0 & 0 \\ 
-t & 1 - t & t & 0 & -t & 0 \\ 
0 & -1 & 1 - t & t & 0 & t \\ 
t & 0 & -1 & 1 - t & 0 & 0 \\
0 & 1 & 0 & 0 & 1 - t & -1 \\
0 & 0 & -1 & 0 & t & 1 - t
\end{pmatrix}
\]

\noindent The determinant of $\Eul_{(Q',W',\deg')}(t)$ exactly recovers the Alexander polynomial of the $(3,4)$-torus knot:
\begin{equation}\label{eq:E6_determinant_Alexander}
\det(\Eul_{(Q',W',\deg')}(t)) = t^{6} - t^{5} + t^{3} - t + 1 = \Delta_{3,4}(t).
\end{equation}

\noindent To motivate such connection, note that the $(3,4)$-torus knot is the link of the $E_6$ simple plane curve singularity. By \cite[Theorem 18.3]{FPST}, the quiver associated to any real Morsification of such $E_6$ singularity is of mutation type $E_6$. The example of the $E_6$ simple plane curve singularity and its quivers will be further discussed in \Cref{ssec:examples}. The equality in \eqref{eq:E6_determinant_Alexander} is an instance of the more general phenomenon proven in \Cref{cor:cokernel_module_is_Alexander_algebraic_links}.

\subsection{Graded Euler pairing and dilation classes}\label{sec:euler-dilation} Let $(Q,W,d)$ be a graded QP and $(\Eul_{ij}(t))$ the matrix of graded Euler characteristics in the simple basis, as in \Cref{def:euler}. By \Cref{def:G1}, a dilation is an equivalence class under vertex gauging, homogeneous stabilizations and graded right-equivalences. Let us study how $(\Eul_{ij}(t))$ behaves under these equivalence relations:

\begin{lemma}[Effect of vertex gauge]\label{prop:euler-gauge}
Let $(Q,W,d)$ be a graded QP, $h:Q_0\lr\Z$ a vertex grading, and \(d^h\) the result of applying a vertex-gauge to $d$, as in \eqref{eq:gauge}. Consider the diagonal matrices
\[
 D_h(t)=\diag(t^{h(i)})_{i\in Q_0}.
\]
Then the resulting matrix $\Eul_{(Q,W,d^h)}(t)$ of graded Euler characteristics is given by
\begin{equation}\label{eq:euler-gauge}
 \Eul_{(Q,W,d^h)}(t)
 =D_h(t)\Eul_{(Q,W,d)}(t)D_h(t)^{-1}.
\end{equation}
\end{lemma}

\begin{proof}
This follows from \Cref{prop:euler-formula}. Indeed, for an arrow \(a:j\to i\),
\[
 t^{d^h(a)}=t^{h(i)-h(j)}t^{d(a)}.
\]
For an arrow \(a:i\to j\),
\[
 t^{1-d^h(a)}=t^{h(i)-h(j)}t^{1-d(a)}.
\]
Thus every \((i,j)\)-entry of \Cref{eq:euler-formula} is multiplied by
\(t^{h(i)-h(j)}\), which implies
\Cref{eq:euler-gauge}.
\end{proof}

\begin{lemma}[Effect of homogeneous stabilization]\label{lem:trivial-cancel}
Let $(Q,W,d)$ be a graded QP. Then, adjoining or deleting a homogeneous trivial QP summand to $(Q,W,d)$ does not change the Euler matrix, i.e.~ the resulting Euler matrix is equal to $\Eul_{(Q,W,d)}(t)$.
\end{lemma}

\begin{proof}
It suffices to consider arrows
\[
 x:i\to j,
 \qquad
 y:j\to i,
 \qquad\mbox{with }
 d(x)+d(y)=1,
\]
and the potential \(yx\).  In the \((j,i)\)-entry, \(x\) contributes
\(-t^{d(x)}\) and \(y\) contributes
\(+t^{1-d(y)}=+t^{d(x)}\), and so these contributions cancel. The two contributions in the
\((i,j)\)-entry cancel similarly, and no other entries change, thus proving that $\Eul_{(Q,W,d)}(t)$ remains the same.
\end{proof}

\begin{lemma}[Effect of graded right-equivalence]\label{lem:right-equivalence}
Let $(Q,W,d)$ be a graded QP. Then, applying a graded right-equivalence does not change the Euler matrix, i.e.~ the resulting Euler matrix is equal to $\Eul_{(Q,W,d)}(t)$.
\end{lemma}

\begin{proof}
Indeed, under a graded right-equivalence, the induced map on
\(\m/\m^2\) is an isomorphism of graded vertex-bimodules.  Therefore, for
every ordered pair \((i,j)\), it preserves the multiset of arrow degrees
from \(i\) to \(j\). Thus Formula \eqref{eq:euler-formula} gives literal equality
of the matrices.
\end{proof}

Lemmas \ref{prop:euler-gauge}, \ref{lem:trivial-cancel} and \ref{lem:right-equivalence} together imply:

\begin{corollary}[Graded Euler pairing and dilations]\label{thm:intrinsic-euler}
Let $(Q,W)$ be a QP and \(\delta\in\Gclass[Q,W]\) a dilation class. Then the following holds:
\begin{enumerate}
    \item The conjugation class of
\(\Eul_{(Q',W',d')}(t)\) is independent of the representative $(Q',W',d')$ of the dilation class $\delta\in\Gclass[Q,W]$. That is, if $(Q_1,W_1,d_1),(Q_2,W_2,d_2)$ represent the same class in $\Gclass[Q,W]$ then $\Eul_{(Q_1,W_1,d_1)}(t)$ and $\Eul_{(Q_2,W_2,d_2)}(t)$ are conjugate by a matrix in $\GL_{|Q_0|}(\Z[t,t^{-1}])$.\\

\item The isomorphism type of the cokernel $\coker\Eul_\delta(t)$, considered as a $\Z[t,t^{-1}]$-module, is a well-defined invariant of the dilation class $\delta\in\Gclass[Q,W]$.\\

\item The determinant
\begin{equation}\label{eq:intrinsic-det}
 \det\Eul_\delta(t)\in\Z[t,t^{-1}]
\end{equation}
is a well-defined invariant of the dilation class $\delta\in\Gclass[Q,W]$.
\end{enumerate} 
\end{corollary}

\begin{remark}
Lemmas \ref{prop:euler-gauge}, \ref{lem:trivial-cancel} and \ref{lem:right-equivalence} actually imply the stronger statement that the diagonal-monomial conjugacy class of the graded Euler matrix is well-defined for a given dilation class. That is, for any two representatives of the same dilation class, the corresponding graded Euler matrices lie in the same orbit under conjugation by diagonal matrices with monomial entries in the $t$-variable. If vertex relabeling is allowed, the conclusion is modified by using permutation matrices with monomial entries in the $t$-variable, instead of just diagonal matrices.\hfill$\Box$
\end{remark}

\subsection{Graded Euler pairing and graded QP-mutation}\label{sec:euler-mutation}
Consider the involution \(\overline{f(t)}:=f(t^{-1})\) of \(\Lambda=\Z[t,t^{-1}]\), and for a
matrix \(M\in\mbox{GL}(\Lambda)\) set  $M^\dagger=\overline M^{\mathsf T}$. This is needed because the graded Euler pairing is sesquilinear. Indeed, in our convention $[M\langle r\rangle]=t^r[M]$, the graded
Euler pairing satisfies
\[
\chi_t(M\langle r\rangle,N)
=
t^{-r}\chi_t(M,N),
\qquad
\chi_t(M,N\langle r\rangle)
=
t^r\chi_t(M,N).
\]
Thus it is sesquilinear for the involution
$\overline{f(t)}=f(t^{-1})$, as claimed. The matrix for the graded Euler pairing in the simple basis, as in \Cref{def:euler}, behaves as follows under graded QP-mutation:

\begin{proposition}[Effect of graded QP-mutation]\label{prop:mutation-euler}
Let \((Q,W,d)\) be a reduced graded QP, $k\in Q_0$ a mutable vertex and $\mu_k^L(Q,W,d)$ its left graded QP-mutation. Then there is a matrix \(M_k(t)\in\mathrm{GL}_{Q_0}(\Z[t,t^{-1}])\) such that
\begin{equation}\label{eq:mutation-congruence}
 \Eul_{\mu_k^L(Q,W,d)}(t)
 =M_k(t)^\dagger\Eul_{(Q,W,d)}(t)M_k(t)
\end{equation}
In particular, we conclude that
\begin{equation}\label{eq:mutation-coker}
 \coker\Eul_{\mu_k^L(Q,W,d)}(t)
 \cong\coker\Eul_{(Q,W,d)}(t).
\end{equation}
as $\Z[t,t^{-1}]$-modules, and thus also

\begin{equation}\label{eq:mutation-det}
 \det\Eul_{\mu_k^L(Q,W,d)}(t)
 =\det\Eul_{(Q,W,d)}(t).
\end{equation}
\end{proposition}

\begin{proof} Let us ease notation by writing $(Q',W',d'):=\mu_k^L(Q,W,d)$ and let \(\Gamma,\Gamma'\) be the corresponding complete differential bigraded Ginzburg algebras. By \cite[Theorem 3.2.(b)]{KellerYang}, QP-mutation induces mutually inverse derived equivalences of $\D_{\cont}(\Gamma)$,
restricting to derived equivalences of the corresponding perfect and finite-dimensional categories. This is achieved by constructing a $(\Gamma,\Gamma')$-bimodule $T$ such that left and right derived tensoring with $T$ induces such derived equivalences, i.e.~they construct a Fourier-Mukai kernel $T$, cf.~\cite[Section 3.4]{KellerYang}. In the case of a graded (homogeneous)
QP-mutation, every generator and differential in their mutation bimodule $T$ is
homogeneous for the internal grading prescribed by
\eqref{eq:mut-composite}--\eqref{eq:mut-outgoing}.  Thus the derived equivalence is
internal-degree preserving.\\

In order to understand the effect on the simples, and the corresponding Euler pairing, we use the mutation triangles for the simples as in
\cite[Remark 3.3]{KellerYang}. Indeed, their derived equivalence for QP-mutation implies that on the graded Grothendieck group
we have
\begin{equation}\label{eq:K0-nonk}
 [F(S_i')]=
 \begin{cases}
[S_i]+p_i(t)[S_k] & \mbox{if }i\ne k\\
-t^m[S_k] & \mbox{if }i=k,
 \end{cases}
\end{equation}
where each $p_i(t)\in\Z[t,t^{-1}]$ is a Laurent polynomial and $m\in\Z$.
The required matrix $M_k(t)\in\Z[t,t^{-1}]$ follows from \Cref{eq:K0-nonk}. Indeed, let \(M_k(t)\)
have these new classes $F(S_i')$ as columns, expressed in the old simple basis. Namely, \(M_k(t)\) is obtained from the identity by adding Laurent multiples of the \(k\)-th basis vector to
other basis vectors and replacing the \(k\)-th basis vector by
\(-t^m\) times itself. Therefore \(\det M_k(t)=-t^m\) and $M_k(t)$ is invertible. Since the derived equivalence in \cite[Theorem 3.2]{KellerYang} preserves the equivariant Euler pairing, by the discussion above, the corresponding matrices satisfy \eqref{eq:mutation-congruence}.\end{proof}

\subsection{A brief summary of perestroikas for the graded Euler form}\label{sec:euler-summary} It might be helpful to summarize the effect on $\Eul_{(Q,W,d)}(t)$ that certain operations on a graded QP have, as follows from \Cref{prop:euler-gauge}, \Cref{lem:trivial-cancel}, \Cref{lem:right-equivalence} and \Cref{prop:mutation-euler}:

\[
\boxed{
\begin{array}{c}
\text{right-equivalence}\\
(Q,W)\rightsquigarrow (Q,W')
\end{array}
\quad\stackrel\Longrightarrow\quad
\quad\Eul_{(Q,W',d)}(t)= \Eul_{(Q,W,d)}(t)
}
\]

\[
\boxed{
\begin{array}{c}
\text{homogeneous stabilization}\\
(Q,W)\rightsquigarrow (Q',W')
\end{array}
\quad\stackrel\Longrightarrow\quad
\quad\Eul_{(Q',W',d)}(t)= \Eul_{(Q,W,d)}(t)
}
\]

\[
\boxed{
\begin{array}{c}
\text{vertex gauging}\\
d\rightsquigarrow d+\delta h
\end{array}
\quad\Longrightarrow\quad
\Eul_{(Q,W,d^h)}(t)= D_h(t)\Eul_{(Q,W,d)}(t)D_h(t)^{-1}
}
\]

\[
\boxed{
\begin{array}{c}
\text{graded QP mutation}\\
(Q,W,d)\rightsquigarrow(Q',W',d')
\end{array}
\quad\Longrightarrow\quad
\Eul_{(Q',W',d')}(t)= M(t)^\dagger \Eul_{(Q,W,d)}(t)M(t)
}
\]

\noindent Therefore, under any of these four operations above, the isomorphism type of the $\Z[t,t^{-1}]$-module $\coker\Eul_{(Q,W,d)}(t)$ remains invariant, see \Cref{rmk:linear_algebra}.

\begin{remark}\label{rmk:linear_algebra} (A commutative algebra reminder) Consider a commutative ring $R$ and a linear map $f:R^n\lr R^m$. Given any $A\in\GL_n(R)$ and $B\in\GL_m(R)$, then the $R$-modules
$$\coker_R(f)\cong \coker_R(A\circ f\circ B)$$
are isomorphic, as $R$-modules. By taking $R=\Z[t,t^{-1}]$ we deduce that both congruence and conjugation each preserves the isomorphism type of the cokernel module $\coker\Eul_{(Q,W,d)}(t)$. For $n=m$, the determinant of $f$ is recovered from the cokernel, up to units of $R$, as the generator of its $0$th Fitting ideal. Since the units of the ring of Laurent polynomials $\Z[t,t^{-1}]$ are $\pm t^{m}$, $m\in\Z$, the determinant $\det\Eul_{(Q,W,d)}(t)$ is recovered from $\coker\Eul_{(Q,W,d)}(t)$ up to a product with $\pm t^{m}$, $m\in\Z$.\hfill$\Box$
\end{remark}

\subsection{Canonically graded quivers}\label{ssec:canonically_gradedQPs} The construction in this section leads to defining the following class of quivers:

\begin{definition}[Canonically graded quivers]\label{def:canonically_graded} A quiver $Q$ is said to be canonically graded if:
\begin{enumerate}
    \item There exists a unique non-degenerate reduced QP $(Q,W)$, up to right-equivalence.

    \item For the unique non-degenerate potential $W$, we have $|\Gclass[Q,W]|=1$.\hfill$\Box$
\end{enumerate}
\end{definition}

\begin{corollary}\label{cor:canonically_graded_invariant}
Let $Q$ be a canonically graded quiver, with its unique non-degenerate reduced QP $(Q,W)$ and dilation $\delta\in \Gclass[Q,W]$. Then:

\begin{enumerate}
    \item The isomorphism type of the $\Z[t,t^{-1}]$-module
    $$\coker \Eul_{(Q,W,\delta)}(t)$$
    is an invariant of the quiver mutation class $[Q]$.\\

    \item The determinant
$$\det\Eul_{(Q,W,\delta)}(t)\in\Z[t,t^{-1}]$$
of the graded Euler pairing is an invariant of the quiver mutation class $[Q]$.
\end{enumerate}
\end{corollary}

\begin{proof}
Let $Q'\in[Q]$ be a quiver mutation equivalent to $Q$. Since QP mutation is invertible and preserves non-degeneracy, $Q$ having a unique non-degenerate potential, up to right equivalence, implies that $Q'$ has a unique non-degenerate potential $W'$, up to right-equivalence. By \Cref{thm:G1-mutation} we obtain that $|\Gclass[Q',W']|=|\Gclass[Q,W]|=1$, and so $\Gclass[Q',W']$ has a unique element $\delta'$. Thus we can name $\det\Eul_{(Q',W',\delta')}(t)$ from $Q'$ intrinsically. \Cref{prop:mutation-euler} then implies the required statements.
\end{proof}

\subsection{The example of tree quivers}\label{ssec:tree_quivers} Let $(Q,W)$ be a reduced QP and $\Gclass[Q,W]$ its set of dilation classes. By the results of \Cref{sec:euler-dilation} and \Cref{sec:euler-mutation}, the set of polynomials
$$\Pclass(Q,W):=\{\det\Eul_\delta(t)\in\Z[t,t^{-1}]:\mbox{ }\delta\in\Gclass[Q,W]\}$$
is an invariant of the QP-mutation class of $(Q,W)$. In particular:\\

\begin{enumerate}
    \item If $|\Gclass[Q,W]|=1$, i.e.~if there exists a unique dilation class, then $\Pclass(Q,W)$ is a unique Laurent polynomial. Thus we obtain in this case a Laurent polynomial as an invariant of a QP-mutation class.\\

    \item If $|\Gclass[Q,W]|=1$ and $W$ is the unique non-degenerate potential for $Q$, up to right equivalence, as in \Cref{def:canonically_graded}, then we obtain in this case a Laurent polynomial as an invariant of a quiver mutation class, cf.~\Cref{cor:canonically_graded_invariant}. Namely, for each quiver in the mutation class we can select the right-equivalence class of the unique non-degenerate potential and consider its unique dilation. From there $\Pclass(Q,W)$ provides the Laurent polynomial invariant of the quiver mutation class.\\
\end{enumerate}

\noindent There is a first example where $(2)$ above holds: for any quiver $Q$ whose underlying graph is a tree, cf.~\cite[Def.~1.6]{FominNeville}. That is, for any quiver obtained by considering a tree and simply orienting its edges. Though not logically needed for our main result, we briefly study this case of tree quivers for completeness. For a tree quiver, since there are no cycles, any potential is automatically zero, and the arrow-gradings are unique up to vertex gauge. More precisely:

\begin{lemma}[$\exists!$ reduced graded QP for tree quivers]\label{lem:unique-tree-qp}
Let $Q$ be a tree quiver. Then:
\begin{enumerate}[label=\textup{(\roman*)}]
\item There exists a unique potential $W=0$ on $Q$ and the corresponding QP is rigid, and hence non-degenerate.\\

\item Every integral arrow grading $d\colon Q_1\to\mathbb Z$ is compatible with the requirement that $W$ is homogeneous of degree $1$, and all such gradings are equivalent under vertex gauge. Thus the unique gauge class is represented by the zero grading $d(a)=0$ for all arrows $a\in Q_1$. 
\end{enumerate}
\end{lemma}

\begin{proof}
For $(i)$, since the underlying unoriented graph is a tree, $Q$ has no unoriented cycles and thus $W=0$ is the only possible potential. As there are therefore no nonzero cyclic deformations, the QP $(Q,0)$ is rigid and, by \cite[Corollary~8.2]{DWZ}, non-degenerate.\\

\noindent For $(ii)$, since the zero potential is homogeneous of every degree, every arrow-grading $d\colon Q_1\to\mathbb Z$ is compatible. Recall that a vertex gauge is a function $h\colon Q_0\to\mathbb Z$ acting on an arrow-grading $d$ by
\[
 d^{h}(a)=d(a)+h(t(a))-h(s(a)).
\]
Thus the quotient of the $\Z$-span of the arrows by vertex gauging is canonically $H^1(|Q|;\mathbb Z)$. Since $|Q|$ is a tree, $H^1(|Q|;\mathbb Z)=0$.  Thus all arrow gradings are gauge-equivalent, and the zero grading is a representative.  Homogeneous trivial QP summands disappear under graded QP-reduction, and so they do not alter the reduced stable class.
\end{proof}

\begin{remark}
The conclusion in \Cref{lem:unique-tree-qp}.(ii) uses that $Q$ is a tree quiver, not just an acyclic quiver. For instance, the quiver $Q$ with three vertices $\{1,2,3\}$ and arrows $1\to2$, $2\to3$ and $1\to3$ is acyclic but $H^1(|Q|;\Z)\cong\Z$ does not vanish. Thus the argument fails and, in fact, there is not a unique dilation class in this case. Similarly, \Cref{lem:matching} below only holds for tree quivers, but not for general acyclic quivers.\hfill$\Box$
\end{remark}

\noindent By \Cref{prop:euler-formula}, and using the zero potential and grading as in \Cref{lem:unique-tree-qp}, the graded Euler pairing $\cE_Q(t):= \cE_{(Q,0,0)}(t)$ in terms of the simple modules of the differential bigraded Ginzburg algebra reads combinatorially as:

\begin{equation}\label{eq:graded-euler-entry}
 \boxed{
 \cE_Q(t)_{ij}
 =(1-t)\delta_{ij}
 -\#\{a:j\to i\}
 +t\,\#\{a:i\to j\}.}
\end{equation}
Equivalently, every arrow $i\to j$ contributes $t$ in position $(i,j)$ and
$-1$ in position $(j,i)$, while every diagonal entry is $1-t$. From \eqref{eq:graded-euler-entry} there is a matching expansion for the determinant in certain cases, as follows:

\begin{lemma}[Matching expansion]\label{lem:matching}
Let $T$ be a tree quiver with $N$ vertices and no multiple arrows, and let $m_k(T)$ be the number of matchings of $T$ with exactly $k$ edges.  Then
\begin{equation}\label{eq:matching-formula}
 \boxed{
 \det\cE_T(t)=\sum_{k\ge 0}m_k(T)t^k(1-t)^{N-2k}.
 }
\end{equation}
In particular, the determinant is independent of the orientation of the tree.
\end{lemma}

\begin{proof}
First, expand the determinant as a sum over permutations. If a nonzero permutation term contained a cycle of length at least three, the corresponding nonzero off-diagonal matrix entries would exhibit an unoriented cycle in $T$. This is impossible because $T$ is a tree and thus every contributing permutation is a product of disjoint transpositions along edges and fixed points. In consequence, its transposed edges form a matching. For each transposed edge $\{i,j\}$, the permutation sign contributes $-1$, whereas the two matrix entries are $t$ and $-1$ (in one order or the other), and hence $\cE_{ij}(t)\cE_{ji}(t)=t(-1)=-t$.  Their combined contribution is therefore $t$. Since each fixed point contributes $1-t$, a matching of size $k$ thus contributes $t^k(1-t)^{N-2k}$. By summing over all matchings we obtain \eqref{eq:matching-formula}.
\end{proof}

\subsubsection{Relation to \cite{FominNeville}} In the article \cite{FominNeville}, the authors use cyclically ordered quivers to construct polynomials that are invariant under proper mutation of a cyclically ordered quiver. The case of a tree quiver $Q$ is discussed in detail in \cite[Section 10]{FominNeville}: by \cite[Prop.~2.6\&Cor.~4.11]{FominNeville} their polynomial $\FNDelta_Q(t)$ is independent of the choice of cyclic ordering. Let us now compare $\det\cE_Q(t)$, as introduced in \Cref{def:euler}, to $\FNDelta_Q(t)$, from \cite{FominNeville}, as follows.\\

\noindent Given a tree quiver $Q$, let us consider the matrix $B_Q$ with entries
\[
 b_{ij}:=\#\{i\to j\}-\#\{j\to i\},
\]
and let $U_Q$ be the unipotent companion, cf.~\cite[Def.~4.2]{FominNeville}, given by the unique upper-triangular matrix with diagonal entries $1$ satisfying
\begin{equation}\label{eq:FN-companion}
 -B_Q=U_Q-U_Q^{\mathsf T}.
\end{equation}
By construction, see \cite[Def.~8.1]{FominNeville}, we have $\FNDelta_Q(t)=\det(tU_Q-U_Q^{\mathsf T})$. The comparison then reads:

\begin{lemma}[Comparison to \cite{FominNeville}]\label{lem:FN-comparison}
Let $Q$ be a tree quiver on $N$ vertices, and choose an ordering of its vertices so that every arrow points from a smaller vertex to a larger one. Then
\begin{equation}\label{eq:E-U-trees}
 \cE_{(Q,0,0)}(t)=U_Q^{\mathsf T}-tU_Q.
\end{equation}
In consequence, we obtain the equality of polynomials
\begin{equation}\label{eq:sign-comparison}
 \boxed{\det\cE_{(Q,0,0)}(t)=(-1)^N\FNDelta_Q(t).}
\end{equation}
\end{lemma}

\begin{proof}
For two vertices $i,j\in Q_0$ with $i<j$, the ordering implies that there are no arrows $j\to i$.  By \eqref{eq:FN-companion},
\[
 (U_Q)_{ij}=-b_{ij}=-\#\{i\to j\},
 \qquad (U_Q)_{ji}=0.
\]
Thus the diagonal entries of $U_Q^{\mathsf T}-tU_Q$ are $1-t$.  For every arrow $i\to j$, the $(i,j)$-entry is $-t(U_Q)_{ij}=t$ whereas the $(j,i)$-entry is
$(U_Q^{\mathsf T})_{ji}=(U_Q)_{ij}=-1$. This yields precisely \eqref{eq:graded-euler-entry} and it implies \eqref{eq:E-U-trees}. For the determinants, note that \eqref{eq:E-U-trees} gives
\[
 \FNDelta_Q(t)
 =\det(tU_Q-U_Q^{\mathsf T})
 =\det(-\cE_{(Q,0,0)}(t))
 =(-1)^N\det\cE_{(Q,0,0)}(t),
\]
which is equivalent to \eqref{eq:sign-comparison}.
\end{proof}

\begin{remark}
Note that all orientations of a fixed tree quiver are mutation-equivalent. In addition, \cite[Theorem 1.1]{Neville} states that mutation-acyclic quivers are totally proper, thus promoting the polynomial $\FNDelta_Q(t)$ to an invariant of any quiver mutation class containing a tree. From this viewpoint, Lemma~\ref{lem:FN-comparison} identifies that invariant with a normalization of the determinant of the graded Euler pairing in the category of finite-dimensional continuous differential bigraded modules over the differential bigraded Ginzburg algebra of $(Q,0)$.\hfill$\Box$
\end{remark}

\noindent We conclude this subsection with a few computations to illustrate the results above.


\subsubsection{The linear tree \texorpdfstring{$A_{2g}$}{A2g}}

Let $g\in\N$ and consider the linearly oriented tree on $2g$ vertices:
\[
\begin{tikzpicture}[baseline=(current bounding box.center),node distance=12mm]
  \node[qvertex] (v1) {$1$};
  \node[qvertex,right=of v1] (v2) {$2$};
  \node[qvertex,right=of v2] (v3) {$3$};
  \node[right=10mm of v3] (dots) {$\cdots$};
  \node[qvertex,right=10mm of dots] (vn1) {$2g-1$};
  \node[qvertex,right=of vn1] (vn) {$2g$};
  \draw[qarrow] (v1)--(v2);
  \draw[qarrow] (v2)--(v3);
  \draw[qarrow] (v3)--(dots);
  \draw[qarrow] (dots)--(vn1);
  \draw[qarrow] (vn1)--(vn);
\end{tikzpicture}
\]
By \eqref{eq:euler-formula} or \eqref{eq:graded-euler-entry}, the graded Euler matrix from \Cref{def:euler} is the tridiagonal matrix
\begin{equation}\label{eq:A-matrix}
 \cE_{A_{2g}}(t)=
 \begin{pmatrix}
 1-t&t&0&\cdots&0\\
 -1&1-t&t&\ddots&\vdots\\
 0&-1&1-t&\ddots&0\\
 \vdots&\ddots&\ddots&\ddots&t\\
 0&\cdots&0&-1&1-t
 \end{pmatrix}.
\end{equation}
Let $D_m(t)$ denote the determinant of the analogous $m\times m$ matrix.  The product of the two off-diagonal entries adjacent to the last diagonal entry is $t(-1)=-t$, so expansion along the last row or column gives
\[
 D_m(t)=(1-t)D_{m-1}(t)+tD_{m-2}(t),
 \qquad D_0(t)=1,\quad D_1(t)=1-t.
\]
Resolving this recursion, e.g.~via induction, yields
\[
 D_m(t)=\sum_{j=0}^{m}(-t)^j.
\]
and thus for $m=2g$ we obtain
\begin{equation}\label{eq:A-det}
 \boxed{
 \det\cE_{A_{2g}}(t)
 =1-t+t^2-\cdots+t^{2g}
 =\frac{1+t^{2g+1}}{1+t}.
 }
\end{equation}

Note that the irreducible plane curve singularity $x^2+y^{2g+1}=0$ has the $(2,2g+1)$-torus knot as its knot, obtained by intersecting with the 3-sphere boundary of a small enough Milnor ball. The Alexander polynomial $\Delta_{T(2,2g+1)}(t)$ of such a torus knot is
\[
 \Delta_{T(2,2g+1)}(t)
 =\frac{(1-t)(1-t^{4g+2})}{(1-t^2)(1-t^{2g+1})}
 =\frac{1+t^{2g+1}}{1+t}.
\]
Thus
\[
 \boxed{
 \det\cE_{A_{2g}}(t)=\Delta_{T(2,2g+1)}(t).
 }
\]
\noindent In subsequent sections, we will be generalizing such equality for any irreducible plane curve singularity that admits a real Morsification with a malleable divide, where the quiver is no longer necessarily a tree quiver.

\subsubsection{\texorpdfstring{$E_6$}{E6} tree}

Choose the following orientation and labeling:
\[
\begin{tikzpicture}[baseline=(current bounding box.center),node distance=11mm]
  \node[qvertex] (v1) {$1$};
  \node[qvertex,right=of v1] (v2) {$2$};
  \node[qvertex,right=of v2] (v3) {$3$};
  \node[qvertex,right=of v3] (v4) {$4$};
  \node[qvertex,right=of v4] (v5) {$5$};
  \node[qvertex,below=10mm of v3] (v6) {$6$};
  \draw[qarrow] (v1)--(v2);
  \draw[qarrow] (v2)--(v3);
  \draw[qarrow] (v3)--(v4);
  \draw[qarrow] (v4)--(v5);
  \draw[qarrow] (v3)--(v6);
\end{tikzpicture}
\]
By \eqref{eq:graded-euler-entry} the graded Euler matrix is
\begin{equation}\label{eq:E6-matrix}
 \cE_{E_6}(t)=
 \begin{pmatrix}
 1-t&t&0&0&0&0\\
 -1&1-t&t&0&0&0\\
 0&-1&1-t&t&0&t\\
 0&0&-1&1-t&t&0\\
 0&0&0&-1&1-t&0\\
 0&0&-1&0&0&1-t
 \end{pmatrix}.
\end{equation}
\noindent For the purpose of illustration, let us compute its determinant via matchings as in \Cref{lem:matching}. In this case the non-zero matching numbers are $(m_0,m_1,m_2,m_3)=(1,5,5,1)$ and \Cref{lem:matching} gives
\begin{align*}
 \det\cE_{E_6}(t)
 &=(1-t)^6+5t(1-t)^4+5t^2(1-t)^2+t^3\\
 &=t^6-t^5+t^3-t+1.
\end{align*}
Therefore
\begin{equation}\label{eq:E6-det}
 \boxed{
 \det\cE_{E_6}(t)=t^6-t^5+t^3-t+1
 =(t^2-t+1)(t^4-t^2+1).
 }
\end{equation}
As in the $A_{2g}$ case above, the irreducible plane curve singularity $x^3+y^{4}=0$ has the $(3,4)$-torus knot as the algebraic knot of the singularity. Its Alexander polynomial $\Delta_{T(3,4)}(t)$ is
\[
 \Delta_{T(3,4)}(t)
 =\frac{(1-t)(1-t^{12})}{(1-t^3)(1-t^4)}
 =t^6-t^5+t^3-t+1.
\]
Hence
\[
 \boxed{\det\cE_{E_6}(t)=\Delta_{T(3,4)}(t).}
\]

\subsubsection{A few more instances} Consider the tree quivers in Figures \ref{fig:E8}, \ref{fig:tree1} and \ref{fig:tree2}. The determinants $\det\cE_T(t)$ of their graded Euler pairings $\cE_T(t)$ are computed in \Cref{table:tree_quivers}.

\begin{center}
\small
\renewcommand{\arraystretch}{1.35}
\begin{table}[htbp]
    \centering
\caption{The determinants $\det\cE_T(t)$ of the graded Euler pairings for a few tree quivers.}
\label{table:tree_quivers}
\begin{tabularx}{\textwidth}{>{\raggedright\arraybackslash}p{1.55cm}
>{\raggedright\arraybackslash}X
>{\raggedright\arraybackslash}p{3.65cm}
>{\raggedright\arraybackslash}p{3.05cm}}
\toprule
Tree & $\det\cE_T(t)$ & Alexander poly.~ of & Algebraic?\\
\midrule
$A_{2g}$
& $1-t+t^2-\cdots+t^{2g}$
& $T(2,2g+1)$
& Yes: $x^2=y^{2g+1}$\\
$E_6$
& $t^6-t^5+t^3-t+1$
& $T(3,4)$
& Yes: $x^3=y^4$\\
$E_8$
& $t^8-t^7+t^5-t^4+t^3-t+1$
& $T(3,5)$
& Yes: $x^3=y^5$\\
$T_{(2,2,3)}$
& $t^8-t^7+2t^5-3t^4+2t^3-t+1$
& Slalom knot $K_{(2,2,3)}$
& No\\
$T_{(4,4,1)}$
& $t^{10}-t^9+t^7-2t^6+3t^5-2t^4+t^3-t+1$
& Slalom knot $K_{(4,4,1)}$
& No\\
\bottomrule
\end{tabularx}
\end{table}
\end{center}

\begin{figure}
\centering
\[
\begin{tikzpicture}[baseline=(current bounding box.center),node distance=9.5mm]
  \node[qvertex] (v1) {$1$};
  \node[qvertex,right=of v1] (v2) {$2$};
  \node[qvertex,right=of v2] (v3) {$3$};
  \node[qvertex,right=of v3] (v4) {$4$};
  \node[qvertex,right=of v4] (v5) {$5$};
  \node[qvertex,right=of v5] (v6) {$6$};
  \node[qvertex,right=of v6] (v7) {$7$};
  \node[qvertex,below=10mm of v3] (v8) {$8$};
  \draw[qarrow] (v1)--(v2);
  \draw[qarrow] (v2)--(v3);
  \draw[qarrow] (v3)--(v4);
  \draw[qarrow] (v4)--(v5);
  \draw[qarrow] (v5)--(v6);
  \draw[qarrow] (v6)--(v7);
  \draw[qarrow] (v3)--(v8);
\end{tikzpicture}
\]
\caption{The $E_8$ tree with a choice of orientation. The associated $\det\Eul_Q(t)$ gives the Alexander polynomial of the $(3,5)$-torus knot, associated to the plane curve singularity $x^3+y^5=0$.}
\label{fig:E8}
\end{figure}
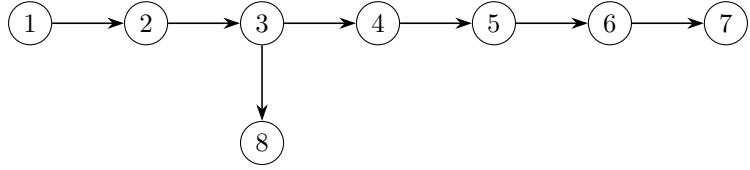

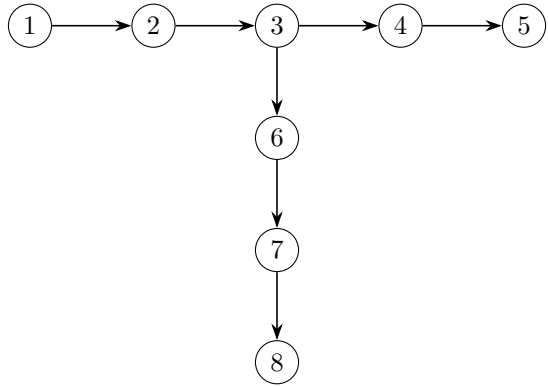
\begin{figure}
\centering
\[
\begin{tikzpicture}[baseline=(current bounding box.center),node distance=10.5mm]
  \node[qvertex] (v1) {$1$};
  \node[qvertex,right=of v1] (v2) {$2$};
  \node[qvertex,right=of v2] (v3) {$3$};
  \node[qvertex,right=of v3] (v4) {$4$};
  \node[qvertex,right=of v4] (v5) {$5$};
  \node[qvertex,below=9mm of v3] (v6) {$6$};
  \node[qvertex,below=9mm of v6] (v7) {$7$};
  \node[qvertex,below=9mm of v7] (v8) {$8$};
  \draw[qarrow] (v1)--(v2);
  \draw[qarrow] (v2)--(v3);
  \draw[qarrow] (v3)--(v4);
  \draw[qarrow] (v4)--(v5);
  \draw[qarrow] (v3)--(v6);
  \draw[qarrow] (v6)--(v7);
  \draw[qarrow] (v7)--(v8);
\end{tikzpicture}
\]
\caption{A tree quiver $T_{(2,2,3)}$, with a central trivalent vertex and three arms of length $(2,2,3)$. In this case $\det\cE_{T_{(2,2,3)}}(t)$ is not the Alexander polynomial of any algebraic knot, by the semi-group gap obstruction, cf.~\cite[Section~3.1]{BorodzikHom} or \cite{CDGZ}. It is the Alexander polynomial of the slalom knot $K_{(2,2,3)}$, cf.~\cite{ACampo}, which is a fibered hyperbolic knot.}
\label{fig:tree1}
\end{figure}

\begin{figure}
\centering
\[
\begin{tikzpicture}[baseline=(current bounding box.center),node distance=8.2mm]
  \node[qvertex] (v1) {$1$};
  \node[qvertex,right=of v1] (v2) {$2$};
  \node[qvertex,right=of v2] (v3) {$3$};
  \node[qvertex,right=of v3] (v4) {$4$};
  \node[qvertex,right=of v4] (v5) {$5$};
  \node[qvertex,right=of v5] (v6) {$6$};
  \node[qvertex,right=of v6] (v7) {$7$};
  \node[qvertex,right=of v7] (v8) {$8$};
  \node[qvertex,right=of v8] (v9) {$9$};
  \node[qvertex,below=10mm of v5] (v10) {$10$};
  \draw[qarrow] (v1)--(v2);
  \draw[qarrow] (v2)--(v3);
  \draw[qarrow] (v3)--(v4);
  \draw[qarrow] (v4)--(v5);
  \draw[qarrow] (v5)--(v6);
  \draw[qarrow] (v6)--(v7);
  \draw[qarrow] (v7)--(v8);
  \draw[qarrow] (v8)--(v9);
  \draw[qarrow] (v5)--(v10);
\end{tikzpicture}
\]
\caption{A tree quiver $T_{(4,4,1)}$, with a central trivalent vertex and three arms of length $(4,4,1)$. In this case $\det\cE_{T_{(4,4,1)}}(t)$ is also not the Alexander polynomial of any algebraic knot, by the same semi-group obstruction. It is the Alexander polynomial of the slalom knot $K_{(4,4,1)}$, a genus-five fibered hyperbolic knot.}
\label{fig:tree2}
\end{figure}
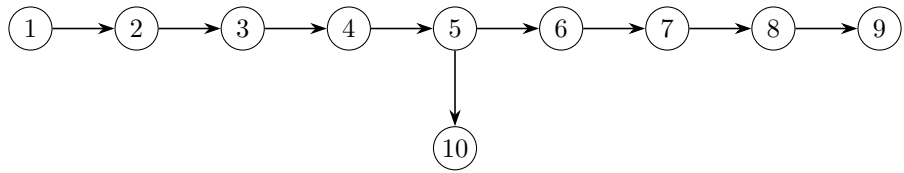

\section{Existence and uniqueness of the dilation of a plabic fence}\label{sec:dilation_plabic_fence}

Let us now focus on a particular class of quivers with potential $(Q,W)$. Namely, those quivers with potential $(Q(\bG),W(\bG))$ associated to a plabic fence $\bG$. We refer to \cite[Section 12]{FPST}, \cite[Def.~3.1 \& Sect.~3.2]{CasalsGao}, and \Cref{ssec:QP_plabicfence} below for the definitions and properties of such QPs $(Q(\bG),W(\bG))$ associated to a plabic fence $\bG$. An important aspect is that $(Q(\bG),W(\bG))$ is a non-degenerate QP, as proven in \cite[Prop.~3.7\&3.11]{CasalsGao}. The main result of this section reads as follows:

\begin{theorem}\label{thm:unique-fence-G1} Let $(Q(\bG),W(\bG))$ be the QP associated to a plabic fence $\bG$. Then
$$|\Gclass[Q(\bG),W(\bG)]|=1,$$
i.e.~there exists a unique dilation class for $(Q(\bG),W(\bG))$.
\end{theorem}

\Cref{thm:unique-fence-G1} is relevant in the study of quivers associated to plabic fences because the potential $W(\bG)$ is in fact the unique non-degenerate potential of $Q(\bG)$, up to right equivalence. Therefore, if a quiver $Q=Q(\bG)$ is associated to a plabic fence, we can canonically endow it with a unique non-degenerate potential, up to right-equivalence, and \Cref{thm:unique-fence-G1} implies that such QP has a unique dilation class. Therefore, if we denote this dilation class by $\delta(\bG)$, the polynomial $\det\Eul_{(Q(\bG),W(\bG),\delta(\bG))}(t)$ can be extracted from $Q(\bG)$, as $W(\bG)$ and $\delta(\bG)$ are determined by $Q(\bG)$, and $\Eul_{(Q(\bG),W(\bG),\delta(\bG))}(t)$ is invariant under the necessary equivalence relations, such as right-equivalences. In other words, for plabic fences $Q(\bG)$ is canonically graded, cf.~\Cref{def:canonically_graded}, and thus \Cref{cor:canonically_graded_invariant} applies. A consequence of \Cref{thm:unique-fence-G1} is thus:

\begin{corollary}\label{cor:unique_dilation_plabic_fence}
Let $Q(\bG)$ be the quiver of a plabic fence $\bG$, $W(\bG)$ its unique non-degenerate potential and $\delta(\bG)$ its unique dilation. Then the following holds:

\begin{enumerate}
\item The $\Z[t,t^{-1}]$-module
$$\coker\Eul_{(Q(\bG),W(\bG),\delta(\bG))}(t)$$
is an invariant of the quiver mutation class of $Q(\bG)$. In consequence, if $\bG$ and $\bG'$ are two plabic fences with mutation-equivalent quivers $[Q(\bG)]=[Q(\bG')]$, then
$$\coker\Eul_{(Q(\bG),W(\bG),\delta(\bG))}(t)\cong\coker\Eul_{(Q(\bG'),W(\bG'),\delta(\bG'))}(t)$$
are isomorphic $\Z[t,t^{-1}]$-modules.\\

    \item The determinant of the graded Euler pairing
$$\det\Eul_{(Q(\bG),W(\bG),\delta(\bG))}(t)\in\Z[t,t^{-1}]$$
is an invariant of the quiver mutation class of $Q(\bG)$. In particular, if $\bG$ and $\bG'$ are two plabic fences with mutation-equivalent quivers $[Q(\bG)]=[Q(\bG')]$, then
$$\det\Eul_{(Q(\bG),W(\bG),\delta(\bG))}(t)=\det\Eul_{(Q(\bG'),W(\bG'),\delta(\bG'))}(t).$$
\end{enumerate}
\end{corollary}

For completeness, a similar statement to \Cref{cor:unique_dilation_plabic_fence} is as follows:
\begin{corollary}\label{cor:unique_dilation_mutation_equivalent_plabic_fence}
Let $Q$ be a quiver. If $Q$ is mutation equivalent to the quiver of a plabic fence, then:
\begin{enumerate}
    \item[$(i)$] $Q$ admits a unique non-degenerate potential $W$, up to right equivalence.

    \item[$(ii)$] $(Q,W)$ admits a unique dilation, i.e.~$\Gclass[Q,W]$ has a unique element $\delta$.

    \item[$(iii)$] The isomorphism type of the $\Z[t,t^{-1}]$-module
    $$\coker\Eul_{(Q,W,\delta)}(t)\in\Z[t,t^{-1}]$$
    is an invariant of the quiver mutation class $[Q]$.\\

    \item[$(iv)$] The polynomial
    $$\det\Eul_{(Q,W,\delta)}(t)\in\Z[t,t^{-1}]$$
    is an invariant of the quiver mutation class $[Q]$.
\end{enumerate}
\end{corollary}

\noindent Subsections \ref{ssec:QP_plabicfence}, \ref{ssec:properties_plabic_QP}, \ref{ssec:dilation_plabic_fence} and \ref{ssec:proof_unique-fence-G1} will now be devoted to the proof of \Cref{thm:unique-fence-G1}, and thus \Cref{cor:unique_dilation_plabic_fence} and \Cref{cor:unique_dilation_mutation_equivalent_plabic_fence}, also introducing the necessary concepts and results. \Cref{ssec:explicit_plabic_gradedQP} explicitly describes an arrow-grading $d(\bG)$ such that $(Q(\bG),W(\bG),d(\bG))$ is a degree-one graded QP, and thus it represents the unique dilation in $\Gclass[Q(\bG),W(\bG)]$.

\subsection{Plabic fences and their non-degenerate QPs}\label{ssec:QP_plabicfence} Let us briefly introduce plabic fences and their non-degenerate QPs, following \cite[Section 3]{CasalsGao}, as they feature in \Cref{thm:unique-fence-G1} and \Cref{cor:unique_dilation_plabic_fence} and are key in our proof of the main result. First, the definition of a plabic fence:

\begin{definition}[Plabic fence]\label{def:plabic_fence}
A \emph{plabic fence} is an embedded planar bicolored graph $\bG\sse\R^2$ such that:
\begin{itemize}\setlength\itemsep{0.5em}
    \item[(i)] The vertices of $\bG\sse\R^2$ belong to the standard integral lattice $\Z^2\sse\R^2$, and these vertices are colored in either black or white.
    \item[(ii)] The edges of $\bG\sse\R^2$ belong to the standard integral grid $(\Z\times\R)\cup (\R\times\Z)\sse\R^2$. The edges that are contained in $\Z\times\R$ are said to be \emph{vertical} edges, and those contained in $\R\times\Z$ are said to be \emph{horizontal} edges. 
    \item[(iii)] A maximal connected union of horizontal edges is said to be a \emph{horizontal line}. All horizontal lines must start, on the left, at univalent white vertices with the same $x$-coordinate and must end, on the right, at univalent white vertices with the same $x$-coordinate.
    \item[(iv)] Each vertical edge must end at trivalent vertices of opposite colors, with white on top and black on bottom, and the end points of a vertical edge must be contained in the interior of a horizontal line. In addition, no two vertical edges are contained in the same (vertical) line.\hfill$\Box$
\end{itemize}
\end{definition}

\begin{center}
	\begin{figure}
		\centering
		\includegraphics[scale=1]{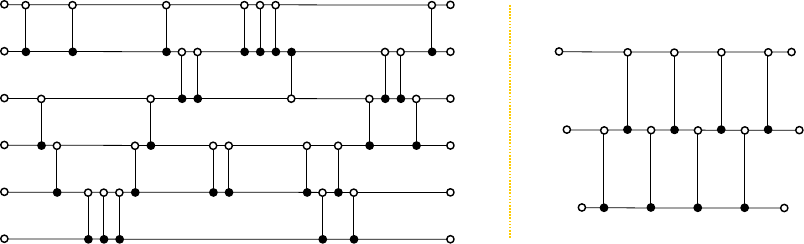}
		\caption{Two instances of plabic fences: the fence on the left encodes the 6-stranded braid word $\beta=\sigma_5\sigma_3\sigma_5\sigma_1^3\sigma_2\sigma_3\sigma_5\sigma_4^2\sigma_2^2\sigma_5^3\sigma_4(\sigma_2\sigma_1)^2\sigma_3\sigma_4^2\sigma_3\sigma_5$, whereas the plabic fence on the right encodes the 3-stranded braid word $\beta=(\sigma_1\sigma_2)^4$.}\label{fig:PlabicFence_Example}
	\end{figure}
\end{center}

\noindent Two plabic fences $\bG$ are drawn in Figure \ref{fig:PlabicFence_Example}, and described using positive braid words, cf.~the paragraph before \cite[Section 3.1]{CasalsGao} or \Cref{ssec:polynomial_plabic_fence} below.

\begin{remark}\label{rmk:double_plabic_fence} As noted in \cite{CasalsGao}, the definition of plabic fences in \cite{FPST,CasalsWeng22} is more general but, for our purposes, it is without loss of generality that we can work with Definition \ref{def:plabic_fence} for now. The more general definitions \cite[Def.~12.1]{FPST} and \cite[Section 2.5]{CasalsWeng22}  allow for vertical edges with black on top and white on bottom. For clarity, we refer to this latter, more general, notion of plabic fences as {\it double} plabic fences.\footnote{That is: we refer to the plabic fences in \cite[Def.~12.1]{FPST} and \cite[Section 2.5]{CasalsWeng22} as double plabic fences.}\hfill$\Box$
\end{remark}

There are different descriptions of the quivers $Q(\bG)$ associated to a plabic fence $\bG$, see e.g.~\cite{CasalsGao,CasalsWeng22,FPST}, but they are all equivalent. We follow \cite{CasalsGao} in the upcoming \Cref{def:quiver_plabicfence}. For that, suppose that $e\sse\bG$ is a vertical edge at level $k$, then we denote by $F_e$ the face of $\bG$ that has $e$ as its right vertical edge. (This face $F_e$ is unique or it does not exist.) By definition, a black pente-row (resp.~white pente-row) is a consecutive collection of black (resp.~white) vertices in the same horizontal line of $\bG$ such that:

\begin{itemize}
    \item[-] There must be two white (resp.~black) vertices bounding it: one white vertex at its left, and one white vertex at its right.

    \item[-] Each connected component of $\R^2\setminus\bG$ whose closure contains any segment between two vertices of the black vertices above or between a black vertex and one of the two white vertices above must be a bounded face.
\end{itemize}

The total number of black (resp.~white) vertices in a black (resp.~white) pente-row is said to be its length. See Figure \ref{fig:PlabicFence_Row1} for a length four black pente-row and a length four white pente-row, where we have marked the connected components of $\R^2\setminus\bG$ that must be faces with a blue dot. The rightmost face in a black (resp.~white) pente-row, which has a left corner at the rightmost black (resp.~white) vertex, is said to be its right face. With that, the QP $(Q(\bG),W(\bG))$ associated to a plabic fence $\bG$ is defined as follows:

\begin{center}
	\begin{figure}
		\centering
		\includegraphics[scale=0.9]{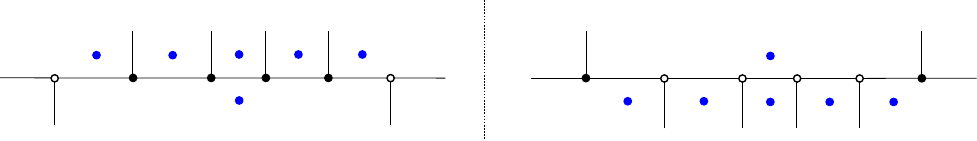}
		\caption{(Left) A black pente-row. (Right) A white pente-row.}\label{fig:PlabicFence_Row1}
	\end{figure}
\end{center}

\begin{definition}[The QP $(Q(\bG),W(\bG))$ of a plabic fence $\bG$]\label{def:quiver_plabicfence} Let $\bG$ be the plabic fence with $n$ horizontal lines, then its associated QP $(Q(\bG),W(\bG))$ is defined as follows. The quiver $Q(\bG)$ has vertex set the set of faces of $\bG$. The arrow set of $Q(\bG)$ is inductively described as follows, scanning $\bG$ left to right:

\begin{itemize}
    \item[(i)] If $\bG$ is the empty plabic fence, then the arrow set of $Q(\bG)$ is empty.\\

    \item[(ii)] Choose a vertical edge $e\sse\bG$ at level $k$ and assume that the arrow set of $Q(\bG_{<e})$ is $A_{<e}$, where $\bG_{<e}$ is the plabic subfence of $\bG$ consisting of those vertical edges to the left of $e$. If $F_e$ does not exist, the arrow set of $Q(\bG_{\leq e})$, where $\bG_{\leq e}=\bG_{<e}\cup\{e\}$, is defined to be $A_{<e}$.\\
    
    \noindent If $F_e$ exists, the arrow set of $Q(\bG_{\leq e})$ is defined to be $A_{<e}$ union the following possible arrows:\\

    \begin{itemize}
        \item[(a)] Let $d$ be the left vertical edge of $F_e$. If $F_d$ exists, then we add an arrow $[de]$ from $F_d$ to $F_e$.\\

       \item[(b)] Let $d^\uparrow$ be the first vertical edge in $\bG$ at level $(k+1)$ to the right of $d$. If $d^\uparrow$ and $F_{d^\uparrow}$ exist, then we add an arrow $[ed^\uparrow]$ from $F_e$ to $F_{d^\uparrow}$. See Figure \ref{fig:PlabicFence_Quiver} (left) for such an arrow, marked with a pink $(Z)$ pattern.\\

       \item[(c)] Let $d^\downarrow$ be the first vertical edge in $\bG$ at level $(k-1)$ to the right of $d$. If $d^\downarrow$ and $F_{d^\downarrow}$ exist, then we add an arrow $[ed^\downarrow]$ from $F_e$ to $F_{d^\downarrow}$. See Figure \ref{fig:PlabicFence_Quiver} (left) for such an arrow, marked with a red $(S)$ pattern.\\
    \end{itemize}

 \noindent If the hypotheses in these cases are not met, no arrows are added: e.g.~if $F_e$ or $F_d$ do not exist in case (a), we do not add any arrows at that stage.\\
\end{itemize}

\begin{center}
	\begin{figure}
		\centering
		\includegraphics[scale=0.9]{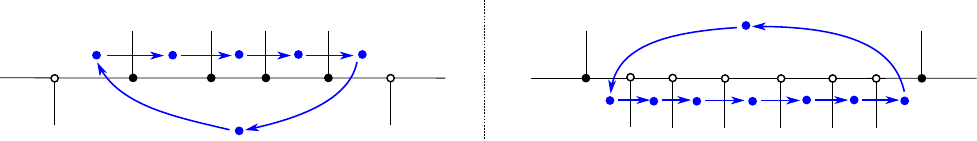}
		\caption{(Left) A black pente-row with its associated cycle in the quiver $Q(\bG)$ around it. (Right) A white pente-row with its corresponding cycle around it.}\label{fig:PlabicFence_Row1_Example}
	\end{figure}
\end{center}

\begin{center}
	\begin{figure}
		\centering
		\includegraphics[scale=0.9]{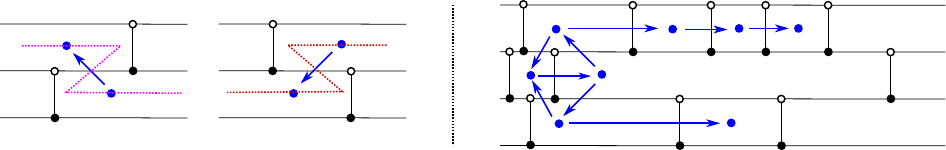}
		\caption{(Left) The two local patterns for arrows in the quiver $Q(\bG)$ that are not horizontal. (Right) An example of a quiver $Q(\bG)$ associated to a plabic fence $\bG$.}\label{fig:PlabicFence_Quiver}
	\end{figure}
\end{center}

The potential $W(\bG)$ is similarly described inductively. If $\bG$ is the empty plabic fence, then $W(\bG)=0$. For each vertical edge $e$ added to the right of $\bG_{<e}$ as above, the potential is defined as
$$W(\bG_{\leq e})=W(\bG_{<e})+P_{black}(e)-P_{white}(e),$$
where the monomials in $P_{black}(e)$ and $P_{white}(e)$ are described as follows. By definition, $P_{black}(e)$, resp.~$P_{white}(e)$, is the (cyclic) monomial in the arrows of $Q(\bG)$ encoding the planar cycle in $Q(\bG)$ associated to the unique black (resp.~white) pente-row with right face equal $F_e$, if such pente-row exists; otherwise it is zero. Note that monomials in $P_{white}(e)$ have an overall minus sign in front due to the counter-clockwise orientation of their cycles.\hfill$\Box$
\end{definition}

Recall that, following \cite[Section 8.4]{kontsevich_soibelman_2008}, class $\mathcal{P}$ is the class of quivers generated by the one vertex quiver by triangular extensions and mutations. A first useful property of $(Q(\bG),W(\bG))$ is the following:

\begin{lemma}[Uniqueness of a non-degenerate QP for plabic fences]\label{prop:P}
Let $\bG$ be a plabic fence, and $(Q(\bG),W(\bG))$ its associated QP. Then
\begin{enumerate}[label=$(\roman*)$]
    \item $Q(\bG)$ belongs to class $\mathcal P$.

    \item $(Q(\bG),W(\bG))$ is the unique non-degenerate potential for $Q(\bG)$, up to right-equivalence. That is, every non-degenerate potential on $Q(\bG)$ is right-equivalent to $W(\bG)$.
\end{enumerate}

\end{lemma}

\begin{proof}
For Part (i), let us scan the vertical edges of $\bG$ from left to right and argue inductively.  If a scanned edge $e$ creates no new bounded face, the quiver is unchanged: $Q(\bG_{<e})=Q(\bG_{\leq e})$.  Otherwise, if adjoining the edge $e$ creates a new face $F_e$, and thus a new vertex for $Q(\bG_{\leq e})$, we apply \cite[Lemma 3.8]{CasalsGao} to $Q(\bG_{\leq e})$, which provides a sequence of mutations at vertices of $Q(\bG_{\leq e})$ different from $F_e$ after which the vertex associated to $F_e$ is a source. By induction on the number of edges, the quiver $Q(\bG_{<e})$ associated to the fence before adding the vertex $F_e$ belongs to $\mathcal P$, and so does its mutation by the prescribed sequence provided by \cite[Lemma 3.8]{CasalsGao}. After performing the same mutations on the full quiver $Q(\bG_{\leq e})$, the resulting quiver is obtained by adjoining the one-vertex quiver $\{F_e\}$ as a source, which is a triangular extension. Hence the mutated full quiver belongs to $\mathcal P$, and mutation invariance of $\mathcal P$ implies that the original full quiver $Q(\bG_{\leq e})$ also belongs to $\mathcal P$. Finally, note that the first time a face is created, the quiver is the one-vertex quiver, which is the base object of $\mathcal P$.\\

For Part (ii), first use that Part (i) implies $Q(\bG)\in\mathcal P$. Then \cite[Theorem~4.6]{Ladkani} implies that every quiver in $\mathcal P$ has a unique non-degenerate potential, up to right-equivalence.  The potential $W(\bG)$ is non-degenerate by \cite[Proposition~3.11]{CasalsGao}, and therefore it represents the unique right-equivalence class, as required.
\end{proof}

\begin{remark} \Cref{prop:P} explains use of the definite article {\it the} in ``The QP of a plabic fence'' in the text and \Cref{def:quiver_plabicfence}. We will always work with such a non-degenerate QP, thus our terminology.\hfill$\Box$
\end{remark}

\subsection{Two properties of plabic fence QPs}\label{ssec:properties_plabic_QP} Let $\bG$ be a plabic fence and $(Q(\bG),W(\bG))$ its QP, as in \Cref{def:quiver_plabicfence}. The first property we will use is that such a quiver $Q(\bG)$ does not have arrows with multiplicity greater than one, i.e.~there are no parallel arrows in $Q(\bG)$:

\begin{lemma}[No parallel arrows in $Q(\bG)$]\label{lem:no-parallel}
Let $\bG$ be a plabic fence and $Q(\bG)$ its quiver. Then $Q(\bG)$ has at most one arrow between every pair of vertices.
\end{lemma}

\begin{proof}
We follow \Cref{def:quiver_plabicfence}, scanning $\bG$, and thus building $Q(\bG)$, left-to-right. When a new (bounded) face $F_e$ of $\bG$ (i.e.~a vertex of $Q(\bG)$) is created at a vertical edge $e$ of level $k$, \Cref{def:quiver_plabicfence}.(ii) allows arrows only between the vertex associated to $F_e$ and the three possible previous faces $F_d,F_{d^\uparrow},F_{d^\downarrow}$ precisely to its left. These lie, respectively, at levels $k$, $k+1$, and $k-1$, and hence are distinct whenever they exist. Exactly one arrow is added in each such case. Thus the addition of these arrows does not create any parallel arrows for $Q(\bG_{\leq e})$. Any later scanning stages, to the right of $F_e$, add only arrows incident to the corresponding newly created vertex: i.e.~each new arrow being created as we scan left-to-right has exactly one end at the new rightmost vertex. Thus a second arrow between an already existing ordered pair of vertices is never created, and the result follows.
\end{proof}

The potential $W(\bG)$ can be described as
\[
W(\bG)=\sum_{\rho\in P_{\mathrm{black}}}p_\rho
-
\sum_{\rho\in P_{\mathrm{white}}}p_\rho,
\]
where $P_{\mathrm{white}}$ and $P_{\mathrm{black}}$ are the sets of white and black pente-rows, respectively, and $p_\rho$ is the monomial for the directed cycle around the corresponding pente-row $\rho$. The cycles of $Q(\bG)$ will have a particular role in studying the possible degree-one integral gradings of the QP $(Q(\bG),W(\bG))$, as follows. Consider the following lattice $C_1(\Gamma_\bG;\Z)$ and operator $\partial:C_1(\Gamma_\bG;\Z)\lr C_0(\Gamma_\bG;\Z)$, where $\Gamma_\bG$ is understood as an oriented planar cellular embedding of the topological graph underlying $Q(\bG)$, with $0$-chains being spanned by the vertices and $1$-chains by the edges:
$$C_0(\Gamma_\bG;\Z):=\bigoplus_{v\in(Q(\bG))_0}\Z v,\quad C_1(\Gamma_\bG;\Z):=\bigoplus_{a\in(Q(\bG))_1}\Z a,$$
$$\partial:C_1(\Gamma_\bG;\Z)\lr C_0(\Gamma_\bG;\Z),\quad \partial a:=t(a)-s(a),$$
where we have given every edge of $\Gamma_\bG$ the orientation of its corresponding quiver arrow to define the differential $\partial$. The integral cycle lattice is $Z_1(\Gamma_\bG;\Z)=\ker\partial$. For every pente-row $\rho$ in $\bG$, let $c_\rho\in Z_1(\Gamma_\bG;\Z)$ be the element corresponding to the monomial $p_\rho$. The following lemma is straightforward:

\begin{lemma}[A cycle basis]\label{lem:face-basis}
The collection of cycles $\{c_\rho\}_\rho$, where $\rho$ runs over all pente-rows, is a $\Z$-basis of the lattice $Z_1(\Gamma_\bG;\Z)$.
\end{lemma}

\subsection{The dilation of a plabic fence}\label{ssec:dilation_plabic_fence} Let $\bG$ be a plabic fence and $(Q(\bG),W(\bG))$ its quiver with potential, as in \Cref{def:quiver_plabicfence}. Let us first show that a degree-one integral grading for $(Q(\bG),W(\bG))$ exists, and thus $\Gclass[Q(\bG),W(\bG)]$ is always non-empty.

\begin{lemma}[Existence of a graded QP]\label{prop:existence}
Let $\bG$ be a plabic fence. Then there exists an integral arrow grading $d:(Q(\bG))_1\to\Z$ such that the potential $W(\bG)$ is homogeneous of degree one.
\end{lemma}
\begin{proof}
By \Cref{lem:face-basis}, the quiver cycles $c_\rho$ associated to pente-rows form a basis of the lattice $Z_1(\Gamma_\bG;\Z)$. Consider the morphism
\begin{equation}\label{eq:grading_one_penterows}
\tilde{d}:Z_1(\Gamma_\bG;\Z)\longrightarrow\Z,
\qquad
\tilde{d}(c_\rho)=1
\end{equation}
for every pente-row $\rho$. The short exact sequence
\[
0\longrightarrow Z_1(\Gamma_\bG;\Z)
\longrightarrow C_1(\Gamma_\bG;\Z)
\xrightarrow{\partial}\im\partial
\longrightarrow0
\]
splits because $\im\partial$ is a subgroup of the free abelian group $C_0(\Gamma_\bG;\Z)$ and is therefore free.  Hence $\tilde{d}$ extends to a morphism
\[
d:C_1(\Gamma_\bG;\Z)\longrightarrow\Z.
\]
Regarding the restriction of $d$ to arrows as an arrow grading, \eqref{eq:grading_one_penterows} implies that $d$ grades every monomial $p_\rho$ in $W(\bG)$ with degree one, as required.
\end{proof}

\noindent If \Cref{prop:existence} addresses existence, the next question is whether there exists more than one degree-one grading for $(Q(\bG),W(\bG))$, up to the appropriate equivalence relations. More precisely, knowing that $\Gclass[Q(\bG),W(\bG)]$ is non-empty, we ask whether it admits more than one element.

\begin{proposition}[Uniqueness of gradings up to vertex gauge]\label{prop:gauge-unique}
Let $\bG$ be a plabic fence and $V$ a non-degenerate potential on $Q(\bG)$. If $d$ and $d'$ are integral arrow gradings for which $V$ is homogeneous of degree one, then $d$ and $d'$ differ by a vertex gauge.
\end{proposition}

\begin{proof}
By \cite[Lemma~3.13]{CasalsGao} and non-degeneracy of the potential $V$, each monomial $p_\rho$ associated to a pente-row $\rho$ must necessarily occur in the potential $V$ with non-zero coefficient. Therefore $d(c_\rho)=1=d'(c_\rho)$ for every pente-row $\rho$. Consider $\eta=d-d'\in\Hom(C_1(\Gamma_\bG;\Z),\Z)$, then \Cref{lem:face-basis} gives
\[
\eta(z)=0,
\qquad
\text{for every }z\in Z_1(\Gamma_\bG;\Z).
\]
Thus $\eta$ factors through $C_1/Z_1\cong\im\partial$.  Since
\[
C_0(\Gamma_\bG;\Z)/\im\partial\cong H_0(\Gamma_\bG;\Z)
\]
is free, the resulting homomorphism $\im\partial\to\Z$ extends to a homomorphism $h:C_0(\Gamma_\bG;\Z)\to\Z$. By construction, such homomorphism $h$ satisfies that for every arrow $a:i\to j$,
\[
\eta(a)=h(j)-h(i).
\]
Hence
\[
d(a)=d'(a)+h(t(a))-h(s(a)),
\]
and thus $d$ and $d'$ are related by a vertex gauge.
\end{proof}

For a fixed non-degenerate potential on $Q(\bG)$, \Cref{prop:gauge-unique} establishes uniqueness of a degree-one arrow grading for $(Q(\bG),V)$, up to vertex gauge. \Cref{prop:P} establishes the uniqueness of a non-degenerate potential for any quiver $Q(\bG)$ associated to a plabic fence, but only up to right-equivalence. A priori, there could be two (necessarily right-equivalent) non-degenerate QPs $(Q(\bG),V)$ and $(Q(\bG),W)$ such that $(Q(\bG),V,d_V)$ and $(Q(\bG),W,d_W)$ give different dilations, where $d_V,d_W$ are the unique (up to vertex gauge) arrow gradings for $(Q(\bG),V)$ and $(Q(\bG),W)$ respectively. In order to conclude \Cref{thm:unique-fence-G1}, we must ensure that $(Q(\bG),V,d_V)$ and $(Q(\bG),W,d_W)$ are necessarily graded right-equivalent. This is implied by the following result:

\begin{proposition}[Graded straightening]\label{prop:straighten}
Let $Q$ be a finite quiver with no parallel arrows, and let $d:Q_1\to\Z$ be an arrow grading. Suppose $V$ and $W$ are reduced potentials on $Q$, both homogeneous of degree one for the grading $d$. Suppose that $(Q,V)$ and $(Q,W)$ are right-equivalent. Then $(Q,V,d)$ and $(Q,W,d)$ are graded right-equivalent.
\end{proposition}

\begin{proof}
To ease notation, we drop the subscripts $Q$ from the algebras, ideals and modules associated to $Q$ in \Cref{ssec:prelim_pathalgebras_and_QPs}. Choose an (ungraded) right-equivalence $\varphi$ from $(Q,V)$ to $(Q,W)$, which exists by hypothesis. Its linear part $L$ of $\varphi$ is an invertible $R$-bimodule map on the arrow span $A\cong \m/\m^2$.  Since $Q$ has no parallel arrows, each $e_jAe_i$ is at most one-dimensional and hence the linear part is diagonal on arrows. That is, for any arrow $a\in Q_1$, there exists $\lambda_a\in\kk^\times$ so that
\[
L(a)=\lambda_a a,
\]
In particular, this linear part is degree preserving. Factoring the original right-equivalence $\varphi$ as
\[
\varphi=u\circ L,
\qquad\mbox{ where } u\equiv\id\pmod{\m^2},
\]
we have that $u(L(V))\sim_{\cyc}W$ are cyclically equivalent by construction. The issue is that $u$ might not be degree preserving, and we must modify $u$ accordingly. For that, let $U\sse \Aut_{R}(\widehat{\kk Q})$ be the subgroup of unitriangular automorphisms and define
\[
\mathbb T
:=
\{v\in U:v(L(V))\sim_{\cyc}W\}.
\]
It is a non-empty right torsor for the unitriangular stabilizer $H=\{g\in U:g(L(V))\sim_{\cyc}L(V)\}$, which is a closed (pro)unipotent subgroup of $U$. Note that an integral arrow grading $d$ on $Q$ defines the rational $\Gm$-action
\[
\rho:\Gm\lr\Aut_R(\widehat{\kk Q}),\quad \rho_z(a):=z^{d(a)}a.
\]
Since $L(V)$ and $W$ have degree one, $\mathbb T$ and $H$ are stable under conjugation by the grading action $\rho_z$. Consider the $\Gm$-action on $\mathbb T$ obtained by restriction. By \Cref{lem:torsor} below, $\mathbb T$ contains a $\Gm$-fixed point $u_{\mathrm{gr}}$ under such restricted action. Since $u_{\mathrm{gr}}$ is a fixed point, we have
\[
\rho_z u_{\mathrm{gr}}\rho_z^{-1}=u_{\mathrm{gr}},
\]
which is exactly that such fixed point $u_{\mathrm{gr}}$ is degree preserving.  Since $L$ is also degree preserving, the composite $u_{\mathrm{gr}}\circ L$ is now a graded right-equivalence from $(Q,V,d)$ to $(Q,W,d)$.
\end{proof}

The proof of \Cref{prop:straighten} used the existence of a fixed point of a certain $\Gm$-action. This is proven in \Cref{lem:torsor} below. In order to state the lemma, let
\[
U^{(r)}=\{u\in\Aut_R(\widehat{\kk Q}):u\equiv\id\pmod{\m^r}\},
\qquad r\ge2,
\]
where recall $\m=\m_Q$ is the arrow ideal of a quiver $Q$. Note that $U^{(2)}$ is (pro)unipotent and the quotients $U^{(r)}/U^{(r+1)}$ are additive vector groups. Given an integral arrow grading $d:Q_1\lr\Z$, note that the rational $\Gm$-action $\rho:\Gm\lr\Aut_R(\widehat{\kk Q})$ determined by $\rho_z(a)=z^{d(a)}a$ is such that conjugation by $\rho_z$ preserves every $U^{(r)}$. Here is the statement that provides the fixed point used in the proof of \Cref{prop:straighten}.

\begin{lemma}\label{lem:torsor} In the notation above, let $U$ be a closed $\Gm$-stable (pro)unipotent subgroup of $U^{(2)}$, complete for the filtration $U\cap U^{(r)}$. Then every nonempty $\Gm$-stable right $U$-torsor compatible with the finite $\m$-adic quotients has a $\Gm$-fixed point.
\end{lemma}

\begin{proof} We proceed order by order by passing to the finite-dimensional quotients modulo $\m^N$, where we will be increasing $N\in\N$ one by one. The image $U_N\sse\Aut_R(\widehat{\kk Q}/\mathfrak{m}^N)$ of $U$ modulo $\m^N$ is a unipotent algebraic group with a rational $\Gm$-action. The filtration induced by $U^{(r)}$ has successive quotients that are rational representations of $\Gm$. Since commutators increase arrow order, this filtration is central in successive finite quotients. Now, for a rational representation $M$ of $\Gm$, we claim that we have $H^1(\Gm,M)=\{0\}$. Indeed, $\Gm$ is diagonalizable and every rational representation is a direct sum of weight spaces, thus the functor of invariants is exact and $H^1(\Gm,M)$ vanishes.  Induction along the central filtration gives that the pointed set $H^1(\Gm,U_N)$ contains one element. Therefore every nonempty $\Gm$-stable $U_N$-torsor has a fixed point. Now we keep applying this order by order: having chosen a fixed point modulo $\m^N$, the set of its lifts modulo $\m^{N+1}$ is either empty or a torsor under a $\Gm$-stable vector subgroup of $(U\cap U^{(N)})/(U\cap U^{(N+1)})$. Indeed, non-emptiness follows from the assumed compatible (pro)torsor, and $H^1(\Gm,M)=0$ gives a fixed lift.  The compatible sequence of fixed lifts then converges in the $\mathfrak{m}$-adic topology to a $\Gm$-fixed point, as required.
\end{proof}

\subsection{Proof of \Cref{thm:unique-fence-G1}}\label{ssec:proof_unique-fence-G1} \Cref{prop:existence} provides an element of $\Gclass[Q(\bG),W(\bG)]$, which we represent by a graded QP $(Q(\bG),W(\bG),d(\bG))$. Let $(Q',V',d')$ be a graded QP with $[Q',V']=[Q(\bG),W(\bG)]$. Homogeneous reduction applied to $(Q',V',d')$ yields a reduced graded QP such that its underlying reduced QP is right-equivalent to $(Q(\bG),W(\bG))$, as the latter is reduced and $[Q',V']=[Q(\bG),W(\bG)]$. (Here we use uniqueness of the ungraded reduced part, cf.~\cite[Theorem~4.6]{DWZ}.) Such reduced graded QP can thus be written as $(Q(\bG),V,d)$ where the underlying reduced QP $(Q(\bG),V)$ is right-equivalent to $(Q(\bG),W(\bG))$. In particular, $V$ is non-degenerate because non-degeneracy is invariant under right-equivalence. Now \Cref{prop:gauge-unique} implies that we can change the grading $d$ to $d(\bG)$, where we use that the quiver is now identically $Q(\bG)$. In particular, $(Q(\bG),V,d)$ and $(Q(\bG),V,d(\bG))$ represent the same class in $\Gclass[Q(\bG),W(\bG)]$. Since both $V$ and $W(\bG)$ are homogeneous of degree one for the same grading $d(\bG)$ and non-degenerate, \Cref{prop:P}.$(ii)$ implies that they are (ungraded) right-equivalent. \Cref{prop:straighten}, which can be applied thanks to \Cref{lem:no-parallel}, implies that they are graded right-equivalent. Therefore every object represents the same equivalence class as $(Q(\bG),W(\bG),d(\bG))$, which implies the result.\hfill$\Box$

\subsection{An explicit representative}\label{ssec:explicit_plabic_gradedQP} Let $\bG$ be a plabic fence. \Cref{prop:existence} established the existence of an arrow-grading $d(\bG)$ for $Q(\bG)$ so that $W(\bG)$ is homogeneous of degree one, i.e.~the existence of a graded QP $(Q(\bG),W(\bG),d(\bG))$. That said, it can be convenient to explicitly describe one such arrow-grading $d_\bG$. This is the object of this subsection.\\

Let $\mathcal F(\bG)$ be the set of bounded faces of $\bG$, i.e.~the (mutable) vertices of $Q(\bG)$, and denote
$m:=|\mathcal F(\bG)|$.  If $F\in\mathcal F(\bG)$, let $\ell(F)$ be the level
of $F$, numbered from bottom to top as in \cite[Section~3.1]{CasalsGao}, and
let $x(F)$ be the horizontal coordinate of the unique vertical edge which is
the right boundary edge of $F$. Consider the total order on faces given by
\begin{equation}\label{eq:fence-face-order}
 F<F' \quad\Longleftrightarrow\quad
 \bigl(\ell(F),x(F)\bigr)<_{\mathrm{lex}}
 \bigl(\ell(F'),x(F')\bigr).
\end{equation}
That is, in this total order levels are read from bottom to top, and the faces in one level are read
from left to right. Let us enumerate the faces of $\bG$ so that $F_1<\cdots<F_m$ and consider the integral arrow-grading $d_{\bG}:Q(\bG)_1\lr\Z$ which for each arrow $a:F_i\to F_j$ of $Q(\bG)$ assigns degree as follows:
\begin{equation}\label{eq:fence-descent-grading}
 d_{\bG}(a):=
 \begin{cases}
 0,&i<j,\\
 1,&i>j.
 \end{cases}
\end{equation}
The claim is that this assignment satisfies:

\begin{lemma}
Let $\bG$ be a plabic fence, $Q(\bG)$ its quiver and $d_\bG$ the grading defined in \eqref{eq:fence-descent-grading}. Then $W(\bG)$ is homogeneous of degree one for the grading $d_{\bG}$.
\end{lemma}

\begin{proof}
By the inductive description of $Q(\bG)$ in
\Cref{def:quiver_plabicfence} we have the following three types of arrows at each step:
\begin{enumerate}
    \item a horizontal arrow goes from a face to the
face immediately to its right, which is increasing in the total face order.

\item an upward arrow goes from a lower level to a
higher level, also increasing in the total face order.

\item a downward arrow goes from a higher level to a lower level, thus decreasing in the total face order.
\end{enumerate}

Consequently, by \eqref{eq:fence-face-order} and \eqref{eq:fence-descent-grading}, horizontal and
upward arrows have degree zero, and downward arrows have degree one. Now, every
cyclic monomial occurring in $W(\bG)$ is the cycle surrounding one black or
white pente-row.  The two local cycles displayed in
\Cref{fig:PlabicFence_Row1_Example} contain exactly one downward
arrow: all remaining arrows either run horizontally from left to right, or move upward.  Hence every cyclic monomial in $W(\bG)$ has total
$d_{\bG}$-degree one. In consequence, $(Q(\bG),W(\bG),d_{\bG})$ is a degree-one graded
QP, as required.
\end{proof}

Therefore, by the uniqueness of the dilation class established in \Cref{thm:unique-fence-G1}, the unique dilation $\delta(\bG)\in\Gclass[Q(\bG),W(\bG)]$ is represented by $(Q(\bG),W(\bG),d(\bG))$ where $d(\bG)=d_\bG$ as in \eqref{eq:fence-descent-grading}.

\section{The topology of the graded Euler form for plabic fences}\label{sec:polynomial_plabic_fence}

The object of this section is to prove \Cref{prop:dil_module_is_Alexander}, which shows that the cokernel $\Z[t,t^{-1}]$-module of the graded Euler form associated to a plabic fence is isomorphic to the torsion part of the Alexander module of the link associated to the plabic fence. This is a key step in the proof of the Main Theorem.\\

Note that \Cref{prop:dil_module_is_Alexander} establishes such isomorphism of $\Z[t,t^{-1}]$-modules for any plabic fence, whether it is associated to an algebraic singularity or not. In the proof of the Main Theorem it will be specialized to plabic fences associated to algebraic malleable divides, as discussed in \Cref{sec:proof_main}.


\subsection{The Alexander module of a smooth link}\label{ssec:Alexander_modules} This subsection defines the Alexander module of  a link and recalls a way to compute it via Seifert matrices. Let $L:=L_1\cup\cdots\cup L_r\subset S^3$ be a smooth oriented link in the 3-sphere and
\[
   X_L:=S^3\setminus\operatorname{int}\nu L
\]
its complement. The oriented meridians $\mu_{1},\ldots,\mu_{r}$ of the components of $L$ give an isomorphism $H_1(X_L;\Z)\cong \Z^r$. The Hurewicz homomorphism composed with a map sending all meridians to the same generator defines a group morphism
\[
   \tau:\pi_1(X_L)\twoheadrightarrow H_1(X_L;\Z)
   \longrightarrow\Z,
   \qquad
   \tau(\mu_i)=1.
\]
Let $p:\widetilde X_L^{\,\tau}\longrightarrow X_L$ be the connected infinite cyclic cover associated to $\ker\tau$, see e.g.~~\cite[Chapter 6]{lickorish1997introduction}. Then the quotient
\[
   \pi_1(X_L)/\ker\tau\cong\Z
\]
acts on $(\ker\tau)_{\mathrm{ab}}\cong H_1(\widetilde X_L^{\,\tau};\Z)$ on the right by conjugation. Specifically, choosing a homotopy class $\mu\in\pi_1(X_L)$ with $\tau(\mu)=1$, we use the convention
\begin{equation}\label{eq:right-action}
t\cdot x:=\mu x\mu^{-1}.
\end{equation}
This is independent of the choice of $\mu$ with $\tau(\mu)=1$ after passing to the abelianization of $\ker\tau$. In particular, $(\ker\tau)_{\mathrm{ab}}\cong H_1(\widetilde X_L^{\,\tau};\Z)$ becomes a module over the group ring of $\Z$:  

\begin{definition}[Alexander module]\label{def:alex-module}
By definition, the integral Alexander module of a link $L\sse S^3$ is the $\Z[t,t^{-1}]$-module 
\[
   A_L^\tau:=H_1(\widetilde X_L^{\,\tau};\Z)
\]
resulting from the action \eqref{eq:right-action}, where $\Z[t,t^{-1}]\cong\Z[\Z]$ is the group ring of $\Z$. We often denote the Alexander module $A_L^\tau$ of $L$ by $A_L$ if $\tau$ is implicit by context.\hfill$\Box$
\end{definition}

Recall that we will often denote $\Lambda:=\Z[t,t^{-1}]$, and that the Alexander module and the Alexander polynomial are only defined up to a product with the units $\pm t^\m$ of $\Lambda$, $m\in\Z$. There is an explicit way to present the Alexander module if a Seifert surface for $L$ has been chosen, cf.~\cite[Chapter 2]{lickorish1997introduction}. Indeed, let $S\subset S^3$ be a connected oriented Seifert surface for $L$, and choose a basis $\gamma_1,\ldots,\gamma_m$ of $H_1(S;\Z)$. Let $\gamma_j^+$ be the positive normal push-off and define the standard Seifert matrix $V=(V_{ij})$ by the entries
\begin{equation}\label{eq:V-def}
   V_{ij}:=\lk(\gamma_i,\gamma_j^+).
\end{equation}

\noindent By \cite[Theorem 6.5]{lickorish1997introduction}, for instance, the Seifert matrix in \eqref{eq:V-def} encodes the integral Alexander module as follows:

\begin{proposition}\label{prop:seifert-presentation} Let $L\sse S^3$ be a link with Seifert surface $S\sse S^3$ and Seifert matrix $V$. Then the Alexander module $A_L$ of $L$ is isomorphic to
   $$A_L\cong\coker_\Lambda(V^{\mathsf T}-tV).$$
   \hfill$\Box$
\end{proposition}

In general, we must be careful with split links. This will not be an issue for algebraic links and it will not affect the cases relevant to the Main Theorem. That said, for completeness and clarity, we present the general argument for an arbitrary plabic fence. For split links, the following is a consequence of Mayer-Vietoris:

\begin{lemma}[Alexander module of a split link]\label{lem:split-module}
Let $L=L'\sqcup L''$ be a split union of two nonempty oriented links, and let $\tau,\tau',\tau''$ denote the corresponding homomorphisms from the fundamental group of the complement to $\Z$. Then the respective Alexander modules satisfy
\begin{equation}\label{eq:split-two}
   A_L^\tau
   \cong
   A_{L'}^{\tau'}\oplus A_{L''}^{\tau''}\oplus\Lambda.
\end{equation}
In consequence, for a general split link $L=L_1\sqcup\cdots\sqcup L_s$, and morphisms $\tau,\tau_1,\ldots,\tau_s$, we thus have
\begin{equation}\label{eq:split-s}
   A_L^\tau
   \cong
   \Lambda^{s-1}\oplus
   \bigoplus_{i=1}^s A_{L_i}^{\tau_i}.
\end{equation}
\end{lemma}

\begin{proof}
It suffices to prove \eqref{eq:split-two}. For that, let us choose a splitting 2-sphere $S\subset X_L$, so that cutting along $S$ gives two pieces obtained from $X_{L'}$ and $X_{L''}$ by deleting an open $3$-ball.  The restrictions of the morphisms $\tau'$ and $\tau''$ to these pieces are surjective, because every nonempty link has a meridian mapped to $1$.  Hence the corresponding covering pieces $\widetilde Y'$ and $\widetilde Y''$ are connected.

Now, deleting the entire orbit of a $3$-ball under the deck covering group does not change first homology. In addition, the inverse image of the splitting sphere is a disjoint union of spheres indexed by the deck group, and thus we obtain:
\[
   H_1(p^{-1}S;\Z)=0,
   \qquad
   H_0(p^{-1}S;\Z)\cong\Z[\Z]\cong\Lambda.
\]
Since each lifted side is connected,
\[
   H_0(\widetilde Y';\Z)
   \cong H_0(\widetilde Y'';\Z)
   \cong \Lambda/(t-1).
\]
The relevant part of the Mayer--Vietoris sequence is therefore
\begin{align*}
  0&\longrightarrow
  A_{L'}^{\tau'}\oplus A_{L''}^{\tau''}
  \longrightarrow A_L^\tau
  \longrightarrow \Lambda
  \xrightarrow{\phi}
  \Lambda/(t-1)\oplus\Lambda/(t-1),
\end{align*}
where the map $\phi$ is given by
\[
   \phi(f)=(\epsilon(f),-\epsilon(f)),
   \qquad \mbox{ where }\epsilon(t)=1.
\]
In particular, its kernel is $\ker\phi=(t-1)\Lambda\cong\Lambda$ and exactness gives a short exact sequence
\[
  0\longrightarrow
  A_{L'}^{\tau'}\oplus A_{L''}^{\tau''}
  \longrightarrow A_L^\tau
  \longrightarrow (t-1)\Lambda
  \longrightarrow0.
\]
Since $(t-1)\Lambda$ is free of rank one over $\Lambda$, the sequence splits, which proves \eqref{eq:split-two}. The general formula \eqref{eq:split-s} follows from repeatedly applying \eqref{eq:split-two}.
\end{proof}

\subsection{The cokernel of the graded Euler form of a plabic fence}\label{ssec:polynomial_plabic_fence} Let $\bG$ be a plabic fence, $Q(\bG),W(\bG)$ its unique non-degenerate QP, up to right equivalence, and $\delta\in\Gclass[Q(\bG),W(\bG)]$ the unique dilation class, cf.~\Cref{thm:unique-fence-G1}. By \Cref{cor:unique_dilation_plabic_fence}, the cokernel of the graded Euler pairing is an invariant of the quiver mutation class of $Q(\bG)$. To ease notation, we denote the graded Euler form of a plabic fence $\bG$ by $\Eul_\bG:=\Eul_{(Q(\bG),W(\bG),\delta(\bG))}$, and its cokernel by
\begin{equation}\label{eq:dilation_module}
\coker\Eul_{\bG}(t)\in\Z[t,t^{-1}]\mbox{-mod}.
\end{equation}
The object of this subsection is to identify $\coker\Eul_{\bG}(t)$ with a more geometric quantity, as achieved in \Cref{prop:dil_module_is_Alexander}.\\

For that, let us consider the smooth link $L(\bG)\sse S^3$ obtained by taking the trivial 0-framed closure of the positive braid word $\beta=\beta(\bG)$ associated to $\bG$. Namely, from a plabic fence $\bG$ we construct a braid word $\beta(\bG)$ iteratively by scanning the plabic fence left-to-right: starting with the empty word $\beta(\bG)$, when we encounter a vertical edge between the $k$ and $(k+1)$st horizontal edges, counting from the bottom, we add the positive Artin generator $\sigma_k$ to the right of $\beta(\bG)$. Note that we obtain a positive braid word through this process, as all Artin generators used in this construction are positive. By closing up $\beta(\bG)$ via the 0-framed closure, we obtain a smooth link $L(\bG)\sse S^3$. For an irreducible plane curve singularity $L(\bG)\sse S^3$ is a knot.\\

For a $\Lambda$-module $M$, we denote by 
\[
   \Tors_{\Lambda}M
   :=\{x\in M\mid f\cdot x=0\text{ for some non-zero }f\in\Lambda\}
\]
the torsion submodule. The cokernel module $\coker\Eul_{\bG}(t)$ in \eqref{eq:dilation_module} relates to the Alexander module from \Cref{def:alex-module} as follows:

\begin{theorem}\label{prop:dil_module_is_Alexander}
Let $\bG$ be a plabic fence and $L(\bG)\sse S^3$ its associated smooth link. Then there exists an isomorphism of $\Lambda$-modules
\begin{equation}\label{eq:coker_module_is_Alexander}
   \boxed{\;\coker\Eul_{\bG}(t)\cong\Tors_\Lambda A_{L(\bG)}.\;}
\end{equation}
That is, the cokernel of the graded Euler form associated to $(Q(\bG),W(\bG),\delta(\bG))$ is isomorphic to the torsion submodule of the Alexander module of $L(\bG)$.
\end{theorem}

The rest of this section is devoted to establishing \Cref{prop:dil_module_is_Alexander}.\\

\subsection{Graded Euler form and Seifert matrices}\label{ssec:graded_Euler_Seifert} Let $\bG$ be a plabic fence, and $L(\bG)\sse S^3$ its link. The plabic fence also provides a Seifert surface $\Sigma_\bG$ and a Seifert matrix $V_\bG$, via the brick diagrams of \cite[Section 3]{BaaderDehornoy}, see also \cite[Section 2]{Ferretti}.\footnote{Equivalently, these are the standard Seifert surface and matrix associated to the positive-band
decomposition from a positive braid word, i.e.~ with one disk for each braid strand and attaching one positive half-twisted
band for each braid letter.} The relation between the graded Euler form of $\bG$ and the Seifert matrix $V_\bG$ is as follows:

\begin{proposition}\label{prop:fence-seifert}
Let $\bG$ be a plabic fence, $L(\bG)$ its associated link, and $V_\bG$ its brick-diagram Seifert matrix. Then there is an equality of matrices
\begin{equation}\label{eq:E-V}
   \boxed{\;\Eul_{\bG}(t)=V_\bG^{\mathsf T}-tV_\bG.\;}
\end{equation}
In addition, if the Seifert surface $\Sigma_\bG$ decomposes as $\Sigma_\bG=\bigsqcup_{i=1}^s\Sigma_i$, then 
\begin{equation}\label{eq:E-block}
   \Eul_{\bG}(t)=\bigoplus_{i=1}^s
   \bigl(V_i^{\mathsf T}-tV_i\bigr).
\end{equation}
after possibly reordering brick cycles componentwise.
\end{proposition}

\begin{proof} Let us use the notation in \Cref{ssec:explicit_plabic_gradedQP} and consider the representative $(Q(\bG),W(\bG),d(\bG))$ of the unique dilation class in $\Gclass[Q(\bG),W(\bG)]$, where $d(\bG)=d_\bG$ is given as in \eqref{eq:fence-descent-grading}. Define
\[
 b_{ij}:=\#\{F_i\to F_j\}-\#\{F_j\to F_i\},
 \qquad B_{\bG}:=(b_{ij})_{1\leq i,j\leq m},
\]
and let $U_{\bG}$ be the unique unipotent upper-triangular integral matrix satisfying
\begin{equation}\label{eq:fence-companion}
 U_{\bG}-U_{\bG}^{\mathsf T}=-B_{\bG}.
\end{equation}
In particular, $(U_{\bG})_{ij}=-b_{ij}$ for $i<j$. Now, the arrow grading \eqref{eq:fence-descent-grading} is such that, if $i<j$, every arrow
$F_i\to F_j$ has degree zero and every arrow $F_j\to F_i$ has degree one.
Therefore \Cref{prop:euler-formula} gives
\begin{align*}
 \Eul_{ij}(t)
 &=-t\,\#\{F_j\to F_i\}
   +t\,\#\{F_i\to F_j\}
   =t b_{ij},\\
 \Eul_{ji}(t)
 &=-\#\{F_i\to F_j\}
   +\#\{F_j\to F_i\}
   =-b_{ij},\\
\Eul_{ii}(t)
&=1-t
\end{align*}
By comparing entries directly, it follows that
\begin{equation}\label{eq:fence-euler-companion}
 \Eul_{(Q(\bG),W(\bG),d_{\bG})}(t)
 =U_{\bG}^{\mathsf T}-tU_{\bG}.
\end{equation}
It suffices to show that $U_{\bG}$ is the Seifert matrix $V_\bG$ for $L(\bG)$. This last step is likely known to experts but, for completeness, we provide the necessary details.\\

First, rotate the plabic fence clockwise by $\pi/2$, so that its horizontal lines become
the strands of the standard brick diagram (cf.~\cite[Section 3]{BaaderDehornoy} or \cite[Section 2]{Ferretti}) of the positive braid
$\beta=\beta(\bG)$, its vertical edges become the crossings, and its bounded
faces become precisely the bricks (the innermost rectangles) of that brick
diagram.  The clockwise rotation is the one for which scanning the fence from
left to right becomes reading the braid word from top to bottom.  Under this
rotation, the ordering \eqref{eq:fence-face-order} reads the brick columns from
left to right and, inside each column, from top to bottom.\\

Now, let $\Sigma_\beta$ be the standard Seifert surface obtained from one disk per
braid strand and one positive half-twisted band per letter of $\beta$, and argue now for each connected component. By construction, the brick diagram embeds in $\Sigma_\beta$ as a deformation
retract, and the counterclockwise core curves of the bricks form a basis of
$H_1(\Sigma_\beta;\Z)$; cf.~
\cite[Section~2]{Ferretti}.
Denote these curves, in the above order, by
$\gamma_1,\ldots,\gamma_m$. Now note that these curves just described are the same curves that feature in the QP $(Q(\bG),W(\bG))$.  Namely, after the above rotation, the three local ribbon
models defining the conjugate surface of a fence in
\cite[Section~3]{CasalsGao} become the disk-and-band
local models of $\Sigma_\beta$.  These local identifications agree on their
boundary intervals and therefore glue to an oriented ribbon-surface
identification which carries the face curve $\gamma_{F_i}$ to the brick core
$\gamma_i$.  By \cite[Def.~2.2\&Prop.~3.7]{CasalsGao}, the
skew-adjacency matrix of $Q(\bG)$ is consequently the algebraic intersection
matrix of this basis, i.e.
\begin{equation}\label{eq:fence-B-intersection}
 b_{ij}=[\gamma_i]\cdot[\gamma_j].
\end{equation}
Here the surface and curve orientations are those of the cited local ribbon
models, which also fixes the signs onwards. Let
\begin{equation}\label{eq:fence-seifert-matrix}
 V_{ij}:=\operatorname{lk}(\gamma_i^+,\gamma_j),
\end{equation}
where $\gamma_i^+$ denotes the positive normal push-off determined by the
corresponding orientation. By construction, $V=(V_{ij})$ is a Seifert matrix for
$\Sigma_\beta$ and it coincides with the Seifert matrix $V=V_\bG^{\mathsf T}$. Let us now relate $U_\bG$ and $V$. In the deplumbing order given by $\gamma_1,\ldots,\gamma_m$, $V$ is lower
unitriangular. Indeed, $\gamma_i$ is the core of a Hopf band lying above the
surface which remains after the first $i$ deplumbings.  Its positive push-off
is therefore unlinked from each later core $\gamma_j$, so $V_{ij}=0$ for
$i<j$, while the core of a positive Hopf band has self-linking $+1$. We therefore have
\begin{equation}\label{eq:V_is_Seifert}
\det(V-tV^{\mathsf T}) =\Delta_{L(\bG)}(t).
\end{equation}
With the convention \eqref{eq:fence-seifert-matrix}, the usual relation
between the Seifert form and the algebraic-intersection form is
\begin{equation}\label{eq:seifert-intersection-convention}
 [\gamma_i]\cdot[\gamma_j]=V_{ij}-V_{ji};
\end{equation}
see e.g.~ \cite[Pages 200--203]{Rolfsen}.  Combining
\eqref{eq:fence-B-intersection} and
\eqref{eq:seifert-intersection-convention} yields
\[
 B_{\bG}=V-V^{\mathsf T},
 \qquad\text{or equivalently}\qquad
 V^{\mathsf T}-V=-B_{\bG}.
\]
Since $V^{\mathsf T}$ is upper unitriangular, the uniqueness in
\eqref{eq:fence-companion} gives
\begin{equation}\label{eq:U-is-Seifert}
 U_{\bG}=V^{\mathsf T}=V_\bG.
 \end{equation}
 Therefore we obtain the required equality \eqref{eq:E-V}. The equality \eqref{eq:E-block} following from \eqref{eq:E-V} applied to each connected component of the Seifert surface $\Sigma_\bG$.
 \end{proof}

\subsection{Proof of \Cref{prop:dil_module_is_Alexander}}\label{ssec:proof_dilation_module_is_Alexander} By \Cref{prop:fence-seifert}, specifically \eqref{eq:E-block}, we have
\begin{equation}\label{eq:coker-blocks-E}
  \Eul_\bG(t)=
  \bigoplus_{i=1}^s
  \bigl(V_i^{\mathsf T}-tV_i\bigr).
\end{equation}
By applying \Cref{prop:seifert-presentation} to \eqref{eq:coker-blocks-E} we obtain
\begin{equation}\label{eq:topological-summands}
\coker_\Lambda \Eul_\bG(t)
   \cong
   \bigoplus_{i=1}^s A_{L_i}.
\end{equation}
Note that each summand on the right is $\Lambda$-torsion, as the corresponding square Seifert blocks are unitriangular. Now, for a plabic fence $\bG$ with Seifert surface $\Sigma_\bG$, a decomposition in connected components
\[
   \Sigma_\bG=\Sigma_1\sqcup\cdots\sqcup\Sigma_s.
\]
implies that the link $L(\bG)$ is a split link, expressed as the split union
\[
   L(\bG)=L_1\sqcup\cdots\sqcup L_s,
\]
where $L_i=\partial\Sigma_i$. By \Cref{lem:split-module}, specifically using \eqref{eq:split-s}, we obtain the following isomorphism of $\Lambda$-modules \begin{equation}\label{eq:proof_dilation_is_Alexander1}
   A_{L(\bG)}
   \cong
   \Lambda^{s-1}
   \oplus
   \bigoplus_{i=1}^s A_{L_i}.
\end{equation}
By combining \eqref{eq:topological-summands} and \eqref{eq:proof_dilation_is_Alexander1} together we obtain
\[
   A_{L(\bG)}
   \cong
   \Lambda^{s-1}\oplus\coker_\Lambda \Eul_\bG(t).
\]
Since $\coker_\Lambda \Eul_\bG(t)$ is $\Lambda$-torsion, it is exactly the torsion submodule of the right side, and thus the required isomorphism
\[
   \coker_\Lambda \Eul_\bG(t)
   \cong
   \Tors_\Lambda A_{L(\bG)},
\]
of $\Lambda$-modules follows, i.e.~we have proven \eqref{eq:coker_module_is_Alexander}.\hfill$\Box$

\section{Plane curve singularities and the proof of the Main Result}\label{sec:proof_main}

The object of this section is to conclude the proof of the Main Theorem. For that, we now finally specialize to the context of isolated plane curve singularities.

\subsection{Plane curve singularities and their integral monodromy modules} In general, we refer to \cite{ACampo1998,ACampo1999,FPST} for the necessary definitions in the study of real Morsifications of plane curve singularities and their divides. That said, we provide here the necessary basic ingredients for our context.\\

Let $(C,0)$ be an isolated plane curve singularity with representative $f:(\C^2,0)\lr(\C,0)$. The Milnor fiber of $f$ will be denoted by $M_f$ and the link of the singularity $(C,0)$ by
$$L(C):=\{(x,y)\in\C^2:f(x,y)=0\}\cap S^3\sse S^3,$$
where $S^3\sse\C^2$ is a 3-sphere of small enough radius, cf.~\cite[Chapter 4]{milnor1968} or \cite[Chapter 4]{Seade19_MilnorFibration}. Note that the intersection of $M_f$ with a 4-ball of that radius is a smoothly embedded surface in a 4-ball bounding $L(C)$. The Milnor fibration provides an explicit model for the geometric monodromy of $(C,0)$, via the fiber bundle
\begin{equation}\label{eq:Milnor_fibration}
\frac{f}{|f|}:S^3\setminus L(C)\lr S^1
\end{equation}
whose fiber is isomorphic to $M_f$, see \cite[Theorem 4.8]{milnor1968}. Specifically, we will focus on the lattice $H_1(M_f;\Z)$ and the action of the geometric monodromy $h\in\pi_0(\mbox{Diff}^c(M_f))$ of \eqref{eq:Milnor_fibration} on $H_1(M_f;\Z)$.

\begin{definition}
Let $(C,0)$ be an isolated plane curve singularity with representative $f:(\C^2,0)\lr(\C,0)$, Milnor fiber $M_f$ and geometric monodromy $h\in\pi_0(\mbox{Diff}^c(M_f))$. By definition, the integral monodromy module of $(C,0)$ is
$$\Z[t,t^{-1}]\circlearrowright H_1(M_f;\Z),\quad t\cdot[\gamma]=h_*([\gamma])$$
where $t$ acts on a cycle in $H_1(M_f;\Z)$ via the induced action
$$h_*:H_1(M_f;\Z)\lr H_1(M_f;\Z).$$
The characteristic polynomial of the monodromy is defined to be $\det(h_*-t\cdot\mbox{Id})$.\hfill$\Box$
\end{definition}

It is a well-known fact, see e.g.~\cite[Section 2]{Durfee1975}, that the Alexander polynomial $\Delta_C(t)$ of the link $L(C)$ of the singularity is the characteristic polynomial of its algebraic monodromy:
\begin{equation}\label{eq:alexander_characteristic_poly}
\Delta_C(t)=\det(t\cdot\mbox{Id}-h_*)
\end{equation}
More generally, the integral monodromy module $\Z[t,t^{-1}]\circlearrowright H_1(M_f;\Z)$ is isomorphic to the Alexander module from \Cref{def:alex-module}:
\begin{equation}\label{eq:alexander_module_monodromy}
A_{L(C)}\cong (H_1(M_f;\Z),h_*)
\end{equation}
as an isomorphism of $\Lambda$-modules, with the covering action on the left hand side, and the algebraic monodromy action on the right hand side. Note that the key isomorphism $H_1(\widetilde{X}_{L(C)};\Z)\cong H_1(M_f;\Z)$ is deduced directly from the fact that the fibers of \eqref{eq:Milnor_fibration} are diffeomorphic to $M_f$ and thus the infinite cover $\widetilde{X}_{L(C)}$ is diffeomorphic to $M_f\times\R$.

\begin{remark}\label{rmk:multivariable_Alexander}
In general, when $L(C)$ has $r$ components, $r\geq2$, there is also the notion of the multi-variable Alexander polynomial $\Delta_{L(C)}^{mv}(t_1,\ldots,t_r)$, see \cite{CDGZ,EisenbudNeumann1985} or \cite{Yamamoto1984} for details on the multi-variable Alexander polynomial of a link. In this manuscript we will not use the multi-variable version beyond noting that the specialization of all $t_i$ variables to be equal to each other yields the identity
$$\det(t\cdot\mbox{Id}-h_*)=(t-1)\Delta_{L(C)}^{mv}(t,\ldots,t),$$
where, by \eqref{eq:alexander_characteristic_poly}, the left hand side can be understood as the one-variable Alexander polynomial.\hfill$\Box$
\end{remark}

Finally, we note that the link of an isolated plane curve singularity is never a split link. Indeed, for two distinct branches $C_i,C_j$ we have the equality:
\begin{equation}\label{eq:linking-intersection}
   \lk(L_i,L_j)=I_0(C_i,C_j)>0,
\end{equation}
where $I_0$ is local intersection multiplicity, cf.~\cite[Chapter~V]{EisenbudNeumann1985} or \cite[Chapters~4\&5]{wall2004}. If a splitting 2-sphere existed, it would partition the components into two nonempty sets and force all linking numbers across the partition to vanish, contradicting \eqref{eq:linking-intersection}. Thus we conclude that an algebraic link is nonsplit, as claimed.\\

\noindent In fact, the entire Alexander module is torsion in the case of algebraic links, as it follows from the existence of the Milnor fibration \cite[Chapter 4]{milnor1968}. Therefore \Cref{prop:dil_module_is_Alexander} implies the following:

\begin{corollary}\label{cor:cokernel_module_is_Alexander_algebraic_links}
Let $\bG$ be a plabic fence and $L(\bG)$ its associated link. Suppose that $L(\bG)$ is an algebraic link. Then there exists an isomorphism of $\Lambda$-modules
\begin{equation}\label{eq:coker_module_is_Alexander_algebraic}
   \boxed{\;\coker\Eul_{\bG}(t)\cong A_{L(\bG)}.\;}
\end{equation}
\end{corollary}

\subsection{A preliminary lemma on malleable divides}

As before, we refer to \cite{ACampo1998,ACampo1999,FPST} for results in the study of real Morsifications of plane curve singularities and their divides. The only relevant fact at this stage is that \cite[Def.~4.2\&Def.~7.1]{FPST} respectively construct a quiver $Q(D)$ and a smooth link $L(D)\sse S^3$ for any divide $D\sse\R^2$. Divides are discussed in detail in \cite[Section 2]{FPST}. Thanks to the results developed in Sections \ref{sec:graded-QP}, \ref{sec:euler}, \ref{sec:dilation_plabic_fence} and \ref{sec:polynomial_plabic_fence}, the only property of $Q(D)$ and $L(D)$ that we need to prove the Main Theorem is the following:

\begin{figure}
    \begin{tikzpicture} [baseline=10,scale=1]
    \foreach \i in {0,1,2}
    {
    \draw (0,\i*0.5) -- (1.5,\i*0.5);
    }
    \vertbar[](0.5,0,0.5);
    \vertbar[](1,1,0.5);
    \end{tikzpicture}\quad $\stackrel{\sim}{\longleftrightarrow}$ \quad
    \begin{tikzpicture} [baseline=10,scale=1]
    \foreach \i in {0,1,2}
    {
    \draw (0,\i*0.5) -- (1.5,\i*0.5);
    }
    \vertbar[](1,0,0.5);
    \vertbar[](0.5,1,0.5);
    \end{tikzpicture}
    \quad \quad \& \quad \quad 
    \begin{tikzpicture} [baseline=10,scale=1]
    \foreach \i in {0,1,2}
    {
    \draw (0,\i*0.5) -- (1.5,\i*0.5);
    }
    \vertbar[](0.5,0.5,0);
    \vertbar[](1,0.5,1);
    \end{tikzpicture}\quad $\stackrel{\sim}{\longleftrightarrow}$ \quad
    \begin{tikzpicture} [baseline=10,scale=1]
    \foreach \i in {0,1,2}
    {
    \draw (0,\i*0.5) -- (1.5,\i*0.5);
    }
    \vertbar[](1,0.5,0);
    \vertbar[](0.5,0.5,1);
    \end{tikzpicture}
    \caption{Instances of sliding vertical edges: under these moves, the quivers of the corresponding double plabic fences remain identical.}
    \label{fig:sliding}
\end{figure}
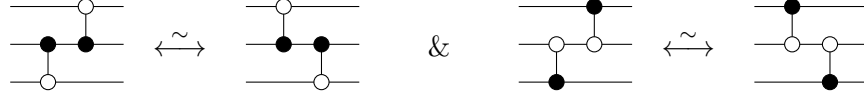

\begin{lemma}\label{lem:malleable_to_plabic_fence}
Let $D\sse\R^2$ be the divide of a real Morsification of a real isolated plane curve singularity, and $Q(D)$ and $L(D)$ its associated quiver and link. Suppose that $D$ is malleable. Then:
\begin{enumerate}[label=$(\roman*)$]
    \item There exists a plabic fence $\bG(D)$ so that the associated quiver $Q(D)$ is mutation equivalent to the quiver $Q(\bG(D))$ of $\bG(D)$.\\

    \item The link $L(\bG(D))$ associated to such plabic fence is smoothly isotopic to the link $L(D)$.
\end{enumerate}
\end{lemma}

\begin{proof} The argument for Part (1) will also imply Part (2). Let us start with Part (1). The hypothesis that $D$ is malleable implies, by definition, that it is triangle-move equivalent to a scannable divide $D'$, cf.~\cite[Definition 14.9]{FPST}. By \cite[Definition 12.2]{FPST}, we can associate a double plabic fence $\Phi(D')$ to the scannable divide $D'$, cf.~\Cref{rmk:double_plabic_fence}. Such double plabic fence $\Phi(D')$ has an associated quiver $Q(\Phi(D'))$, as in \cite[Def.~6.8]{FPST}, which coincides with the quiver $Q(D')$ associated via \cite[Def.~4.2]{FPST}, i.e.~$Q(\Phi(D'))=Q(D')$. Since triangle moves preserves the mutation class of the associated quivers, we have $[Q(D)]=[Q(D')]=[Q(\Phi(D'))]$.\\

Now we can use the reflection and sliding edge moves from \cite[Sect.~5.2\&5.3]{CasalsWeng22} on the plabic fence $\Phi(D')$, which are depicted in \Cref{fig:sliding} and \Cref{fig:reduction_and_sq_move}. In words, a sliding edge move allows us to move a given vertical edge $e$ of a double plabic fence left and right across edges at different levels which are vertically opposite to $e$.\footnote{Here, if $e$ has white on top and black on the bottom, a vertical edge is said to be vertically opposite to $e$ if it has black on top and white on the bottom. Similarly, if $e$ has black on top and white on the bottom, an edge is said to be vertically opposite to $e$ if it has white on top and black on the bottom.} Such sliding moves do not change the underlying quiver $Q(\Phi(D'))$ of the double plabic fence. A reflection move allows us to change a vertical edge $e$ into its vertically opposite {\it only} when $e$ is the rightmost or leftmost edge of a given level of the double plabic fence. For instance, the word {\it end} in \Cref{fig:reduction_and_sq_move} (left) is indicating that the pictured vertical edge is the leftmost vertical edge of its level. A reflection move does not change the underlying quiver $Q(\Phi(D'))$ of the double plabic fence. (Recall that we are only considering mutable parts.) In addition to sliding moves and reflection moves, we can also perform square moves as in \Cref{fig:reduction_and_sq_move} (left). A square move changes the underlying quiver $Q(\Phi(D'))$ of the double plabic fence by a mutation at the vertex associated to the corresponding square face at which we are applying the square move.\\

\noindent By using sliding moves, reflection moves, and square moves, as above, we can modify any double plabic fence into a plabic fence, as in \Cref{def:plabic_fence}. That is, we can transform a double plabic fence into a double plabic fence whose vertical edges have all white on top and black on bottom. As discussed above, the corresponding quiver undergoes a series of quiver mutations. By applying this procedure to $\Phi(D')$ we obtain a plabic fence $\overline{\Phi}(D')$ such that its quiver $Q(\overline{\Phi}(D'))$ is in the same mutation class as $Q(\Phi(D'))$, i.e.~$[Q(\overline{\Phi}(D'))]=[Q(\Phi(D'))]$. Therefore
$$[Q(D)]=[Q(D')]=[Q(\Phi(D'))]=[Q(\overline{\Phi}(D'))],$$
which concludes the required statement in Part (1) by setting $\bG(D):=\overline{\Phi}(D')$. For Part (2), it suffices to note that triangle moves, sliding and reflection moves, and square moves all preserve the smooth isotopy type of their underlying smooth links.
\end{proof}

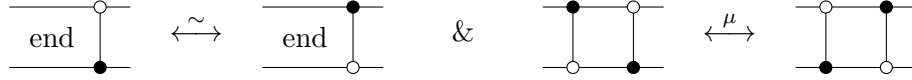
\begin{figure}
    \begin{tikzpicture}[baseline=10,scale=0.8]
    \draw (0,0) -- (2,0);
    \draw (0,1) -- (2,1);
    \node[left] at (1.25, 0.5) {end};
    \vertbar[](1.5,1,0);
    \end{tikzpicture} \quad $\stackrel{\sim}{\longleftrightarrow}$ \quad
    \begin{tikzpicture}[baseline=10,scale=0.8]
    \draw (0,0) -- (2,0);
    \draw (0,1) -- (2,1);
    \node[left] at (1.25, 0.5) {end};
    \vertbar[](1.5,0,1);
    \end{tikzpicture} \quad\quad \& \quad \quad     \begin{tikzpicture}[baseline=10,scale=0.8]
    \draw (0,0) -- (2,0);
    \draw (0,1) -- (2,1);
    \vertbar[](0.5,0,1);
    \vertbar[](1.5,1,0);
    \end{tikzpicture} \quad $\stackrel{\mu}{\longleftrightarrow}$ \quad
    \begin{tikzpicture}[baseline=10,scale=0.8]
    \draw (0,0) -- (2,0);
    \draw (0,1) -- (2,1);
    \vertbar[](0.5,1,0);
    \vertbar[](1.5,0,1);
    \end{tikzpicture}
     \caption{(Left) An instance of a reflection move. Under a reflection move, the quivers of the corresponding double plabic fences remain identical. (Right) A square move, the quivers of the corresponding double plabic fences undergo a quiver mutation at the vertex associated to the square face of the plabic fence where the move is applied.}
    \label{fig:reduction_and_sq_move}
\end{figure}

\subsection{Proof of Main Theorem} Let $D$ be a malleable divide. By \Cref{lem:malleable_to_plabic_fence}.(i), its quiver $Q(D)$ is mutation equivalent to the quiver of a plabic fence $\bG(D)$:
\begin{equation}\label{eq:proof_mutation_classes}
[Q(D)]=[Q(\bG(D))]
\end{equation}

By \Cref{cor:unique_dilation_mutation_equivalent_plabic_fence}.$(i)\&(ii)$, $Q(D)$ admits a unique non-degenerate potential $W(D)$, up to right equivalence, and a unique dilation $\delta(D)\in\Gclass[Q(D),W(D)]$. By \Cref{cor:unique_dilation_mutation_equivalent_plabic_fence}.$(iii)$, the cokernel module
    $$\coker\Eul_{(Q(D),W(D),\delta(D))}(t)\in\Z[t,t^{-1}]\mbox{-mod}$$
    is an invariant of the quiver mutation class $[Q(D)]$. By using \Cref{prop:P}.(ii) and \Cref{thm:unique-fence-G1}, we can endow the quiver $Q(\bG(D))$ with its unique non-degenerate potential $W(\bG(D))$, up to right equivalence, and its unique dilation $\delta(\bG(D))$. Then \eqref{eq:proof_mutation_classes} implies the isomorphism
    \begin{equation}\label{eq:determinants_equal_divide_fence}
    \coker\Eul_{(Q(D),W(D),\delta(D))}(t)\cong\coker\Eul_{(Q(\bG(D)),W(\bG(D)),\delta(\bG(D)))}(t)
    \end{equation}
    of $\Lambda$-modules. By \Cref{prop:dil_module_is_Alexander}, the right hand side of \eqref{eq:determinants_equal_divide_fence} equals the Alexander module $A_{L(\bG(D))}$ of $L(\bG(D))$. By \Cref{lem:malleable_to_plabic_fence}.(ii), $L(\bG(D))$ is smoothly isotopic to $L(D)$ and thus we have the following sequence of isomorphisms of $\Lambda$-modules:
    $$\coker\Eul_{(Q(D),W(D),\delta(D))}(t)\cong\coker\Eul_{(Q(\bG(D)),W(\bG(D)),\delta(\bG(D)))}(t)\cong A_{L(\bG(D))}\cong A_{L(D)},$$
    where the key isomorphism in the middle is \eqref{eq:coker_module_is_Alexander_algebraic} in \Cref{cor:cokernel_module_is_Alexander_algebraic_links}. This shows that the Alexander module $A_{L(D)}(t)$ of the link $L(D)$ of the divide $D$ can be entirely recovered from the quiver mutation class $[Q(D)]$. By \eqref{eq:alexander_module_monodromy}, the Alexander module $A_{L(D)}(t)$ is isomorphic to the integral monodromy module of the singularity. Therefore, the mutation class $[Q(D)]$ recovers the integral monodromy module of the singularity.\hfill$\Box$
    
\subsection{An explicit non-degenerate graded QP for divides} Even if it is not logically needed to conclude the Main Theorem, given a divide $D$ it can be useful to explicitly describe a representative for a non-degenerate potential $W(D)$ and compatible arrow-grading $d(D):Q(D)_1\lr\Z$, so that $W(D)$ is homogeneous of degree one. For that, recall from \cite[Definition 4.2]{FPST} that the vertices of $Q(D)$ are partitioned into three classes
\[
 Q(D)_0=V_{\ominus}(D)\sqcup V_{\bullet}(D)\sqcup V_{\oplus}(D),
\]
where $V_\ominus(D)$ denotes vertices labeled by $\ominus$ in \cite[Definition 4.2]{FPST}, $V_{\bullet}(D)$ is the set of vertices labeled by $\bullet$, and $V_\oplus(D)$ consists of the vertices labeled by $\oplus$. Consider the arrows in $Q(D)$ coming from the $\mbox{A}\Gamma$-diagram of $D$, which are pointing according to $\bullet\to\oplus\to\ominus\to\bullet$. Now, the potential and corresponding grading can be described as follows:\\

\begin{enumerate}

    \item {\it Potential $W(D)$}. First, one can equivalently describe $W(D)$ directly from $D$ and $Q(D)$. Indeed, by construction $Q(D)\sse\R^2$ can be considered as embedded in the plane where the divide is drawn. Then, the monomials contributing to the potential $W(D)$ are exactly the monomials coming from all the (planar) cycles of the quiver, i.e.~the bounded faces of the complement of the (graph underlying the) quiver. The monomials associated to cycles whose orientation coincides with the planar orientation contribute to $W(D)$ with a plus sign, and the others with a minus.\\

 \noindent   Alternatively, one can equivalently describe $W(D)$ by considering the plabic graph corresponding to $D$, via \cite[Definition 6.11]{FPST}. Then $W(D)$ can be declared to be the standard triangle potential associated to such plabic graph. See \cite[Definition 2.4]{Pressland19} or \cite[Remark 3.6]{BKM13} for the definition of potentials for (reduced) plabic graphs.\\
    
    \item {\it Grading $d(D)$}. Given the QP $(Q(D),W(D))$, we can choose the following arrow-grading $d_D$:
    \begin{equation}\label{eq:canonical-divide-grading}
 d_D(a)=
 \begin{cases}
 0&\mbox{ if }a:V_\ominus(D)\to V_\bullet(D)
       \text{ or }a:V_\bullet(D)\to V_\oplus(D),\\
 1&\mbox{ if }a:V_\oplus(D)\to V_\ominus(D).
 \end{cases}
\end{equation}
Since each summand of $W(D)$ contains exactly one arrow of type
\(V_\oplus\to V_\ominus\), the potential $W(D)$ has degree one, as required.
\end{enumerate}

\begin{remark}
It can be verified that triangle moves in divides, and square moves in the corresponding plabic graphs, yield QP-mutations compatible with dilations, not just quiver mutations.\hfill$\Box$
\end{remark}

\color{black}
\section{Some illustrative instances of plane curve singularities}\label{sec:illustrative_examples}

This section discusses the currently known mutation invariants of the quiver associated to a plane curve singularity, provides examples to illustrate their strength and shortcomings, and concludes with a list of all plane curve singularities with $\mu\leq16$ and their Alexander polynomials. In particular, it follows that the implication $(ii)\Longrightarrow(i)$ of the Main Conjecture holds for these singularities.

\subsection{Invariants of plane curve singularities from quiver mutation classes}\label{ssec:InvariantsSingularities_MutationClass}

Let $(C,0)$ be an isolated plane curve singularity, $D\sse\R^2$ a malleable divide for $(C,0)$ and $Q(D)$ the corresponding quiver. A central question towards understanding the implication $(ii)\Longrightarrow(i)$ in the Main Conjecture is: how do we recover invariants of $(C,0)$ from the quiver mutation class $[Q(D)]$ ?\\

\noindent With the Main Theorem now established, there are at least three sources of invariants from $[Q(D)]$:

\begin{enumerate}[label=$(\roman*)$]
    \item In \cite[Proposition 4.7]{FPST}, it is shown that the Milnor number $\mu(C,0)$ corresponds to the number of vertices of any quiver in $[Q(D)]$, equivalently the size of the exchange matrix $B(D)$ of $Q(D)$. It is also proven there that the rank of $B(D)$ recovers the number of complex local branches and the $\delta$-invariant of the singularity. Since the rank of the exchange matrix of $Q(D)$ is an invariant of the quiver mutation class $[Q(D)]$, cf.~\cite[Theorem 2.8.3]{fomin2024introductionclusteralgebraschapters}, the number of complex local branches and the $\delta$-invariant can be recovered from $[Q(D)]$.\\

    \item By either \cite[Proposition 2.10]{GL} or \cite[Theorem 4.22]{casals_positroid}, the top $a$-degree part of the HOMFLY polynomial of the link of the singularity $(C,0)$ is an invariant of the quiver mutation class $[Q(D)]$. 
    For completeness, we specify our convention for the HOMFLY polynomial, which coincides with \cite[Eq.~(2.3)]{GL}:
\[
 P_{\bigcirc}(a,z)=1,
 \qquad
 aP_{L_+}(a,z)-a^{-1}P_{L_-}(a,z)=zP_{L_0}(a,z).
\]
To be clear, by the top
$a$-degree part of a Laurent polynomial we always mean the sum of all terms whose $a$-exponent is maximal.\\

    \item By the Main Theorem, the integral monodromy module $\Z[t,t^{-1}]\circlearrowright H_1(M_f,\Z)$, up to $\Z[t,t^{-1}]$-module isomorphism, is an invariant of $[Q(D)]$. In particular, the characteristic polynomial of the monodromy of $f$, and equivalently the one-variable Alexander polynomial of $f$, are invariants of the quiver mutation class $[Q(D)]$.
\end{enumerate}

\begin{remark} Beyond \cite[Proposition 2.10]{GL} or \cite[Theorem 4.22]{casals_positroid}, there is another geometric reason explaining why the top $a$-degree coefficient of the HOMFLY polynomial of the link of the singularity $(C,0)$ is a quiver mutation invariant. Namely, by \cite[Section 4]{Rutherford_HOMFLY}, it coincides with the oriented ruling polynomial of any max-tb Legendrian approximation of the link of the singularity. Since \cite[Proposition 2.2]{Casals_LagrangianPlaneCurve} shows that such a max-tb Legendrian approximation is unique, one can use the $\mathbb{L}$-compressing system built in  \cite[Theorem 1.1]{Casals_LagrangianPlaneCurve}, along with \cite{CasalsGao}, and conclude invariance under quiver mutation.\hfill$\Box$
\end{remark}

For reference, we record these invariants in the following:

\begin{corollary}[Invariants from mutation classes]\label{cor:invariants_from_mutation_classes} Let $(C,0)$ be an isolated plane curve singularity, $D\sse\R^2$ a malleable divide for $(C,0)$ and $Q(D)$ the corresponding quiver. Then the following invariants of the singularity $(C,0)$ are determined by the quiver mutation class $[Q(D)]$:

\begin{enumerate}[label=$(\roman*)$]
    \item The Milnor number $\mu(C,0)$, the number of local branches, and the $\delta$-invariant.
    \item The top $a$-degree part of the HOMFLY polynomial of the link of $(C,0)$.
    \item The integral monodromy module $\Z[t,t^{-1}]\circlearrowright H_1(M_f,\Z)$, up to $\Z[t,t^{-1}]$-module isomorphism.
\end{enumerate}
In particular, $(iii)$ implies that the one-variable Alexander polynomial of $f$ can be recovered from $[Q(D)]$, up to product with units in $\Z[t,t^{-1}]$.\hfill$\Box$
\end{corollary}

The following subsections explore how \Cref{cor:invariants_from_mutation_classes} can be used to deduce that the implication $(ii)\Longrightarrow(i)$ in the Main Conjecture holds for malleable divides of certain classes of singularities. For instance, the integral monodromy module $\Z[t,t^{-1}]\circlearrowright H_1(M_f,\Z)$ distinguishes all 74 isolated plane curve singularities with Milnor number $\mu\leq 16$ except for one pair, cf.~\Cref{tab:main}. The singularities in that pair, $D_{12}$ and $Z_{12}$, are distinguished by the top $a$-degree coefficient of the HOMFLY polynomial, cf.~\Cref{ex:singularity_pair_3}. Thus the invariants from \Cref{cor:invariants_from_mutation_classes} uniquely determine the topological type of any isolated plane curve singularity with Milnor number $\mu$ less equal than $\mu\leq 16$. More examples and classes where that is the case are provided below.

\subsection{Examples to illustrate the invariants}\label{ssec:AlexanderPoly_Singularities}

In order to better understand how the invariants \Cref{cor:invariants_from_mutation_classes} can be used and relate to each other, the following three examples are illustrative:

\begin{example}\label{ex:singularity_pair_1} In \cite[Chapitre 3]{grima1976monodromie}, M.-C.~Grima considers the following pair of singularities
$$(C_1,0)=\{(x^{11}-y^{14})(x^{44}-y^{21})=0\},\qquad (C_2,0)=\{(x^{22}-y^{7})(x^{33}-y^{28})=0\}.$$
See also \cite[Subsection 1.3.(2)]{MichelWeber1986}. It is shown in \cite{grima1976monodromie} that their one-variable Alexander polynomials coincide, as do their rational monodromy modules. Nevertheless, their integral monodromy modules are not isomorphic, and thus their topological types differ. This is thus an example illustrating that the integral monodromy module is a stronger invariant than the one-variable Alexander polynomial, or the rational monodromy module.\hfill$\Box$
\end{example}

\begin{example}\label{ex:singularity_pair_2} Consider the following pair of singularities:
\[
(C_1,0)=\{x(y^2-x^9)=0\}
 \qquad
(C_2,0)=\{x(y^3-x^4)=0\}
\]
\noindent Each of them has two branches and they both have Milnor number $\mu(C_1,0)=\mu(C_2,0)=11$. In fact, the former is the simple singularity $D_{11}$ and the latter is the exceptional unimodal singularity $Z_{11}$. Their one-variable Alexander polynomials are:
\[
 \Delta_{C_1}(t)=\Phi_1(t)\Phi_4(t)\Phi_{20}(t)
 \quad\text{and}\quad
 \Delta_{C_2}(t)=\Phi_1(t)\Phi_3(t)\Phi_{15}(t).
\]
Thus their one-variable Alexander polynomials are distinct, as are their integral monodromy modules. In contrast, the top $a$-degree parts of the HOMFLY polynomials of their links coincide. Indeed, their links can be respectively constructed as the closures of the following two positive braids:
\[
 \beta_1=\sigma_1^9\sigma_2\sigma_1^2\sigma_2\in B_3,
 \qquad
 \beta_2=(\sigma_1\sigma_2)^4
 \sigma_3\sigma_2\sigma_1^2\sigma_2\sigma_3\in B_4. 
\]
\noindent The HOMFLY polynomials can be computed directly from these positive braid words by calculating in the Hecke algebra. For the link $L_1:=\widehat{\beta_1}$ of the singularity $(C_1,0)$ we obtain:
\[
 \begin{aligned}
 P_{L_1}(a,z)={}&a^{-11}C(z)\\
 &-a^{-13}\bigl(z^9+10z^7+36z^5+57z^3+39z+9z^{-1}\bigr)\\
 &+a^{-15}\bigl(z^5+6z^3+10z+4z^{-1}\bigr).
 \end{aligned}                                         
\]
Alternatively, one can use \cite[Theorem~3.11]{Schwartz} to obtain $P_{L_1}(a,z)$, as a $D_{11}$-quiver is a tree quiver. For the second link $L_2:=\widehat{\beta_2}$ of $(C_2,0)$ we obtain the following polynomial:
\[
 \begin{aligned}
 P_{L_2}(a,z)={}&a^{-11}C(z)\\
 &-a^{-13}\bigl(z^9+10z^7+36z^5+57z^3+39z+10z^{-1}\bigr)\\
 &+a^{-15}\bigl(z^5+6z^3+11z+6z^{-1}\bigr)-a^{-17}z^{-1}.
 \end{aligned}                                        
\]
The top $a$-exponent is of degree $-11$ in both cases and we conclude
\[
 \boxed{
 \operatorname{Top}_a P_{L_1}(a,z)
 =\operatorname{Top}_a P_{L_2}(a,z)
 =a^{-11}\bigl(z^{11}+11z^9+45z^7+85z^5+76z^3+31z+5z^{-1}\bigr).}
\]

\noindent This illustrates that it is possible for two singularities to be such that the top $a$-degree parts of the HOMFLY polynomials coincide whereas their integral monodromy modules are not isomorphic.\hfill$\Box$
\end{example}

\begin{example}\label{ex:singularity_pair_3} Consider the following pair of singularities:
$$(C_1,0)=\{x(y^2-x^{10})=0\},\quad\mbox{ and }\quad (C_2,0)=\{xy(x^2+y^3)=0\}.$$
\noindent They each have three branches and Milnor number $\mu(C_1,0)=\mu(C_2,0)=12$. The former is the simple singularity $D_{12}$ and the latter is $Z_{12}$. Their one-variable Alexander polynomials coincide, as they are both equal to
$$\Delta_{C_1}(t)=\Phi_{1}^{2}\Phi_{11}=\Delta_{C_2}(t).$$
This is in fact the pair of singularities with smallest Milnor number for which the one-variable Alexander polynomials coincide. For the record, their integral monodromy modules are both isomorphic to $\Z[t,t^{-1}]/(t-1)\oplus \Z[t,t^{-1}]/(t^{11}-1)$ as well. For the computation of the HOMFLY polynomial for $D_{12}$, we can either use \cite[Theorem~3.11]{Schwartz} or directly compute in the Hecke algebra. Since we have to use\footnote{Alternatively, one can also compute these by using SnapPy to first identify the link from a link diagram and then find the HOMFLY polynomial in a table.} the latter for $Z_{12}$ either way, here are two positive braid representatives of the corresponding links:
$$\beta_1:=\Delta_3^2\sigma_1^8=(\sigma_1\sigma_2)^3\sigma_1^8,\quad\mbox{ and }\quad \beta_2:=\Delta_4^2\sigma_1\sigma_2^2
       =(\sigma_1\sigma_2\sigma_3)^4\sigma_1\sigma_2^2$$
From here, the resulting HOMFLY polynomials read
\begin{align}\label{eq:fullD}
P(L_{D_{12}};a,z)
={}&
 a^{-12}\Bigl(
 z^{12}+12z^{10}+55z^8+121z^6+133z^4
 +71z^2+17+z^{-2}
 \Bigr)
\notag\\
&-a^{-14}\Bigl(
 z^{10}+11z^8+45z^6+85z^4
 +75z^2+27+2z^{-2}
 \Bigr)
\notag\\
&+a^{-16}\Bigl(
 z^6+7z^4+15z^2+10+z^{-2}
 \Bigr).
\end{align}
and\\
\begin{align}\label{eq:fullZ}
P(L_{Z_{12}};a,z)
={}&
 a^{-12}\Bigl(
 z^{12}+12z^{10}+55z^8+121z^6+133z^4
 +71z^2+18+2z^{-2}
 \Bigr)
\notag\\
&-a^{-14}\Bigl(
 z^{10}+11z^8+45z^6+85z^4
 +76z^2+31+5z^{-2}
 \Bigr)
\notag\\
&+a^{-16}\Bigl(
 z^6+7z^4+16z^2+14+4z^{-2}
 \Bigr)
\notag\\
&-a^{-18}(1+z^{-2}).
\end{align}
In particular, the top $a$-exponent has $a$-degree $-12$. Those coefficients differ as
\[
\boxed{
\operatorname{Top}_aP(L_{Z_{12}})-\operatorname{Top}_aP(L_{D_{12}})
=a^{-12}(1+z^{-2})\neq0.
}
\]
\noindent In conclusion, this example illustrates that it is possible for two singularities to be such that their integral monodromy modules are isomorphic whereas the top $a$-degree parts of the HOMFLY polynomials differ.\hfill$\Box$
\end{example}

\begin{remark}
For other interesting pairs of singularities, see also the examples and results in \cite{MichelWeber1986}, \cite[Section 4]{Yamamoto1984}, \cite[Example 3.2]{yoshinaga1979topological} and \cite{dubois2003seifert}.\hfill$\Box$
\end{remark}

\subsection{Cases in which the one-variable Alexander polynomial suffices}\label{ssec:one_variable_Alexander_suffices}

The one-variable Alexander polynomial of an irreducible plane curve singularity determines the complex topological type of the singularity. This is a well-known fact, see e.g.~\cite[Section 5]{Burau1932}, \cite[Lemma 2.5.1]{Le1972}, \cite{Yamamoto1984} or \cite[Section 1]{CDGZ}. In particular, as stated in the introduction, the Main Theorem shows that the quiver mutation class of any irreducible malleable divide determines the complex topological type of the singularity.

\begin{remark} This is a rigid property of algebraic knots, not shared by more general smooth knots, not even positive ones. For instance, as pointed out in \cite[Section 7]{Kalman2006}, there exist pairs of plabic fences whose associated knots are not smoothly isotopic and yet they share the same Alexander polynomial. For instance, the knots $13n_{981}$ and $13n_{1104}$ are respectively associated to the plabic fences for the positive braid words
$$\beta_1=\sigma_1^3\sigma_2^3\sigma_1^2\sigma_3\sigma_2^2\sigma_3,\qquad \beta_2=\sigma_1^3\sigma_2^2\sigma_1^2\sigma_3\sigma_2^3\sigma_3^2.$$
These are indeed distinct knots, as the former is invertible and the latter is not invertible. They are nevertheless mutant to each other, which explains geometrically why their Alexander polynomials coincide. Neither of them is an algebraic knot, and so there is no contradiction with the fact that algebraic knots are determined by their Alexander polynomial. \hfill$\Box$
\end{remark}

For reducible plane curve singularities, the one-variable Alexander polynomial does not suffice, as illustrated in \cite[Section 4.A]{Yamamoto1984} or \Cref{ex:singularity_pair_3} above. That said, there are many interesting reducible singularities for which the one-variable Alexander polynomial of the link of the singularity uniquely determines the complex topological type of the singularity. The Main Theorem thus implies that the implication $(ii)\Longrightarrow(i)$ of the Main Conjecture holds in these cases. For instance, for arbitrary Milnor number $\mu$, there is the general result \cite[Theorem 3.1]{CDGZ2} for singularities with certain symmetries. Specifically, we can apply with $r=1$ in the notation of \cite[Theorem 3.1]{CDGZ2} in order to obtain the following statement:

\begin{theorem}[\cite{CDGZ2}]\label{cor:onevariable_Alexander_recovers1}
Let $(C_1,0)$ and $(C_2,0)$ be two reduced isolated complex plane curve
singularities. Suppose that there exist two finite groups $G_1$ and $G_2$ of the same order such that $G_i$ acts on the minimal resolution graphs of $(C_i,0)$ preserving the ages of the vertices, and acts transitively on the components of each $(C_i,0)$, for each $i\in\{1,2\}$. Then
\begin{equation*}
(C_1,0)\mbox{ and } (C_2,0)\mbox{ have the same topological type }\Longleftrightarrow \Delta_{C_1}(t)=\Delta_{C_2}(t).
\end{equation*}
\end{theorem}

Note that the two symmetry groups $G_1,G_2$ in \Cref{cor:onevariable_Alexander_recovers1} do not need to be isomorphic, just have the same order. For instance, \Cref{cor:onevariable_Alexander_recovers1} can be used to detect the singularity type $\bigl(\{x^a-y^b=0\},0\bigr)$, which has $\gcd(a,b)$ branches with a transitive $\Z/\gcd(a,b))$ symmetry. Indeed, we have:

\begin{corollary}\label{cor:onevariable_Alexander_recovers2}
Consider the isolated singularity $(C_{a,b},0)=\bigl(\{x^a-y^b=0\},0\bigr)\subset (\mathbb{C}^2,0)$, $a,b\in\N$ with $d:=\gcd(a,b)$. Suppose that $(C',0)$ is a reducible isolated complex plane curve
singularity such that the set of irreducible branches of $(C',0)$ admits a transitive action of a finite group of order $d$, preserving the ages of the vertices. Then
\begin{equation*}
(C',0)\mbox{ and } (C_{a,b},0)\mbox{ have the same topological type }\Longleftrightarrow \Delta_{C'}(t)=\Delta_{C_{a,b}}(t).
\end{equation*}
\end{corollary}

As implied by \cite[Lemma 15]{Leviant2017}, the results from \cite{CDGZ2} also lead to the following sufficient criterion for reducible singularities with two branches:

\begin{corollary}[\cite{Leviant2017}]\label{cor:onevariable_Alexander_recovers3} Let $(C_1,0)$ and $(C_2,0)$ be two isolated plane curve singularities with exactly two branches each. Suppose that the two branches of $(C_1,0)$ are topologically equivalent, and same for the two branches of $(C_2,0)$. Then
\begin{equation*}
(C_1,0)\mbox{ and } (C_2,0)\mbox{ have the same topological type }\Longleftrightarrow \Delta_{C_1}(t)=\Delta_{C_2}(t).
\end{equation*}
\end{corollary}

\noindent Depending on the precise situation, there can be more criteria in the spirit of \Cref{cor:onevariable_Alexander_recovers1}, \Cref{cor:onevariable_Alexander_recovers2} and \Cref{cor:onevariable_Alexander_recovers3}, showing that the one-variable Alexander polynomial can sometimes determine the topological type of a reducible isolated singularity. We hope these three instances above and \Cref{rmk:ordinary_fold} are illustrative instances of such situations.

\begin{remark}\label{rmk:ordinary_fold}
Let $(C,0)$ be an isolated plane curve singularity and let $\mbox{ord}_{t=1}\Delta_C(t)$ be the order of vanishing of $\Delta_C(t)$ at $t=1$. It follows from \cite[Proposition 3.(i)]{Durfee1975} and \cite[Theorems 7.2 \& 10.5]{milnor1968} that the equality
$$\mbox{deg }\Delta_C(t)=(\mbox{ord}_{t=1}\Delta_C(t))^2$$
implies that $(C,0)$ is an ordinary multiple point, with $\mbox{ord}_{t=1}\Delta_C(t)+1$ smooth branches at $0$ and distinct tangent lines. Thus this is another instance where the one-variable Alexander polynomial can be used to determine the topological type of a reducible singularity.\hfill$\Box$
\end{remark}

\subsection{Plane curve singularities with $\mu\leq 16$ and their Alexander polynomials} For convenience, we have prepared \Cref{tab:main}, which is a table of all isolated plane curve singularities with Milnor number $\mu\leq16$ along with their one-variable Alexander polynomials. The list essentially follows from the tables and results from \cite{Arnold1976}.\\

\noindent The table can be used along with the Main Result for a given singularity with Milnor number $\mu\leq 16$: if the reader has a quiver mutation class from an algebraic divide with $\mu\leq 16$, the Main Result can be used to compute the Alexander polynomial -- e.g.~via \Cref{eq:euler-formula-intro} -- and then the table will uniquely determine the topological type of the singularity, unless it happens to be $D_{12}$ or $Z_{12}$. In that exceptional case, one can compute the top $a$-degree piece of HOMFLY or, if known, apply simplicity or non-simplicity and \cite[Theorem 18.3]{FPST} to deduce whether the singularity is $D_{12}$ or $Z_{12}$.\\

Each row in \Cref{tab:main} contains four pieces of data:

\begin{enumerate}
    \item {\bf Singularity}. Indicates the name of the singularity, according to \cite{Arnold1976}. Those names with an $\ast$ added are {\it simple} singularities.\\

    \item {\bf Representative}. Presents a formula for the germ $(C,0)$ representing the embedded complex topological type. Note that we have not indicated analytic moduli, as many such singularities have infinitely many analytic right-equivalence classes (with the same topological type), and instead chosen a particular representative. For instance, instead of writing $X_9$ as $x^4+y^4+ax^2y^2$ for some $a\in\C$ with $a^2\neq 4$, we simply choose a value for such parameter $a\in\C$, in this case $a=0$, and write $X_9$ as $x^4+y^4$.\\

    \item {\bf Number of branches}. Records the number of local branches of the singularity. For instance, $1$ indicates an irreducible plane curve singularity.\\

    \item {\bf Alexander polynomial}. Records the one-variable Alexander polynomial $\Delta_{C}(t)$ of the link of the singularity $(C,0)$. Recall that we define $\Delta_{C}(t)$ to be the characteristic polynomial of the monodromy. Specifically, let $M_f$ be the Milnor fiber and let
\[
h_*\colon H_1(M_f;\mathbb Z)\longrightarrow H_1(M_f;\mathbb Z)
\]
be the algebraic monodromy.  We use the convention
\begin{equation}\label{eq:DeltaConvention}
 \Delta_C(t):=\det\!\bigl(t\cdot\mbox{Id}-h_*\bigr).
\end{equation}
Equivalently this is the order of the total-linking Alexander module, with every positively oriented meridian sent to $1\in\mathbb Z$. In particular, the roots of $\Delta_C(t)$ are all roots of unity and thus $\Delta_{C}(t)$ is always a product of cyclotomic polynomials. We denote by $\Phi_n(t)$ the $n$th cyclotomic polynomial, with $\Phi_1(t)=t-1$. Note also that $\deg \Delta_C(t)=\mu(C,0)$, and if $C$ has $r$ branches then the eigenvalue $1$ has multiplicity $r-1$, so the factor $\Phi_1(t)=t-1$ occurs to exponent $r-1$.\\
\end{enumerate}

\noindent The rows of the table are organized based on Milnor number $\mu$, increasing as the table is read top-to-bottom. The only new aspect beyond applying the results from \cite{Arnold1976} is the computation of the Alexander polynomials. The following methods to compute Alexander polynomials suffice:

\begin{enumerate}[label=$(\alph*)$]
    \item For a Brieskorn plane curve singularity, 
\[
f(x,y)=x^p+y^q,\qquad\mbox{ with }\quad d:=\gcd(p,q),\quad L:=\operatorname{lcm}(p,q),
\]
it follows from \cite[Section 3]{MilnorOrlik1970} or \cite[Theorem~6.4(a)]{HertlingMase2022} that the Alexander polynomial of an algebraic $(p,q)$-torus link is
\begin{equation}\label{eq:Brieskorn}
\Delta_{p,q}(t)
=
\frac{(t^L-1)^d(t-1)}
{(t^p-1)(t^q-1)}.
\end{equation}
This concludes all the $A_k$-singularities, as well as $E_6,E_8,E_{12},E_{14},W_{12},J_{3,0},W_{1,0}$ and the chosen $N_{16}$ representative. For example,
\[
\Delta_{N_{16}}(t)
=
\frac{(t^5-1)^5(t-1)}{(t^5-1)^2}
=(t-1)(t^5-1)^3
=\Phi_1^4\Phi_5^3.
\]

  \item A second method is to use the formula for the monodromy zeta function in terms of Newton diagrams as in \cite[Theorem 4.1]{Varchenko1976}, or A'Campo's resolution formula in \cite{ACampo1975}. For instance, for the $D_k$-series, we can write $g:=\gcd(2,k-2)$ and then the Newton edge joining $(2,1)$ to $(0,k-1)$ has lattice length $g$ and Newton multiplicity $2(k-1)/g$. Therefore
  \begin{equation}\label{eq:D}
\Delta_{D_k}(t)
=(t-1)\frac{(t^{2(k-1)/g}-1)^g}{t^{k-1}-1}
=
\begin{cases}
(t-1)(t^{k-1}-1),& k\ \text{even},\\[1mm]
(t-1)(t^{k-1}+1),& k\ \text{odd}.
\end{cases}
\end{equation}
which can each be factorized in terms of cyclotomic polynomials as in the table. For the hyperbolic $T_{2,q,r}$ representatives, which covers $J_{2,p}$, $X_{1,p}$ and $Y^1_{r,s}$, this is literally the first example in \cite{Varchenko1976}. Specifically, if we write $g_q:=\gcd(q-2,2)$ and $g_r:=\gcd(r-2,2)$, then the two compact Newton edges have lattice lengths $g_q,g_r$ and Newton multiplicities $2q/g_q,2r/g_r$. Then \cite[Theorem 4.1]{Varchenko1976} gives
\begin{equation}\label{eq:T}
\Delta_{T_{2,q,r}}(t)
=
(t-1)
\frac{
(t^{2q/g_q}-1)^{g_q}
(t^{2r/g_r}-1)^{g_r}
}
{(t^q-1)(t^r-1)}.
\end{equation}
Since the number of branches is $r(C)=g_q+g_r$, by factoring $t^m-1=\prod_{d\mid m}\Phi_d(t)$ we obtained the Alexander polynomials displayed for $J_{2,p}$, $X_{1,p}$ and $Y^1_{r,s}$ in \Cref{tab:main}.\\

All the remaining singularities $f:\C^2\lr\C$, except for $W_{1,1}^*$ which has a degenerate Newton polygon, can be computed via \cite{Varchenko1976}. Specifically, for a planar representative whose Newton boundary is nondegenerate, \cite{Varchenko1976} gives the zeta function (and thus the Alexander polynomial) directly from the compact Newton edges. In dimension two, if an edge $e$ has primitive positive inward normal $(a_e,b_e)$, with $a_e i+b_ej=N_e$, and lattice length $\ell_e$, the specialization gives
\begin{equation}\label{eq:Newton2D}
\Delta_f(t)
=(1-t)\,
\frac{\prod_e(1-t^{N_e})^{\ell_e}}
{\prod_{\text{vertices }(m,0)}(1-t^m)
 \prod_{\text{vertices }(0,n)}(1-t^n)}.
\end{equation}
This gives the Alexander polynomials for $E_7,E_{13},Z_{11},Z_{12},Z_{13},W_{13},Z_{1,0},Z_{1,1}$ and $W_{1,1}$. For instance, consider
\[
Z_{1,1}\quad \mbox{ with }\quad  f(x,y)=x^3y+x^2y^3+y^8.
\]
Then the compact edges are $(3,1)\to(2,3)$ and $(2,3)\to(0,8)$, with Newton multiplicities $7$ and $16$ and lattice lengths $1,1$.  Thus
\[
\Delta_{Z_{1,1}}(t)
=(1-t)\frac{(1-t^7)(1-t^{16})}{1-t^8}
=\Phi_1^2\Phi_7\Phi_{16}.
\]

\item For the remaining singularity $W^*_{1,1}$ the principal Newton face is degenerate. We can instead use a direct Newton--Puiseux computation that gives a single branch with characteristic exponents $(4;6,7)$. Indeed, we can directly set
\[
y=t^4,\qquad x=\alpha t^6+\beta t^7+\cdots,
\]
and then the equations at orders $12$ and $26$ give $\alpha^2+1=0$ and $-4\beta^2+\alpha=0$, hence $\beta\ne0$.  The value semigroup $S$ is therefore $S=\langle4,6,13\rangle$ and \cite[Theorem 1]{CDGZ} thus gives
\[
\sum_{s\in S}t^s
=
\frac{(1-t^{12})(1-t^{26})}
{(1-t^4)(1-t^6)(1-t^{13})},
\]
and hence
\[
\Delta_{W^*_{1,1}}(t)
=(1-t)\sum_{s\in S}t^s
=\Phi_{12}\Phi_{26}.
\]
\end{enumerate}

\noindent Note that these computations lead to the fact that any isolated plane curve singularity with $\mu\leq16$ is uniquely determined by its one-variable Alexander polynomial, except in the case of $D_{12}$ and $Z_{12}$. As discussed above, the other invariants already tell $D_{12}$ and $Z_{12}$ apart, so the invariants of \Cref{cor:invariants_from_mutation_classes} uniquely determine the topological type of the singularity if it has $\mu\leq16$.

\small
\setlength{\tabcolsep}{5pt}
\begin{longtable}{
  >{\raggedright\arraybackslash}p{0.22\linewidth}
  >{\raggedright\arraybackslash}p{0.25\linewidth}
  >{\raggedright\arraybackslash}p{0.2\linewidth}
  >{\raggedright\arraybackslash}p{0.10\linewidth}}
\caption{All isolated complex plane-curve singularity types with $\mu\le16$.  An asterisk denotes a simple singularity.}
\label{tab:main}\\
\toprule
\textbf{Singularity} & \textbf{Representative} & \textbf{No.~Branches\qquad} & \textbf{$\Delta_C(t)$}\\
\midrule
\endfirsthead
\multicolumn{4}{c}{\textit{Table \ref{tab:main} continued}}\\
\toprule
\textbf{Singularity} & \textbf{Representative} & \textbf{No.~Branches\qquad} & \textbf{$\Delta_C(t)$}\\
\midrule
\endhead
\midrule
\multicolumn{4}{r}{\textit{Continued on next page}}\\
\endfoot
\bottomrule
\endlastfoot
\multicolumn{4}{l}{\textbf{$\mu=1$}}\\
\midrule
$A_{1}^{\ast}$ & $x^2+y^{2}=0$ & 2 & $\Phi_{1}$ \\
\addlinespace[1.5pt]
\midrule
\multicolumn{4}{l}{\textbf{$\mu=2$}}\\
\midrule
$A_{2}^{\ast}$ & $x^2+y^{3}=0$ & 1 & $\Phi_{6}$ \\
\addlinespace[1.5pt]
\midrule
\multicolumn{4}{l}{\textbf{$\mu=3$}}\\
\midrule
$A_{3}^{\ast}$ & $x^2+y^{4}=0$ & 2 & $\Phi_{1}\Phi_{4}$ \\
\addlinespace[1.5pt]
\midrule
\multicolumn{4}{l}{\textbf{$\mu=4$}}\\
\midrule
$A_{4}^{\ast}$ & $x^2+y^{5}=0$ & 1 & $\Phi_{10}$ \\
$D_{4}^{\ast}$ & $x^2y+y^{3}=0$ & 3 & $\Phi_{1}^{2}\Phi_{3}$ \\
\addlinespace[1.5pt]
\midrule
\multicolumn{4}{l}{\textbf{$\mu=5$}}\\
\midrule
$A_{5}^{\ast}$ & $x^2+y^{6}=0$ & 2 & $\Phi_{1}\Phi_{3}\Phi_{6}$ \\
$D_{5}^{\ast}$ & $x^2y+y^{4}=0$ & 2 & $\Phi_{1}\Phi_{8}$ \\
\addlinespace[1.5pt]
\midrule
\multicolumn{4}{l}{\textbf{$\mu=6$}}\\
\midrule
$A_{6}^{\ast}$ & $x^2+y^{7}=0$ & 1 & $\Phi_{14}$ \\
$D_{6}^{\ast}$ & $x^2y+y^{5}=0$ & 3 & $\Phi_{1}^{2}\Phi_{5}$ \\
$E_6^{\ast}$ & $x^3+y^4=0$ & 1 & $\Phi_{6}\Phi_{12}$ \\
\addlinespace[1.5pt]
\midrule
\multicolumn{4}{l}{\textbf{$\mu=7$}}\\
\midrule
$A_{7}^{\ast}$ & $x^2+y^{8}=0$ & 2 & $\Phi_{1}\Phi_{4}\Phi_{8}$ \\
$D_{7}^{\ast}$ & $x^2y+y^{6}=0$ & 2 & $\Phi_{1}\Phi_{4}\Phi_{12}$ \\
$E_7^{\ast}$ & $x^3+xy^3=0$ & 2 & $\Phi_{1}\Phi_{9}$ \\
\addlinespace[1.5pt]
\midrule
\multicolumn{4}{l}{\textbf{$\mu=8$}}\\
\midrule
$A_{8}^{\ast}$ & $x^2+y^{9}=0$ & 1 & $\Phi_{6}\Phi_{18}$ \\
$D_{8}^{\ast}$ & $x^2y+y^{7}=0$ & 3 & $\Phi_{1}^{2}\Phi_{7}$ \\
$E_8^{\ast}$ & $x^3+y^5=0$ & 1 & $\Phi_{15}$ \\
\addlinespace[1.5pt]
\midrule
\multicolumn{4}{l}{\textbf{$\mu=9$}}\\
\midrule
$A_{9}^{\ast}$ & $x^2+y^{10}=0$ & 2 & $\Phi_{1}\Phi_{5}\Phi_{10}$ \\
$D_{9}^{\ast}$ & $x^2y+y^{8}=0$ & 2 & $\Phi_{1}\Phi_{16}$ \\
$X_9=X_{1,0}$ & $x^4+y^4=0$ & 4 & $\Phi_{1}^{3}\Phi_{2}^{2}\Phi_{4}^{2}$ \\
\addlinespace[1.5pt]
\midrule
\multicolumn{4}{l}{\textbf{$\mu=10$}}\\
\midrule
$A_{10}^{\ast}$ & $x^2+y^{11}=0$ & 1 & $\Phi_{22}$ \\
$D_{10}^{\ast}$ & $x^2y+y^{9}=0$ & 3 & $\Phi_{1}^{2}\Phi_{3}\Phi_{9}$ \\
$J_{10}=J_{2,0}$ & $x^3+x^2y^2+y^6=0$ & 3 & $\Phi_{1}^{2}\Phi_{2}^{2}\Phi_{3}\Phi_{6}^{2}$ \\
$X_{1,1}=T_{2,4,5}$ & $x^4+x^2y^2+y^{5}=0$ & 3 & $\Phi_{1}^{2}\Phi_{2}^{2}\Phi_{4}\Phi_{10}$ \\
\addlinespace[1.5pt]
\midrule
\multicolumn{4}{l}{\textbf{$\mu=11$}}\\
\midrule
$A_{11}^{\ast}$ & $x^2+y^{12}=0$ & 2 & $\Phi_{1}\Phi_{3}\Phi_{4}\Phi_{6}\Phi_{12}$ \\
$D_{11}^{\ast}$ & $x^2y+y^{10}=0$ & 2 & $\Phi_{1}\Phi_{4}\Phi_{20}$ \\
$J_{2,1}=T_{2,3,7}$ & $x^3+x^2y^2+y^{7}=0$ & 2 & $\Phi_{1}\Phi_{2}^{2}\Phi_{6}\Phi_{14}$ \\
$X_{1,2}=T_{2,4,6}$ & $x^4+x^2y^2+y^{6}=0$ & 4 & $\Phi_{1}^{3}\Phi_{2}^{2}\Phi_{3}\Phi_{4}\Phi_{6}$ \\
$Y^1_{1,1}=T_{2,5,5}$ & $x^{5}+x^2y^2+y^{5}=0$ & 2 & $\Phi_{1}\Phi_{2}^{2}\Phi_{10}^{2}$ \\
$Z_{11}$ & $x^3y+y^5=0$ & 2 & $\Phi_{1}\Phi_{3}\Phi_{15}$ \\
\addlinespace[1.5pt]
\midrule
\multicolumn{4}{l}{\textbf{$\mu=12$}}\\
\midrule
$A_{12}^{\ast}$ & $x^2+y^{13}=0$ & 1 & $\Phi_{26}$ \\
$D_{12}^{\ast}$ & $x^2y+y^{11}=0$ & 3 & $\Phi_{1}^{2}\Phi_{11}$ \\
$E_{12}$ & $x^3+y^7=0$ & 1 & $\Phi_{21}$ \\
$J_{2,2}=T_{2,3,8}$ & $x^3+x^2y^2+y^{8}=0$ & 3 & $\Phi_{1}^{2}\Phi_{2}^{2}\Phi_{4}\Phi_{6}\Phi_{8}$ \\
$X_{1,3}=T_{2,4,7}$ & $x^4+x^2y^2+y^{7}=0$ & 3 & $\Phi_{1}^{2}\Phi_{2}^{2}\Phi_{4}\Phi_{14}$ \\
$Y^1_{2,1}=T_{2,6,5}$ & $x^{6}+x^2y^2+y^{5}=0$ & 3 & $\Phi_{1}^{2}\Phi_{2}^{2}\Phi_{3}\Phi_{6}\Phi_{10}$ \\
$Z_{12}$ & $x^3y+xy^4=0$ & 3 & $\Phi_{1}^{2}\Phi_{11}$ \\
$W_{12}$ & $x^4+y^5=0$ & 1 & $\Phi_{10}\Phi_{20}$ \\
\addlinespace[1.5pt]
\midrule
\multicolumn{4}{l}{\textbf{$\mu=13$}}\\
\midrule
$A_{13}^{\ast}$ & $x^2+y^{14}=0$ & 2 & $\Phi_{1}\Phi_{7}\Phi_{14}$ \\
$D_{13}^{\ast}$ & $x^2y+y^{12}=0$ & 2 & $\Phi_{1}\Phi_{8}\Phi_{24}$ \\
$E_{13}$ & $x^3+xy^5=0$ & 2 & $\Phi_{1}\Phi_{5}\Phi_{15}$ \\
$J_{2,3}=T_{2,3,9}$ & $x^3+x^2y^2+y^{9}=0$ & 2 & $\Phi_{1}\Phi_{2}^{2}\Phi_{6}^{2}\Phi_{18}$ \\
$X_{1,4}=T_{2,4,8}$ & $x^4+x^2y^2+y^{8}=0$ & 4 & $\Phi_{1}^{3}\Phi_{2}^{2}\Phi_{4}^{2}\Phi_{8}$ \\
$Y^1_{3,1}=T_{2,7,5}$ & $x^{7}+x^2y^2+y^{5}=0$ & 2 & $\Phi_{1}\Phi_{2}^{2}\Phi_{10}\Phi_{14}$ \\
$Y^1_{2,2}=T_{2,6,6}$ & $x^{6}+x^2y^2+y^{6}=0$ & 4 & $\Phi_{1}^{3}\Phi_{2}^{2}\Phi_{3}^{2}\Phi_{6}^{2}$ \\
$Z_{13}$ & $x^3y+y^6=0$ & 2 & $\Phi_{1}\Phi_{9}\Phi_{18}$ \\
$W_{13}$ & $x^4+xy^4=0$ & 2 & $\Phi_{1}\Phi_{8}\Phi_{16}$ \\
\addlinespace[1.5pt]
\midrule
\multicolumn{4}{l}{\textbf{$\mu=14$}}\\
\midrule
$A_{14}^{\ast}$ & $x^2+y^{15}=0$ & 1 & $\Phi_{6}\Phi_{10}\Phi_{30}$ \\
$D_{14}^{\ast}$ & $x^2y+y^{13}=0$ & 3 & $\Phi_{1}^{2}\Phi_{13}$ \\
$E_{14}$ & $x^3+y^8=0$ & 1 & $\Phi_{6}\Phi_{12}\Phi_{24}$ \\
$J_{2,4}=T_{2,3,10}$ & $x^3+x^2y^2+y^{10}=0$ & 3 & $\Phi_{1}^{2}\Phi_{2}^{2}\Phi_{5}\Phi_{6}\Phi_{10}$ \\
$X_{1,5}=T_{2,4,9}$ & $x^4+x^2y^2+y^{9}=0$ & 3 & $\Phi_{1}^{2}\Phi_{2}^{2}\Phi_{4}\Phi_{6}\Phi_{18}$ \\
$Y^1_{4,1}=T_{2,8,5}$ & $x^{8}+x^2y^2+y^{5}=0$ & 3 & $\Phi_{1}^{2}\Phi_{2}^{2}\Phi_{4}\Phi_{8}\Phi_{10}$ \\
$Y^1_{3,2}=T_{2,7,6}$ & $x^{7}+x^2y^2+y^{6}=0$ & 3 & $\Phi_{1}^{2}\Phi_{2}^{2}\Phi_{3}\Phi_{6}\Phi_{14}$ \\
\addlinespace[1.5pt]
\midrule
\multicolumn{4}{l}{\textbf{$\mu=15$}}\\
\midrule
$A_{15}^{\ast}$ & $x^2+y^{16}=0$ & 2 & $\Phi_{1}\Phi_{4}\Phi_{8}\Phi_{16}$ \\
$D_{15}^{\ast}$ & $x^2y+y^{14}=0$ & 2 & $\Phi_{1}\Phi_{4}\Phi_{28}$ \\
$J_{2,5}=T_{2,3,11}$ & $x^3+x^2y^2+y^{11}=0$ & 2 & $\Phi_{1}\Phi_{2}^{2}\Phi_{6}\Phi_{22}$ \\
$X_{1,6}=T_{2,4,10}$ & $x^4+x^2y^2+y^{10}=0$ & 4 & $\Phi_{1}^{3}\Phi_{2}^{2}\Phi_{4}\Phi_{5}\Phi_{10}$ \\
$Y^1_{5,1}=T_{2,9,5}$ & $x^{9}+x^2y^2+y^{5}=0$ & 2 & $\Phi_{1}\Phi_{2}^{2}\Phi_{6}\Phi_{10}\Phi_{18}$ \\
$Y^1_{4,2}=T_{2,8,6}$ & $x^{8}+x^2y^2+y^{6}=0$ & 4 & $\Phi_{1}^{3}\Phi_{2}^{2}\Phi_{3}\Phi_{4}\Phi_{6}\Phi_{8}$ \\
$Y^1_{3,3}=T_{2,7,7}$ & $x^{7}+x^2y^2+y^{7}=0$ & 2 & $\Phi_{1}\Phi_{2}^{2}\Phi_{14}^{2}$ \\
$Z_{1,0}$ & $x^3y+y^7=0$ & 4 & $\Phi_{1}^{3}\Phi_{7}^{2}$ \\
$W_{1,0}$ & $x^4+y^6=0$ & 2 & $\Phi_{1}\Phi_{3}\Phi_{4}\Phi_{6}\Phi_{12}^{2}$ \\
\addlinespace[1.5pt]
\midrule
\multicolumn{4}{l}{\textbf{$\mu=16$}}\\
\midrule
$A_{16}^{\ast}$ & $x^2+y^{17}=0$ & 1 & $\Phi_{34}$ \\
$D_{16}^{\ast}$ & $x^2y+y^{15}=0$ & 3 & $\Phi_{1}^{2}\Phi_{3}\Phi_{5}\Phi_{15}$ \\
$J_{2,6}=T_{2,3,12}$ & $x^3+x^2y^2+y^{12}=0$ & 3 & $\Phi_{1}^{2}\Phi_{2}^{2}\Phi_{3}\Phi_{4}\Phi_{6}^{2}\Phi_{12}$ \\
$J_{3,0}$ & $x^3+y^9=0$ & 3 & $\Phi_{1}^{2}\Phi_{3}\Phi_{9}^{2}$ \\
$X_{1,7}=T_{2,4,11}$ & $x^4+x^2y^2+y^{11}=0$ & 3 & $\Phi_{1}^{2}\Phi_{2}^{2}\Phi_{4}\Phi_{22}$ \\
$Y^1_{6,1}=T_{2,10,5}$ & $x^{10}+x^2y^2+y^{5}=0$ & 3 & $\Phi_{1}^{2}\Phi_{2}^{2}\Phi_{5}\Phi_{10}^{2}$ \\
$Y^1_{5,2}=T_{2,9,6}$ & $x^{9}+x^2y^2+y^{6}=0$ & 3 & $\Phi_{1}^{2}\Phi_{2}^{2}\Phi_{3}\Phi_{6}^{2}\Phi_{18}$ \\
$Y^1_{4,3}=T_{2,8,7}$ & $x^{8}+x^2y^2+y^{7}=0$ & 3 & $\Phi_{1}^{2}\Phi_{2}^{2}\Phi_{4}\Phi_{8}\Phi_{14}$ \\
$Z_{1,1}$ & $y(x^3+x^2y^2+y^7)=0$ & 3 & $\Phi_{1}^{2}\Phi_{7}\Phi_{16}$ \\
$W_{1,1}$ & $x^4+x^2y^3+y^7=0$ & 3 & $\Phi_{1}^{2}\Phi_{3}\Phi_{6}\Phi_{7}\Phi_{12}$ \\
$W^*_{1,1}$ & $(x^2+y^3)^2+xy^5=0$ & 1 & $\Phi_{12}\Phi_{26}$ \\
$N_{16}$ & $x^5+y^5=0$ & 5 & $\Phi_{1}^{4}\Phi_{5}^{3}$ \\
\end{longtable}

\section{Further comments and examples}\label{sec:comments_and_examples}

In this last section we compile a few comments related to \cite{FPST} and provide an example further illustrating the techniques developed in Sections \ref{sec:graded-QP}, \ref{sec:euler}, \ref{sec:dilation_plabic_fence} and \ref{sec:polynomial_plabic_fence}.

\subsection{A counterexample to \cite[Conjecture 6.17]{FPST}} The article \cite{FPST} contains other interesting conjectures. For instance, it is proven in \cite[Proposition 6.16]{FPST} that if two plabic graphs are move-and-switch equivalent to each other, then their associated quivers are mutation equivalent. It is then stated in \cite[Conjecture 6.17]{FPST} that it is conjectured to be an equivalence, i.e.~that two plabic graphs are move-and-switch equivalent if and only if their associated quivers are mutation equivalent.\\

\begin{figure}
\centering
\begin{tikzpicture}[scale=.92,
  bvertex/.style={circle,fill=black,inner sep=2.25pt},
  wvertex/.style={circle,fill=white,draw=black,line width=.6pt,inner sep=2.25pt},
  every path/.style={line width=.8pt}]
\coordinate (v0) at (0,3.2);
\coordinate (v1) at (0,2.05);
\coordinate (v2) at (0,.9);
\coordinate (v3) at (0,-.9);
\coordinate (v4) at (0,-2.05);
\coordinate (v5) at (0,-3.2);
\draw (v0) .. controls (-2.6,1.9) and (-2.6,-1.9) .. (v5);
\draw (v0) .. controls (2.6,1.9) and (2.6,-1.9) .. (v5);
\draw (v0)--(v1);
\draw (v1) .. controls (-.85,1.75) and (-.85,1.2) .. (v2);
\draw (v1) .. controls (.85,1.75) and (.85,1.2) .. (v2);
\draw (v2)--(v3);
\draw (v3) .. controls (-.85,-1.2) and (-.85,-1.75) .. (v4);
\draw (v3) .. controls (.85,-1.2) and (.85,-1.75) .. (v4);
\draw (v4)--(v5);
\node[bvertex] at (v0) {};
\node[wvertex] at (v1) {};
\node[bvertex] at (v2) {};
\node[bvertex] at (v3) {};
\node[wvertex] at (v4) {};
\node[bvertex] at (v5) {};
\node at (-1.35,0) {$F_L$};
\node at (1.35,0) {$F_R$};
\node at (0,1.47) {$F_T$};
\node at (0,-1.47) {$F_B$};
\end{tikzpicture}\qquad\qquad
\begin{tikzpicture}[scale=1.03,
  bvertex/.style={circle,fill=black,inner sep=2.25pt},
  wvertex/.style={circle,fill=white,draw=black,line width=.6pt,inner sep=2.25pt},
  every path/.style={line width=.8pt}]
\coordinate (u1) at (0,2);
\coordinate (u2) at (1.73,1);
\coordinate (u3) at (1.73,-1);
\coordinate (u4) at (0,-2);
\coordinate (u5) at (-1.73,-1);
\coordinate (u6) at (-1.73,1);
\draw (u1)--(u2)--(u3)--(u4)--(u5)--(u6)--cycle;
\draw (u1) .. controls (1.0,2.35) and (2.15,2.65) .. (u2);
\draw (u3) .. controls (1.95,-1.85) and (.8,-3.45) .. (u4);
\draw (u5) .. controls (-2.85,-.15) and (-2.75,.15) .. (u6);
\node[bvertex] at (u1) {};
\node[wvertex] at (u2) {};
\node[bvertex] at (u3) {};
\node[wvertex] at (u4) {};
\node[bvertex] at (u5) {};
\node[wvertex] at (u6) {};
\node at (0,0) {$F_0$};
\node at (1.18,1.72) {$F_1$};
\node at (.98,-1.85) {$F_2$};
\node at (-2.1,0) {$F_3$};
\end{tikzpicture}
\caption{Two plabic graphs, $G_0$ on the left, and $G_1$ on the right.}
\label{fig:plabic_graphs_Shapiro}
\end{figure}
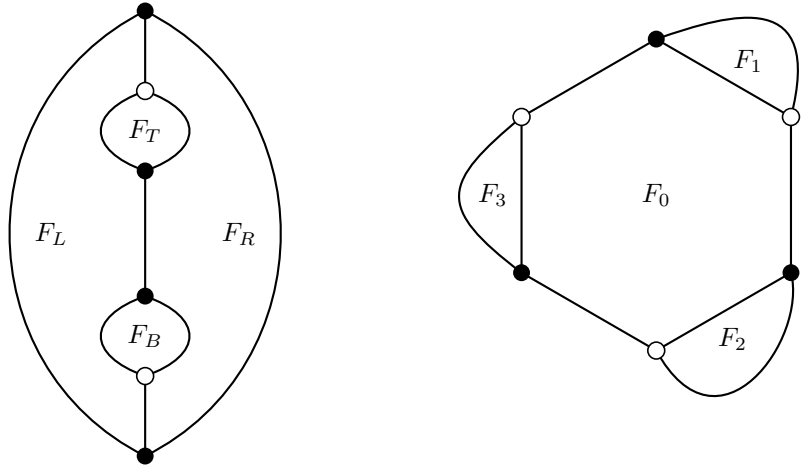

Consider the two plabic graphs in \Cref{fig:plabic_graphs_Shapiro}. The plabic graph $G_0$ in \Cref{fig:plabic_graphs_Shapiro} (left) is the one drawn in \cite[Figure 8.(b)]{GL}. The corresponding quivers $Q(G_0)$ and $Q(G_1)$ are

\[
Q(G_0)\;=\;
\begin{tikzcd}[row sep=3.2em,column sep=3.8em]
F_T \arrow[r] & F_R \arrow[d] \\
F_L \arrow[u] & F_B \arrow[l]
\end{tikzcd}
\quad\quad\&\quad\quad
Q(G_1)\;=\;
\begin{tikzcd}[row sep=3.2em,column sep=3.8em]
F_1 \arrow[dr] & & F_2 \arrow[dl] \\
& F_0 & \\
& F_3 \arrow[u] &
\end{tikzcd}
\]
Note that $Q(G_1)$ is an orientation of the $D_4$-Dynkin diagram. By mutating $Q(G_0)$ first at $F_T$ and then at $F_R$ and $F_L$, say, it follows that $Q(G_0)$ is mutation equivalent to $Q(G_1)$.\\

Now, for a plabic graph $\mathbb{P}$, the local moves in \cite[Definition 6.2]{FPST} preserve the smooth type of the associated link $L(\mathbb{P})$.\footnote{See \cite[Section 9]{FPST} for the definition of $L(\mathbb{P})$. In brief, it is the smooth link obtained by satelliting along the unknot the Legendrian link in the unit cotangent bundle of the plane whose front projection is the alternating strand diagram of the plabic graph.} In contrast, the switch operation on a plabic graph $\mathbb{P}$, as defined in \cite[Definition 6.15]{FPST}, has the effect of changing the associated link $L(\mathbb{P})$ by a positive Conway mutation. See \cite{Conway1970} and \cite[Section~1]{Morton2015} for the notion of a tangle mutation. Now, by \cite[Statement~11]{Hoste1986}, the HOMFLY polynomial of a link is invariant under positive Conway
mutation, see also \cite[Proposition~11]{LickorishMillett1987} and ~\cite{FreydEtAl1985}. In particular, \cite[Conjecture 9.6]{GL} is true. Therefore, if $G_0$ and $G_1$ were move-and-switch equivalent, their links $L(G_0)$ and $L(G_1)$ would have identical HOMFLY polynomials.\\

For $G_0$, the associated HOMFLY polynomial of the plabic link is computed explicitly \cite[Section~9.3]{GL}:
\begin{equation}\label{eq:P0}
\begin{aligned}
\HOMFLY(L(G_0);a,z)
={}&a^{-4}(z^4+4z^2+2)+a^{-8}z^{-2}-a^{-10}(z^2+3+2z^{-2})
+a^{-12}(1+z^{-2}).
\end{aligned}
\end{equation}

\noindent Here $\HOMFLY(L;a,z)$ denotes the oriented HOMFLY polynomial normalized by
\begin{equation}\label{eq:homfly-skein}
a\HOMFLY(L_+;a,z)-a^{-1}\HOMFLY(L_-;a,z)=z\HOMFLY(L_0;a,z),
\qquad
\HOMFLY(\bigcirc;a,z)=1.
\end{equation}

\noindent To compute the HOMFLY polynomial associated to $G_1$ we use \cite{Schwartz}, as $G_1$ is simple. Specifically, \cite[Theorem~3.11]{Schwartz} shows that, for a connected plabic graph whose face quiver is a forest, the HOMFLY polynomial of the plabic link is the forest-quiver polynomial. Note that this polynomial depends only on the underlying unoriented forest, which in our case is a $D_4$-Dynkin diagram. Therefore \cite[Example~4.8]{Schwartz} yields
\begin{equation}\label{eq:P1}
\begin{aligned}
\HOMFLY(L(G_1);a,z)
={}&a^{-4}(z^4+4z^2+3+z^{-2})-a^{-6}(z^2+3+2z^{-2})
+a^{-8}z^{-2}.
\end{aligned}
\end{equation}

\noindent Since the coefficient of $a^{-6}$ in \eqref{eq:P1} is non-zero, and it is zero in \eqref{eq:P0}, these two HOMFLY polynomials are distinct. Thus $G_0$ and $G_1$ are not move-and-switch equivalent. We conclude that it is possible for two plabic graphs to have mutation equivalent quivers while not being move-and-switch equivalent, contradicting \cite[Conjecture 6.17]{FPST}.

\subsection{A detailed example}\label{ssec:examples}

Let us consider the irreducible isolated plane curve singularity
\[
 C:=\{(x,y)\in\C^2:x^3-y^4=0\}.
\]
The singularity is at the origin $(x,y)=(0,0)$: it is the $E_6$ simple singularity and its link is the $(3,4)$-torus knot. See also \cite[Figure 4]{FPST} for two divides associated to two real Morsifications of this singularity. Let us consider the top divide for $E_6$ in \cite[Figure 4]{FPST}. Its associated quiver $Q$ is

\[
\begin{tikzcd}[column sep=large, row sep=large]
1
  \arrow[r, "a"]
  \arrow[d, "e"']
&
2
  \arrow[r, "b"]
  \arrow[d, "f"']
&
3
  \arrow[d, "g"]
\\
4
  \arrow[r, "c"']
&
5
  \arrow[r, "d"']
  \arrow[ul, "p"]
&
6
  \arrow[ul, "q"]
\end{tikzcd}
\]
\noindent The associated non-degenerate potential is
\begin{equation}\label{eq:E6-W}
 W=pfa-pce+qgb-qdf.
\end{equation}
An arrow-grading $d:Q_1\lr\Z$ for which $W$ in \eqref{eq:E6-W} is homogeneous of degree 1 is
\begin{equation}\label{eq:E6-d}
 d(p)=d(q)=1,
 \qquad
 d(a)=d(b)=d(c)=d(d)=d(e)=d(f)=d(g)=0.
\end{equation}

Consider the bigraded Ginzburg dg algebra $\Gamma(Q,W,d)$ as in \Cref{def:bigradedGinzburg} and the derived category $\D^{\gr}_{\fd}(\Gamma)$ as in \Cref{def:categories_over_bigraded}. The graded Euler pairing expressed in the simple basis reads
\begin{equation}\label{eq:E6-E}
 \Eul_{Q,W,d}(t)=
 \begin{pmatrix}
 1-t&-1&0&-1&1&0\\
 t&1-t&-1&0&-1&1\\
 0&t&1-t&0&0&-1\\
 t&0&0&1-t&-1&0\\
 -t&t&0&t&1-t&-1\\
 0&-t&t&0&t&1-t
 \end{pmatrix},
\end{equation}
which can be computed directly or via \Cref{prop:euler-formula}. In particular, its determinant is
\begin{equation}\label{eq:E6-det}
 \det\Eul_{Q,W,d}(t)
 =(t^2-t+1)(t^4-t^2+1)
 =t^6-t^5+t^3-t+1.
\end{equation}
Note that \eqref{eq:E6-det} coincides with the Alexander polynomial $\Delta_{T(3,4)}(t)$ of the $(3,4)$-torus knot, in accordance with \Cref{cor:cokernel_module_is_Alexander_algebraic_links}.\\

Let us now mutate the quiver $Q$ at the vertex $1$. The only incoming arrow is \(p:5\to1\), and there are two outgoing arrows \(a:1\to2\) and \(e:1\to4\). In order to perform the QP mutation of $(Q,W)$, introduce the composite arrows
\[
 x=[ap]:5\to2,\qquad y=[ep]:5\to4,
\]
and reversed arrows
\[
 \alpha=a^*:2\to1,\qquad
 \varepsilon=e^*:4\to1,\qquad
 \pi=p^*:1\to5.
\]
Then the premutated potential is cyclically equivalent to
\begin{equation}\label{eq:E6-premut}
 \widetilde W=xf-yc+qgb-qdf+\pi\alpha x+\pi\varepsilon y.
\end{equation}
Now we need to extract the reduced part. For that let us set
\[
 x_0:=x-qd,\qquad c_0:=c-\pi\varepsilon,
 \qquad f_0:=f+\pi\alpha.
\]
Then we obtain the cyclic equivalence
\[
 \widetilde W\sim_{\cyc}
 x_0f_0-yc_0+qgb+\pi\alpha qd,
\]
so that $x_0f_0$ and $yc_0$ are trivial pairs. After deleting these two trivial pairs, the quiver $Q'$ for the reduced QP $(Q',W')$ reads
\[
\begin{tikzcd}[column sep=large, row sep=large]
1
  \arrow[dr, "\pi"]
&
2
  \arrow[l, "\alpha"']
  \arrow[r, "b"]
&
3
  \arrow[d, "g"]
\\
4
  \arrow[u, "\varepsilon"]
&
5
  \arrow[r, "d"']
&
6
  \arrow[ul, "q"']
\end{tikzcd}
\]
and the corresponding potential $W'$ reads
\begin{equation}\label{eq:E6-mut-W}
 W'=qgb+\pi\alpha qd.
\end{equation}
The corresponding mutated grading is
\begin{equation}\label{eq:E6-mut-d}
 d'(q)=1,\qquad d'(u)=0\quad(u\ne q),
\end{equation}
which indeed makes $W'$ homogeneous of degree one. Note that $Q'$ is {\it not} the quiver associated to any divide, as the 4-cycle $\pi\alpha qd$ would violate the rule $\bullet\to\oplus\to\ominus\to\bullet$ in \cite[Definition 4.2]{FPST}. This illustrates that quiver mutation typically forces one out of the class of quivers with divides. This is what makes the Main Conjecture challenging. Finally, the graded Euler pairing for $(Q',W',d')$ in the simple basis reads
\begin{equation}\label{eq:E6-mut-E}
 \begin{pmatrix}
 1-t&t&0&t&-1&0\\
 -1&1-t&-1&0&0&1\\
 0&t&1-t&0&0&-1\\
 -1&0&0&1-t&0&0\\
 t&0&0&0&1-t&-1\\
 0&-t&t&0&t&1-t
 \end{pmatrix}.
\end{equation}
Its determinant is indeed still \(t^6-t^5+t^3-t+1\), as in \Cref{eq:E6-det}.


\subsection{A related viewpoint}\label{ssec:geometric} Let $(C,0)$ be an isolated plane curve singularity and $f:(\C^2_{x,y},0)\lr(\C,0)$ a polynomial representing it. Another viewpoint on the Main Theorem is as follows. Tautologically, the polynomial $f(x,y)$ contains the data of the topological type of $(C,0)$, by restricting to a Milnor ball near the origin $(0,0)\in\C^2$. A Morsification $\tilde{f}:\C^2\lr\C$, seen as a Lefschetz fibration, encodes the data of the monodromy of $(C,0)$. Appropriately understood, the proof of the Main Theorem explains how to recover the monodromy of $(C,0)$ by studying the complex 3-fold
$$X_{\tilde{f}}:=\{(x,y,z,w)\in\C^4: \tilde{f}(x,y)+zw=0\}.$$
The challenge is that, a priori, we are only provided the exact symplectomorphism type of $X_{\tilde{f}}$, or at most a certain Lagrangian skeleton, but not $\tilde{f}$ or any particular equation for it. A new idea used in the Main Theorem is that appropriately using Floer theory on $X_{\tilde{f}}$ one should be able to intrinsically recover the monodromy. Note that the Serre functor of the Fukaya-Seidel category of the Lefschetz fibration $\tilde{f}:\C^2\lr\C$ recovers the algebraic monodromy by acting on the Grothendieck group. Nevertheless, the double-suspension $X_{\tilde{f}}$ of the Milnor fiber is 3-Calabi-Yau and thus the Serre functor of its wrapped Fukaya category does not provide particularly enlightening information about the monodromy of $\tilde{f}:\C^2\lr\C$. Intuitively, via a form of homological mirror symmetry, we can relate such wrapped Fukaya category to the Ginzburg dg algebra $\Gamma(Q(D),W(D))$, the latter seen as a non-commutative crepant resolution. In the $B$-side, the question is then about recovering information about the monodromy of $\tilde{f}:\C^2\lr\C$ from a certain category of modules over $\Gamma(Q(D),W(D))$. The content of Sections \ref{sec:graded-QP}, \ref{sec:euler}, \ref{sec:dilation_plabic_fence} and \ref{sec:polynomial_plabic_fence} provide a solution to this reconstruction problem.

\bibliographystyle{plain}
\bibliography{main}

@article {Seade19_MilnorFibration,
    AUTHOR = {Seade, Jos\'e},
     TITLE = {On {M}ilnor's fibration theorem and its offspring after 50
              years},
   JOURNAL = {Bull. Amer. Math. Soc. (N.S.)},
  FJOURNAL = {American Mathematical Society. Bulletin. New Series},
    VOLUME = {56},
      YEAR = {2019},
    NUMBER = {2},
     PAGES = {281--348},
      ISSN = {0273-0979,1088-9485},
   MRCLASS = {32Sxx (14B05 55S35 57M27 57M50 57R77)},
  MRNUMBER = {3923346},
MRREVIEWER = {Maria\ Aparecida Soares Ruas},
       DOI = {10.1090/bull/1654},
       URL = {https://doi.org/10.1090/bull/1654},
}

@book{lickorish1997introduction,
  title={An Introduction to Knot Theory},
  author={Lickorish, W. B. Raymond},
  series={Graduate Texts in Mathematics},
  volume={175},
  year={1997},
  publisher={Springer-Verlag},
  address={New York}
}

@incollection {Casals_LagrangianPlaneCurve,
    AUTHOR = {Casals, Roger},
     TITLE = {Lagrangian skeleta and plane curve singularities},
 BOOKTITLE = {Symplectic geometry---a {F}estschrift in honour of {C}laude
              {V}iterbo's 60th birthday},
     PAGES = {181--223},
      NOTE = {Reprint of [4405603]},
 PUBLISHER = {Birkh\"auser/Springer, Cham},
      YEAR = {[2022] \copyright 2022},
      ISBN = {978-3-031-19110-7},
   MRCLASS = {53D12 (14H20 57K33)},
  MRNUMBER = {4696540},
       DOI = {10.1007/978-3-031-19111-4\_9},
       URL = {https://doi.org/10.1007/978-3-031-19111-4_9},
}

@article {Rutherford_HOMFLY,
    AUTHOR = {Rutherford, Dan},
     TITLE = {Thurston-{B}ennequin number, {K}auffman polynomial, and ruling
              invariants of a {L}egendrian link: the {F}uchs conjecture and
              beyond},
   JOURNAL = {Int. Math. Res. Not.},
  FJOURNAL = {International Mathematics Research Notices},
      YEAR = {2006},
     PAGES = {Art. ID 78591, 15},
      ISSN = {1073-7928,1687-0247},
   MRCLASS = {57M27 (57R17)},
  MRNUMBER = {2219227},
MRREVIEWER = {Justin\ Sawon},
       DOI = {10.1155/IMRN/2006/78591},
       URL = {https://doi.org/10.1155/IMRN/2006/78591},
}

@article{dubois2003seifert,
  title={Sur la forme de {S}eifert enti{\`e}re des germes de courbe plane {\`a} deux branches},
  author={Du Bois, Philippe},
  journal={Comptes Rendus Math{\'e}matique},
  volume={336},
  number={9},
  pages={757--762},
  year={2003},
  publisher={Elsevier},
  doi={10.1016/S1631-073X(03)00165-1}
}

@article {casals_positroid,
    AUTHOR = {Casals, Roger and Gorsky, Eugene and Gorsky, Mikhail and
              Simental, Jos\'e},
     TITLE = {Positroid links and braid varieties},
   JOURNAL = {J. Reine Angew. Math.},
  FJOURNAL = {Journal f\"ur die Reine und Angewandte Mathematik. [Crelle's
              Journal]},
      YEAR = {2026},
    NUMBER = {837},
     PAGES = {1--54},
      ISSN = {0075-4102,1435-5345},
   MRCLASS = {99-06},
  MRNUMBER = {5110495},
       DOI = {10.1515/crelle-2026-0007},
       URL = {https://doi.org/10.1515/crelle-2026-0007},
}

@misc{fomin2024introductionclusteralgebraschapters,
      title={Introduction to Cluster Algebras. Chapters 1-3}, 
      author={Sergey Fomin and Lauren Williams and Andrei Zelevinsky},
      year={2024},
      eprint={1608.05735},
      archivePrefix={arXiv},
      primaryClass={math.CO},
      url={https://arxiv.org/abs/1608.05735}, 
}

@article{ACampo1975,
  author  = {A'Campo, N.},
  title   = {La fonction z\^eta d'une monodromie},
  journal = {Comment. Math. Helv.},
  volume  = {50},
  year    = {1975},
  pages   = {233--248},
  note    = {The resolution formula for the monodromy zeta function is developed in this paper.}
}

@article{Arnold1976,
  author  = {Arnold, V. I.},
  title   = {Local normal forms of functions},
  journal = {Invent. Math.},
  volume  = {35},
  year    = {1976},
  pages   = {87--109},
  note    = {See \S1, especially pp.~92--97, for the normal-form tables and Milnor-number formulas; see \S2, pp.~101--103, especially statements 2--49 and statements 47--49, for the exhaustive two-variable determinator and the $N_{16}$ endpoint.}
}

@article{HertlingMase2022,
  author  = {Hertling, C. and Mase, M.},
  title   = {The combinatorics of weight systems and characteristic polynomials of isolated quasihomogeneous singularities},
  journal = {J. Algebraic Combin.},
  volume  = {56},
  year    = {2022},
  pages   = {929--954},
  note    = {See Theorem~6.4(a) for a precise modern statement of the Milnor--Orlik characteristic-polynomial formula.}
}

@article{MichelWeber1986,
  author  = {Michel, F. and Weber, C.},
  title   = {Sur le r\^ole de la monodromie enti\`ere dans la topologie des singularit\'es},
  journal = {Ann. Inst. Fourier (Grenoble)},
  volume  = {36},
  number  = {1},
  year    = {1986},
  pages   = {183--218},
  note    = {See \S1.3(3), p.~185; Proposition~3.14, p.~196; and \S3.16, p.~197 for the $D_{12}/Z_{12}$ integral-monodromy collision.}
}

@article{MilnorOrlik1970,
  author  = {Milnor, J. and Orlik, P.},
  title   = {Isolated singularities defined by weighted homogeneous polynomials},
  journal = {Topology},
  volume  = {9},
  year    = {1970},
  pages   = {385--393},
  note    = {The paper gives explicit formulas for the Milnor number and characteristic polynomial of the monodromy from the weights.}
}

@article{Varchenko1976,
  author  = {Varchenko, A. N.},
  title   = {Zeta-function of monodromy and Newton's diagram},
  journal = {Invent. Math.},
  volume  = {37},
  year    = {1976},
  pages   = {253--262},
  note    = {The main theorem computes the monodromy zeta function of a Newton-nondegenerate germ from its Newton diagram; its two-variable specialization is used in \eqref{eq:Newton2D}.}
}

@article{Durfee1975,
  author  = {Durfee, Alan},
  title   = {The characteristic polynomial of the monodromy},
  journal = {Pacific Journal of Mathematics},
  year    = {1975},
  volume  = {59},
  pages   = {21--26},
  doi     = {10.2140/pjm.1975.59.21}
}

@article{Leviant2017,
  author  = {Leviant, Peter and Shustin, Eugenii},
  title   = {Morsifications of real plane curve singularities},
  journal = {J. Singularities 18 (2018), 307--328},
  year    = {2017},
  doi     = {10.48550/arxiv.1703.05510}
}

@article{yoshinaga1979topological,
  author  = {Yoshinaga, Etsuo and Suzuki, Masahiko},
  title   = {Topological types of quasihomogeneous singularities in {C}$^2$},
  journal = {Topology},
  volume  = {18},
  number  = {2},
  year    = {1979},
  pages   = {113--116},
  doi     = {10.1016/0040-9383(79)90013-3}
}

@incollection{grima1976monodromie,
  author    = {Grima, Marie-Claire},
  title     = {La monodromie rationnelle ne d{\'e}termine pas la topologie d'une hypersurface complexe},
  booktitle = {Fonctions de plusieurs variables complexes {II}},
  series    = {Lecture Notes in Mathematics},
  volume    = {538},
  editor    = {Norguet, Fran{\c{c}}ois},
  publisher = {Springer},
  address   = {Berlin, Heidelberg},
  year      = {1976},
  pages     = {58--79}
}

@article {CDGZ2,
    AUTHOR = {Campillo, A. and Delgado, F. and Gusein-Zade, S. M.},
     TITLE = {On the topological type of a set of plane valuations with
              symmetries},
   JOURNAL = {Math. Nachr.},
  FJOURNAL = {Mathematische Nachrichten},
    VOLUME = {290},
      YEAR = {2017},
    NUMBER = {13},
     PAGES = {1925--1938},
      ISSN = {0025-584X,1522-2616},
   MRCLASS = {13A18 (13D40 14B05)},
  MRNUMBER = {3695804},
MRREVIEWER = {Guillaume\ Rond},
       DOI = {10.1002/mana.201600198},
       URL = {https://doi.org/10.1002/mana.201600198},
}

@book{EisenbudNeumann1985,
  title     = {Three-Dimensional Link Theory and Invariants of Plane Curve Singularities},
  author    = {Eisenbud, David and Neumann, Walter D.},
  series    = {Annals of Mathematics Studies},
  volume    = {110},
  year      = {1985},
  publisher = {Princeton University Press},
  address   = {Princeton, NJ},
  isbn      = {978-0691083810}
}

@article{HerschendIyama2010,
  author  = {Herschend, Martin and Iyama, Osamu},
  title   = {Selfinjective quivers with potential and 2-representation-finite algebras},
  journal = {Compositio Math. 147 (2011) 1885-1920},
  year    = {2010},
  doi     = {10.1112/S0010437X11005367}
}

@misc{Keller2026_GradedPreprint,
  author       = {Keller, B.},
  title        = {Non degenerate homogeneous potentials},
  howpublished = {preprint},
  year         = {2026},
}

@misc{Wu2026Graded,
  author       = {Wu, Y.},
  title        = {\href{https://cloud.math.univ-paris-diderot.fr/s/2P364q8sRGLfqWT?dir=/Today\%27s\%20slides\%20or\%20notes&editing=false&openfile=true}{Graded Higgs categories}},
  howpublished = {Notes from a talk at the Paris algebra seminar},
  year         = {2026},
  month        = {June}
}

@misc{Keller2020Cartan,
  author       = {Keller, B.},
  title        = {\href{https://cloud.math.univ-paris-diderot.fr/s/Y6RnZcFTnQxzoP7?dir=/&editing=false&openfile=true}{Cartan matrices and Calabi-Yau completions}},
  howpublished = {Notes from a talk at the Sherbrooke Meeting on Representation Theory of Algebras},
  year         = {2020},
  month        = {September}
}

@article{FujitaMurakami2022,
  author  = {Fujita, Ryo and Murakami, Kota},
  title   = {Deformed Cartan Matrices and Generalized Preprojective Algebras I: Finite Type},
  journal = {International Mathematics Research Notices},
  year    = {2022},
  volume  = {2023},
  pages   = {6924--6975},
  doi     = {10.1093/imrn/rnac054}
}

@article{FanKellerQiu2024,
  author  = {Fan, Li and Keller, Bernhard and Qiu, Yu},
  title   = {Dg enhanced orbit categories and applications},
  journal = {arXiv},
  year    = {2024},
  doi     = {10.48550/arxiv.2405.00093}
}

@article{ACampo1998,
  author  = {N. A'Campo},
  title   = {Generic immersions of curves, knots, monodromy and Gordian number},
  journal = {Publ. Math. Inst. Hautes \`Etudes Sci.},
  volume  = {88},
  year    = {1998},
  pages   = {151--169}
}

@article{ACampo1999,
  author  = {N. A'Campo},
  title   = {Real deformations and complex topology of plane curve singularities},
  journal = {Ann. Fac. Sci. Toulouse Math. (6)},
  volume  = {8},
  number  = {1},
  year    = {1999},
  pages   = {5--23},
  note    = {erratum, ibid. no.~2, 343}
}

@article{AO,
  author  = {C. Amiot and S. Oppermann},
  title   = {Cluster equivalence and graded derived equivalence},
  journal = {Doc. Math.},
  volume  = {19},
  year    = {2014},
  pages   = {1155--1206}
}

@article{CDGZ,
  author  = {A. Campillo and F. Delgado and S. M. Gusein-Zade},
  title   = {The Alexander polynomial of a plane curve singularity via the ring of functions on it},
  journal = {Duke Math. J.},
  volume  = {117},
  number  = {1},
  year    = {2003},
  pages   = {125--156}
}

@article{CasalsGao,
  author  = {R. Casals and H. Gao},
  title   = {A Lagrangian filling for every cluster seed},
  journal = {Invent. Math.},
  volume  = {237},
  year    = {2024},
  pages   = {809--868}
}

@article{DWZ,
  author  = {H. Derksen and J. Weyman and A. Zelevinsky},
  title   = {Quivers with potentials and their representations I: Mutations},
  journal = {Selecta Math. (N.S.)},
  volume  = {14},
  year    = {2008},
  pages   = {59--119}
}

@article{FominNeville,
  author  = {S. Fomin and S. Neville},
  title   = {Cyclically ordered quivers},
  journal = {Adv. Math.},
  volume  = {500},
  year    = {2026},
  pages   = {Article 111059},
  note    = {arXiv:2406.03604v3}
}

@article{FPST,
  author  = {S. Fomin and P. Pylyavskyy and E. Shustin and D. Thurston},
  title   = {Morsifications and mutations},
  journal = {J. Lond. Math. Soc. (2)},
  volume  = {105},
  number  = {4},
  year    = {2022},
  pages   = {2478--2554}
}

@misc{Ginzburg,
  author       = {V. Ginzburg},
  title        = {Calabi--Yau algebras},
  howpublished = {arXiv:math/0612139}
}

@article{KellerYang,
  author  = {B. Keller and D. Yang},
  title   = {Derived equivalences from mutations of quivers with potential},
  journal = {Adv. Math.},
  volume  = {226},
  number  = {3},
  year    = {2011},
  pages   = {2118--2168}
}

@misc{Ladkani,
  author       = {S. Ladkani},
  title        = {On cluster algebras from once punctured closed surfaces},
  howpublished = {arXiv:1310.4454}
}

@article{ACampo,
  author  = {N. A'Campo},
  title   = {Planar trees, slalom curves and hyperbolic knots},
  journal = {Publ. Math. Inst. Hautes \`Etudes Sci.},
  volume  = {88},
  year    = {1998},
  pages   = {171--180},
  url     = {https://arxiv.org/abs/math/9906087}
}

@article{BorodzikHom,
  author  = {M. Borodzik and J. Hom},
  title   = {Involutive Heegaard Floer homology and rational cuspidal curves},
  journal = {Proc. Lond. Math. Soc.},
  volume  = {118},
  year    = {2019},
  pages   = {1337--1381}
}

@misc{Neville,
  author       = {S. Neville},
  title        = {Mutation-acyclic quivers are totally proper},
  howpublished = {\url{https://arxiv.org/abs/2409.17832}}
}

@article{CasalsWeng22,
  author  = {R. Casals and D. Weng},
  title   = {Microlocal theory of Legendrian links and cluster algebras},
  journal = {Geometry \& Topology},
  volume  = {28},
  year    = {2024},
  pages   = {901--1000}
}

@misc{kontsevich_soibelman_2008,
  author       = {M. Kontsevich and Y. Soibelman},
  title        = {Stability structures, motivic Donaldson-Thomas invariants and cluster transformations},
  howpublished = {arXiv preprint arXiv:0811.2435},
  year         = {2008}
}

@article{BaaderDehornoy,
  author  = {S. Baader and P. Dehornoy},
  title   = {Trefoil plumbing},
  journal = {Proc. Amer. Math. Soc.},
  volume  = {144},
  number  = {1},
  year    = {2016},
  pages   = {387--397},
}

@article{Ferretti,
  author  = {L. Ferretti},
  title   = {On positive braids, monodromy groups and framings},
  journal = {Algebr. Geom. Topol.},
  volume  = {25},
  number  = {1},
  year    = {2025},
  pages   = {161--205},
}

@book{Rolfsen,
  author    = {D. Rolfsen},
  title     = {Knots and Links},
  series    = {Mathematics Lecture Series},
  volume    = {7},
  publisher = {Publish or Perish},
  address   = {Berkeley, CA},
  year      = {1976}
}

@article{Yamamoto1984,
  author  = {M. Yamamoto},
  title   = {Classification of isolated algebraic singularities by their Alexander polynomials},
  journal = {Topology},
  volume  = {23},
  number  = {3},
  year    = {1984},
  pages   = {277--287},
  doi     = {10.1016/0040-9383(84)90018-0},
  url     = {https://doi.org/10.1016/0040-9383(84)90018-0}
}

@article{Burau1932,
  author  = {W. Burau},
  title   = {Kennzeichnung der Schlauchknoten},
  journal = {Abhandlungen aus dem Mathematischen Seminar der Universit\"{a}t Hamburg},
  volume  = {9},
  year    = {1932},
  pages   = {125--133}
}

@article{Le1972,
  author  = {D. T. L\^{e}},
  title   = {Sur les noeuds alg\'{e}briques},
  journal = {Compositio Mathematica},
  volume  = {25},
  year    = {1972},
  pages   = {281--321}
}

@article{Fomin2001,
  author  = {S. Fomin and A. Zelevinsky},
  title   = {Cluster algebras I: Foundations},
  journal = {Journal of the American Mathematical Society},
  volume  = {15},
  year    = {2001},
  pages   = {497--529},
  doi     = {10.1090/s0894-0347-01-00385-x}
}

@article{Fomin2003,
  author  = {S. Fomin and A. Zelevinsky},
  title   = {Cluster algebras II: Finite type classification},
  journal = {Inventiones mathematicae},
  volume  = {154},
  year    = {2003},
  pages   = {63--121},
  doi     = {10.1007/s00222-003-0302-y}
}

@article{Berenstein2005,
  author  = {A. Berenstein and S. Fomin and A. Zelevinsky},
  title   = {Cluster algebras III: Upper bounds and double Bruhat cells},
  journal = {Duke Mathematical Journal},
  volume  = {126},
  year    = {2005},
  doi     = {10.1215/s0012-7094-04-12611-9}
}

@misc{BKM13,
  author       = {K. Baur and A. King and B. Marsh},
  title        = {Dimer models and cluster categories of Grassmannians},
  howpublished = {arXiv},
  year         = {2013},
  url          = {https://doi.org/10.1112/plms/pdw029}
}

@article{Kalman2006,
  author    = {K{\'a}lm{\'a}n, T.},
  title     = {Braid-positive {L}egendrian links},
  journal   = {International Mathematics Research Notices},
  volume    = {2006},
  year      = {2006},
  pages     = {Art. ID 14874, 24},
  doi       = {10.1155/IMRN/2006/14874},
  eprint    = {math/0608457},
  archivePrefix = {arXiv},
  primaryClass  = {math.SG}
}

@article{Pressland19,
  author  = {M. Pressland},
  title   = {Calabi-Yau properties of Postnikov diagrams},
  journal = {Forum of Mathematics, Sigma},
  volume  = {10},
  year    = {2022},
  pages   = {e56},
  url     = {https://doi.org/10.1017/fms.2022.52}
}

@incollection{Conway1970,
  author    = {J. H. Conway},
  title     = {An enumeration of knots and links, and some of their algebraic properties},
  booktitle = {Computational Problems in Abstract Algebra (Proc. Conf., Oxford, 1967)},
  editor    = {J. Leech},
  publisher = {Pergamon Press},
  address   = {Oxford},
  year      = {1970},
  pages     = {329--358},
  doi       = {10.1016/B978-0-08-012975-4.50034-5}
}

@incollection{Morton2015,
  author    = {H. R. Morton},
  title     = {Mutant knots},
  booktitle = {New Ideas in Low Dimensional Topology},
  editor    = {L. H. Kauffman and V. O. Manturov},
  series    = {Series on Knots and Everything},
  volume    = {56},
  publisher = {World Scientific},
  address   = {Hackensack, NJ},
  year      = {2015},
  pages     = {379--412},
  doi       = {10.1142/9789814630627_0010}
}

@article{FreydEtAl1985,
  author  = {P. Freyd and D. Yetter and J. Hoste and W. B. R. Lickorish and K. C. Millett and A. Ocneanu},
  title   = {A new polynomial invariant of knots and links},
  journal = {Bull. Amer. Math. Soc. (N.S.)},
  volume  = {12},
  number  = {2},
  year    = {1985},
  pages   = {239--246},
  doi     = {10.1090/S0273-0979-1985-15361-3}
}

@article{Hoste1986,
  author  = {J. Hoste},
  title   = {A polynomial invariant of knots and links},
  journal = {Pacific J. Math.},
  volume  = {124},
  number  = {2},
  year    = {1986},
  pages   = {295--320}
}

@article{LickorishMillett1987,
  author  = {W. B. R. Lickorish and K. C. Millett},
  title   = {A polynomial invariant of oriented links},
  journal = {Topology},
  volume  = {26},
  number  = {1},
  year    = {1987},
  pages   = {107--141},
  doi     = {10.1016/0040-9383(87)90025-5}
}

@article{GL,
  author  = {P. Galashin and T. Lam},
  title   = {Plabic links, quivers, and skein relations},
  journal = {Algebraic Combinatorics},
  volume  = {7},
  number  = {2},
  year    = {2024},
  pages   = {431--474},
  url     = {https://doi.org/10.5802/alco.345}
}

@article{Schwartz,
  author  = {A. Schwartz},
  title   = {The HOMFLY polynomial of a forest quiver},
  journal = {International Mathematics Research Notices},
  volume  = {2026},
  number  = {8},
  year    = {2026},
  pages   = {rnag069},
  url     = {https://doi.org/10.1093/imrn/rnag069}
}

@book{wall2004,
  author    = {C. T. C. Wall},
  title     = {Singular Points of Plane Curves},
  series    = {London Mathematical Society Student Texts},
  volume    = {63},
  publisher = {Cambridge University Press},
  address   = {Cambridge},
  year      = {2004}
}

@book{zariski2006,
  author    = {O. Zariski},
  title     = {The Moduli Problem for Plane Branches},
  series    = {University Lecture Series},
  volume    = {39},
  publisher = {American Mathematical Society},
  address   = {Providence, RI},
  year      = {2006}
}

@book{ghys2017,
  author    = {\'{E}. Ghys},
  title     = {A Singular Mathematical Promenade},
  publisher = {ENS \'{E}ditions},
  address   = {Lyon},
  year      = {2017}
}

@book{casas2000,
  author    = {E. Casas-Alvero},
  title     = {Singularities of Plane Curves},
  series    = {London Mathematical Society Lecture Note Series},
  volume    = {276},
  publisher = {Cambridge University Press},
  address   = {Cambridge},
  year      = {2000}
}

@book{brieskorn1986,
  author    = {E. Brieskorn and H. Kn\"{o}rrer},
  title     = {Plane Algebraic Curves},
  publisher = {Birkh\"{a}user Verlag},
  address   = {Basel},
  year      = {1986}
}

@book{milnor1968,
  author    = {J. Milnor},
  title     = {Singular Points of Complex Hypersurfaces},
  series    = {Annals of Mathematics Studies},
  volume    = {61},
  publisher = {Princeton University Press},
  address   = {Princeton, NJ},
  year      = {1968}
}
\end{document}